\documentclass[11pt,reqno,a4paper]{amsart}
\usepackage[margin=1in]{geometry}
\usepackage{amsmath, amsthm, amssymb}
\usepackage{enumitem}
\usepackage[colorlinks=true, pdfstartview=FitV, linkcolor=blue,citecolor=blue, urlcolor=blue]{hyperref}
\usepackage{times}
\usepackage[dvipsnames]{xcolor}
\usepackage{subcaption}
\usepackage{mathtools}
\usepackage{bbm}
\usepackage{upgreek}
\usepackage{comment}
\usepackage{dsfont}
\usepackage{tikz}
\usepackage{subcaption}
\usepackage{booktabs}
\usepackage[export]{adjustbox}
\usepackage[numbers,sort&compress]{natbib}
\usepackage{placeins}
 \usepackage[capitalise]{cleveref}
\usepackage{autonum}
\crefname{enumi}{}{}
\Crefname{enumi}{}{}

\crefformat{equation}{(#2#1#3)}
\crefrangeformat{equation}{(#3#1#4) to~(#5#2#6)}
\crefmultiformat{equation}{(#2#1#3)}%
{ and~(#2#1#3)}{, (#2#1#3)}{ and~(#2#1#3)}

\def\dd{{\rm d}}

\def\eps{\varepsilon}
\def\epsilon{\varepsilon}
\def\e{{\rm e}} 
\def\de{{\partial}}

\def\RR {\mathbb{R}}
\def\R {\mathbb{R}} 

\def\NN {\mathbb{N}}

\def\Re{{\rm Re}}
\def\Im{{\rm Im}}

\newcommand{\cB}{\mathcal{B}}

\newcommand{\cD}{\mathcal{D}}
\newcommand{\cE}{\mathcal{E}}

\newcommand{\cH}{\mathcal{H}}

\newcommand{\cJ}{\mathcal{J}}
\newcommand{\cK}{\mathcal{K}}
\newcommand{\cL}{\mathcal{L}}
\newcommand{\cM}{\mathcal{M}}
\newcommand{\cN}{\mathcal{N}}
\newcommand{\cO}{\mathcal{O}}
\newcommand{\cP}{\mathcal{P}}
\newcommand{\cQ}{\mathcal{Q}}
\newcommand{\cR}{\mathcal{R}}
\newcommand{\cS}{\mathcal{S}}
\newcommand{\cT}{\mathcal{T}}

\newcommand{\cV}{\mathcal{V}}

\newcommand{\cX}{\mathcal{X}}
\newcommand{\cY}{\mathcal{Y}}
\newcommand{\cZ}{\mathcal{Z}}

\newcommand{\adv}{\mathsf{adv}}
\newcommand{\diff}{\mathsf{diff}}

\newcommand{\app}{\mathrm{app}}

\newcommand{\Ein}{\mathrm{Ein}}

\newcommand{\mrm}{\mathrm{m}}

\newcommand{\mrM}{\mathrm{M}}

\newtheorem{proposition}{Proposition}[section]
\newtheorem{theorem}[proposition]{Theorem}
\newtheorem{corollary}[proposition]{Corollary}
\newtheorem{lemma}[proposition]{Lemma}
\theoremstyle{definition}
\newtheorem{definition}[proposition]{Definition}
\newtheorem{remark}[proposition]{Remark}

\theoremstyle{definition}

\numberwithin{equation}{section}

\newcommand{\Mapp}[2]{#1^{(#2)}_\mathrm{app}}
\newcommand{\Mappp}[3]{#1^{(#2)}_\mathrm{app,#3}}

\title{Viscous vortex crystals}
\author{Michele Dolce}
\address{Institute of Mathematics, EPFL, Station 8, 1015 Lausanne, Switzerland}
\email{michele.dolce@epfl.ch}

\author{Martin Donati}
\address{CNRS, Université de Poitiers, LMA, Poitiers, France}
\email{martin.donati@math.univ-poitiers.fr}

\begin{document}
\begin{abstract}
    We study the solution to the two-dimensional incompressible Navier-Stokes equations arising from a sum of Dirac masses in a particular co-rotating configuration. This configuration consists of a polygonal vortex crystal with or without a central vortex. By exploiting the  symmetries and stability properties of the system, we describe and control the solution up to sub-diffusive time scales, prior to the expected onset of vortex merging.
\end{abstract}

\maketitle

\tableofcontents

\section{Introduction}
To understand large-scale atmospheric dynamics, it is essential to characterize the motion and interaction of  coherent vortices. These vortices can be relatively short-lived and extreme, such as cyclones, or persistent and somewhat stable, such as Earth's tropospheric polar vortices. In this paper, we examine a specific yet not uncommon scenario where a group of vortices organizes into a highly stable structure, achieving an abnormally long lifetime due to this particular configuration. A striking example of this phenomenon is currently observed at Jupiter's poles, where polygonally arranged vortices rotate collectively in an almost rigid manner around each pole, see \cite{Adriani2018Nature}.

More broadly, relative equilibria of vortices manifest in diverse contexts, including metastable states in electron plasmas \cite{schecter1999vortex}, wake structures behind wind turbines, and as outcomes of two-dimensional turbulence relaxation \cite{Fin95}. In the latter case, the stability of these equilibria is pivotal, as it determines which structures persist even with turbulent chaos. For a comprehensive mathematical analysis of this phenomenon, we refer to \cite{Drivas_Elgindi_2023_Singularity}, while \cite{SieYouIng22} provides numerical simulations illustrating the formation of such vortex structures from turbulence, particularly near the poles of rotating planets. Moreover, the emergence of large vortex structures following a turbulent phase was explored in \cite{Modin22}, where a finite-dimensional approximation of ideal fluids on a sphere, known as the \emph{Zeitlin model}, is used. The authors consistently observed at most four large vortices in the long-time regime (see also \cite{Modin26} for the mathematical foundations of the matrix hydrodynamic model used in \cite{Modin22}).

Despite the inherent complexity of real atmospheric systems, where multiple physical processes interact, many fundamental stability mechanisms and emergent behaviors can be effectively captured by even the simplest fluid models. For instance, the \emph{point-vortex} model pioneered by Helmholtz \cite{Helmholtz_1858} and Kirchhoff \cite{kirchhoff1876vorlesungen}, in which  vortices are represented as singularities concentrated at discrete points, provides an accurate first approximation of the motion of strong, concentrated vortices. Within this model, polygonally arranged point vortices indeed rotate in a rigid manner. The study of their stability properties dates back to Thomson's 1883 foundational result \cite{thomson1883treatise}.   Since then, the topic has been extensively studied, with \cite{CabSch1999,aref2002vortex} providing modern results and historical accounts.

Another natural research direction involves the \emph{desingularization} of the point-vortex model to account for finite-size effects. In this context, particular exact concentrated solutions can be constructed through variational methods (e.g. \cite{Turkington_1985,Cao25}), the study of contour dynamics \cite{HmMa17,Hassainia25,hassainia2024desingularization}, or the \emph{gluing method} introduced in \cite{Davila20}, while the control of any non-viscous concentrated solution to the planar incompressible Euler equation can be obtained by the \emph{vortex blobs method}, e.g. \cite{MarPul84, Butta18}.
 Furthermore, the second author \cite{Donati_2025_Crystal} proved that non-singular, concentrated vortices also exhibit robust stability properties when arranged in polygonal configurations. In the present work, we consider these polygonal configurations in weakly viscous fluid flows.

\subsection{Mathematical setting}
We consider a two-dimensional incompressible and viscous fluid flow,  described by a velocity field $u$ with vorticity $\omega=\nabla^\perp\cdot u$ and viscosity $\nu > 0$. This evolves according to the Navier-Stokes equations in $\R^2$, which in vorticity form are given by:
\begin{equation}
\label{eq:NS}
    \begin{cases}
        \displaystyle \de_t \omega(x,t) +u(x,t)\cdot \nabla \omega(x,t) =\nu \Delta \omega(x,t),  & x\in \RR^2, \quad t>0 \vspace{2mm}\\
        \displaystyle u(x,t) = \frac{1}{2\pi}\int_{\RR^2}\frac{(x-y)^\perp}{|x-y|^2}\omega(y,t) \dd y, & x\in \RR^2, \quad t \ge 0\vspace{2mm}\\
        \displaystyle  \omega|_{t=0}=\omega^{in},
    \end{cases}
\end{equation}
where the initial datum $\omega^{in}$ has to be specified. We aim to describe the evolution of vortices that are very concentrated in space around their center of mass. It is well known that, at least over some time interval, if the viscosity of the fluid is small enough, the motion of those vortices (the motion of their center of mass) is well approximated by the so-called \emph{point-vortex dynamics}, or \emph{Helmholtz-Kirchhoff} equations
\begin{align}
\label{eq:HK}
    & z_j'(t)=\sum_{\substack{l=1 \\ l\neq j}}^{N+1}\frac{\Gamma_l}{2\pi}\frac{(z_j(t)-z_l(t))^\perp}{|z_j(t)-z_l(t)|^2}, \qquad j,l=1,\dots, N+1, \\
    & z_j(0) = z_j^{in}.
\end{align}
where $z_1(t),\dots,z_{N+1}(t)\in \mathbb{R}^2$ are the positions of the $N+1$ point-vortices  and $\Gamma_1,\dots,\Gamma_{N+1}\in \RR$ their circulations. Going back to Helmholtz \cite{Helmholtz_1858}, a rigorous proof of this statement was made for a non-viscous fluid in \cite{MarPul84} (see also \cite{MarPul94}). For weakly-viscous fluids, it was done in \cite{Gallay2011} in the case of a sum of Dirac masses of vorticity (point-vortices) as initial data, where control of the solution was established on a non-viscous time scale. In the case of concentrated (but not infinitely concentrated) initial vortices, it was done by
 \cite{Mar98,CeSe18} with a viscosity depending on the considered initial data, then by \cite{CeSeis24} with fixed small viscosity. In \cite{dolce2024long}, the same study as \cite{Gallay2011} was conducted, but on a particular configuration of vortices: the translating pair of opposite vortices. In that case, using the specific symmetries and stability of the configuration, the control of the solution is obtained on a much larger time scale, nearly reaching what we call the \emph{diffusive time}, which is the time after which the diffusive effects become so strong that the vortices cannot be considered concentrated anymore. In the recent contribution \cite{zhang2025long}, the authors treated the case of two vortices of any given intensity, obtaining the result on the same large time scales reached in \cite{dolce2024long}. For details regarding the ODE system~\eqref{eq:HK} itself, we refer to \cite{Aref_2007}.

In this paper, we perform the same analysis but on a configuration consisting of equal vortices placed at the vertices of a regular polygon, with or without a central vortex of given circulation. 
We thus consider equations~\cref{eq:NS} with initial datum
\begin{equation}\label{eq:omegain}
    \omega^{in}=\Gamma\sum_{j=1}^N\delta_{z_j^{in}}+\gamma \Gamma \delta_{0},
\end{equation}
where $N\geq 2$, $\Gamma \in \R\setminus\{0\}$, $\gamma \in \R$, and the $z_j^{in}$ are
\begin{equation}\label{eq:zin}
    z_j^{in} = \mathsf{r}\Big(\cos\big(\frac{2\pi j}{N}\big) ,\sin\big(\frac{2\pi j}{N}\big) \Big), \qquad \text{for}\quad  j=1,\dots, N,
\end{equation}
with $\mathsf{r} >0$. This initial datum is a sum of equal Dirac masses regularly placed on the disk of center 0 and radius $\mathsf{r}$, with (or without if $\gamma=0$) a central vortex, see Figure~\ref{fig:crystal1}. Note that the cases $N=1,\gamma\neq0$ and $N=2$, $\gamma=0$ are fully covered by \cite{dolce2024long,zhang2025long}, which is why we restrict our attention to $N\geq2$. Let us recall that existence and uniqueness of a solution to the system~\cref{eq:NS} with a measure as initial data were proven in \cite{GaGa2005}. Moreover, the solution of the point-vortex dynamics~\cref{eq:HK} with initial datum~\cref{eq:zin} is well known (see for instance \cite{aref2002vortex}) to be given by
\begin{align}
\label{eq:zjt}
\begin{cases}z_j(t)=\mathsf{r}\big(\cos\big(\frac{2\pi j}{N}+\mathsf{a} t\big) ,\sin\big(\frac{2\pi j}{N}+\mathsf{a} t\big) \big), \qquad &\text{for}\quad  j=1,\dots, N\\
    z_{N+1}(t)=(0,0),\\
    \mathsf{a}=\frac{\Gamma}{4\pi \mathsf{r}^2}(N-1+2\gamma).
    \end{cases}
\end{align}

\begin{figure}
    
    \centering
    \begin{tikzpicture} 

        \draw (-5+0,0) node {$\bullet$};
        \draw (-5+1.2,0) node {$\bullet$};
        \draw (-5+0.5*1.2,0.866*1.2) node {$\bullet$};
        \draw (-5-0.5*1.2,0.866*1.2) node {$\bullet$};
        \draw (-5-1.2,0) node {$\bullet$};
        \draw (-5-0.5*1.2,-0.866*1.2) node {$\bullet$};
        \draw (-5+0.5*1.2,-0.866*1.2) node {$\bullet$};

        \draw (1.2,0) node {$\bullet$};
        \draw (0.309*1.2,0.951*1.2) node {$\bullet$};
        \draw (-0.809*1.2,0.587*1.2) node {$\bullet$};
        \draw (-0.809*1.2,-0.587*1.2) node {$\bullet$};
        \draw (0.309*1.2,-0.951*1.2) node {$\bullet$};

    \end{tikzpicture}
    \caption{Two examples of polygonal vortex configuration as described at \cref{eq:zin}, with $N=6$ and $\gamma = 1$ (left), and $N=5$ and $\gamma =0$ (right).}
    \label{fig:crystal1}
\end{figure}
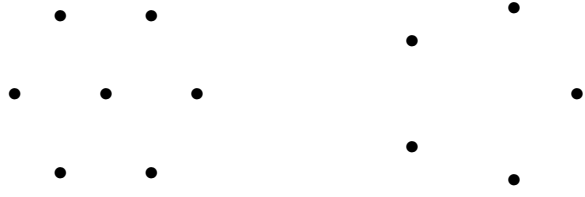
This configuration is uniformly rotating around the origin. It is a relative equilibrium of the dynamics, called in this context a \emph{vortex crystal}. The stability of this configuration was studied in \cite{CabSch1999}\footnote{Let us mention that \cite{LP05} disagrees with the main result of~\cite{CabSch1999} concerning the case $N \le 6$, $\gamma < 0$.}, and the result is recalled in Section~\ref{sec:stability}. In the non-viscous case, for a special value of $\gamma$ and for $N \ge 4$, a result of long-time stability for non-singular vortices concentrated near this initial configuration was obtained in \cite{Donati_2025_Crystal}, see also \cite{DR26} for the same problem on the rotating sphere. In that regard, we now perform a viscous analysis of this configuration, but we have to first introduce several dimensionless parameters that allow us to describe the solution precisely. Let $L_N=2\mathsf{r}\sin(\pi/N)$ be the length of the sides of our regular polygon and define the characteristic distance in the system as
\begin{equation}
\label{def:d}
    d:=\min\{ \mathsf{r},L_N\}.
\end{equation}
Note that $L_N=\cO(\mathsf{r}N^{-1})$ as $N\to \infty$. We then define our adimensional quantities of interest as
\begin{equation}
\label{def:adpar}
    \eps(t):=\frac{\sqrt{\nu t}}{d}, \qquad \delta:=\frac{\nu}{\Gamma}, \qquad T_\adv:=\frac{d^2}{\Gamma}, \qquad T_\diff:=\frac{d^2}{\nu},
\end{equation}
where $\eps(t)$ is the aspect ratio between the core size of the vortices and their mutual distances, $\delta$ the inverse circulation Reynolds number, $T_{\rm adv}$ the advection time  and $T_{\rm diff}$ is the standard diffusive time. Note that $T_{\rm adv}$ corresponds to a time where the point vortex polygon rotates by a fixed angle, so that one can effectively measure the point vortex dynamics at this time. In fact, to actually observe the inviscid dynamics described by \cref{eq:zjt} in the viscous problem, we need to work in the high Reynolds number regime $\delta\ll 1$, which implies $T_{\rm adv}\ll T_{\rm diff}$. Note that $\eps(T_{\rm diff})=1$, and we therefore will only consider sub-diffusive time scales.

\subsection{Main result}
We aim to describe the solution of equations~\cref{eq:NS} with initial datum given by~\cref{eq:omegain} up to a large time scale. 
Applying the result of~\cite{Gallay2011}, we know that the solution is well approximated by Lamb-Oseen vortices moving according to the point-vortex dynamics on a non-viscous time scale. Let us state this fact precisely. Let
\begin{equation}
    \omega_{\mathrm{PV}}(x,t) = \sum_{j=1}^N \frac{\Gamma}{4\pi\nu t} \exp\Big(-\frac{|x-z_j(t)|^2}{4\nu t}\Big) + \gamma\Gamma \exp\Big(-\frac{|x|^2}{4\nu t}\Big)
\end{equation}
where the trajectories $t \mapsto z_j(t)$ are given by~\eqref{eq:zjt}.
We have the following.
\begin{proposition}[Corollary of {\cite[Theorem 2.4]{Gallay2011}}]\label{prop:gallay1}
    Let $K > 0$ be a fixed constant. Then there exist  constants $\delta_0 \in(0,1)$ and  $C_0>0$ such that for every $\delta \le \delta_0$, the unique solution $\omega$ of equations~\cref{eq:NS} with initial datum~\cref{eq:omegain} satisfies
    \begin{equation}
        \frac{1}{|\Gamma|}\int_{\R^2} \big|\omega(x,t) - \omega_{\mathrm{PV}}(x,t)\big| \dd x \le C_0 \eps^2(t), \qquad \forall t \in (0,KT_{\rm adv}),
    \end{equation}
    where $\delta,\eps(t)$ and $T_{\rm adv}$ are defined in \cref{def:adpar}.
\end{proposition}

Fixing $\Gamma$ and taking $\nu$ sufficiently small, we naturally expect that not only the control of the solution becomes better, as it is already visible in Proposition~\ref{prop:gallay1}, but also that it should hold for a longer time, as the viscous effects that will eventually destroy the structure are less powerful. As remarked in \cite{donati2025fast}, it is only possible in all generality to adapt the proof made in \cite{Gallay2011} to obtain the control of the solution up to a time of order $T_{\rm adv} |\ln \delta|$, which is also the result obtained in \cite{CeSeis24}. So far, to the best of our knowledge, the only known way to push further in time is to take advantage of stability properties arising from special vortex configurations. This was done in \cite{dolce2024long} for the pair of counter-rotating vortices and in \cite{zhang2025long} for any pair of vortices. We establish an equivalent result for our configuration of polygonally arranged vortices with (or without) central vortex. We show that for longer times, the angular velocity of the configuration of vortices starts changing due to viscous effects. 
We thus introduce corrected approximate positions as follows.
For all $j \in \{1,\ldots,N\}$, let
\begin{equation}
    \zeta_j(t) = R(t)\Big(\cos\Big(\frac{2\pi j}{N}+\alpha(t)\Big) ,\sin\Big(\frac{2\pi j}{N}+\alpha(t) \Big) \Big),
\end{equation}
where with $R(0)=\mathsf{r}$ and $\alpha(0) = 0$. These quantities act as \emph{modulation parameters} used to precisely track the center of each Lamb-Oseen vortex. We then define
\begin{align}
    \label{eq:PVnu}&\omega_{\mathrm{PV}_\nu}(x,t) = \frac{\Gamma}{4\pi\nu t}\Bigg(\sum_{j=1}^N  \exp\Big(-\frac{|x-\zeta_j(t)|^2}{4\nu t}\Big) + \gamma \exp\Big(-\frac{|x|^2}{4\nu t}\Big)\Bigg), \\
    \label{eq:appM}&\omega_{{\rm app},M}(x,t)=\frac{\Gamma}{\nu t}\sum_{m=2}^M\eps^m(t)\Bigg(\sum_{j=1}^N  \Omega_{m}\Big(\frac{x-\zeta_j(t)}{\sqrt{\nu t}}\Big) + \gamma\widehat{\Omega}_{m}\Big(\frac{x}{\sqrt{\nu t}}\Big)\Bigg).
\end{align}
Here, $\omega_{{\rm app},M}$ captures higher-order viscous corrections, up to a fixed order $M$, to the radial profile close to the vortex cores. Note that these corrections slowly evolve in a self-similar way. The main result of this paper is the following.

\begin{theorem}\label{theo:main}
        Let $\sigma\in [0,1)$ and $M>(5-\sigma)/(1-\sigma)$. Then there exist constants $\delta_1 \in (0,1)$, and  $C_1>0$ such that for every $\delta \in(0, \delta_1]$, there exists $\Omega_m,\widehat{\Omega}_m$ for $m=2,\dots,M$, and modulation parameters $R(t),\alpha(t)$ such that the unique solution $\omega$ of equations~\cref{eq:NS} with initial datum~\cref{eq:omegain} satisfies
    \begin{equation}
    \label{bd:mainNLintro}
        \frac{1}{|\Gamma|}\int_{\R^2} \big|\omega - (\omega_{\mathrm{PV}_\nu}+\omega_{{\rm app},M})\big|(x,t) \dd x \le C_1 \delta\eps^2(t), \qquad \forall t \in \big(0,T_{\rm adv} \delta^{-\sigma} \big).
    \end{equation}
    Furthermore, the radius $R(t)$ and the angular speed $\alpha'(t)$ are smooth functions satisfying 
    \begin{equation}
    \label{eq:expRa}
    R(t)=\mathsf{r}\big(1+r_6 \eps^6(t) + \cO(\eps^7(t))\big), \qquad \alpha'(t) = \mathsf{a} \big(1+\alpha_4\eps^4(t)+\cO(\eps^5(t)+\delta \eps^4(t))\big),
\end{equation}
where $\alpha(0)=0$ and the exact values of the constants $r_6$ and $\alpha_4$ are given in Proposition~\ref{prop:deviation}. 
\end{theorem}
Naturally, as $\delta \to 0$, the error in \cref{bd:mainNLintro} is much smaller than that in Proposition~\ref{prop:gallay1}\footnote{In fact, in \cite{Gallay2011} the final result includes the approximate solution up to order $M=4$ with $\delta^{-\sigma}$ replaced by the constant $K$ in Proposition~\ref{prop:gallay1}. The analogue of Theorem~\ref{theo:main} is proved with $\delta \eps^2(t)$ replaced by $\eps^3(t)$, but on the time scale under consideration, $\delta \ll \eps(t)$ for any $t \ge T_{\rm adv}$ and thus the error is still much smaller after the advection time.}. 
Moreover, since $T_{\rm adv} \delta^{-\sigma} \gg T_{\rm adv}$ and $T_{\rm adv} \delta^{-\sigma} \gg T_{\rm adv} |\ln\delta|$, Theorem~\ref{theo:main} controls the solution on a significantly longer time scale than Proposition~\ref{prop:gallay1}. Regarding the  radius and angular speed, \cref{eq:expRa} only highlights the leading-order viscous corrections to the inviscid Helmholtz-Kirchhoff dynamics. The full, rigorous characterization of $R(t)$ and $\alpha(t)$ is not entirely captured by \cref{eq:expRa} alone. Their precise definition is one of the central components of our proof and relies on technical constraints that we will carefully detail later in the paper. Nonetheless, we emphasize that relation~\cref{eq:expRa} implies
\begin{equation}\label{maj:zeta_j}
  \frac{1}{\mathsf{r}}  |\zeta_j(t) - z_j(t)| = \cO\Big( \frac{t}{T_{\rm adv}} \eps^4(t) \Big).
\end{equation}
Observing that $t \eps^4(t) / T_{\rm adv} \ll 1$ is valid only for $t \ll T_{\rm adv} \delta^{-2/3}$, it follows that the point-vortex dynamics in \cref{eq:zjt} remains a good approximation of the viscous vortices up to a time of order $\cO(T_{\rm adv}\delta^{-2/3})$. Beyond this time scale, the configuration becomes out of phase relative to the inviscid point-vortex motion, see Remark~\ref{rem:delta23} and Corollary~\ref{cor:alpha} for more details.

To prove Theorem~\ref{theo:main}, we proceed similarly to \cite{dolce2024long}. After reformulating the equations in self-similar variables, the first step is to construct an approximate solution $\omega_{\rm app} = \omega_{\mathrm{PV}_\nu} + \omega_{{\rm app},M}$ that remains close to $\omega$ over the time interval $(0,T_\adv\delta^{-\sigma} )$ up to any power of $\eps(t)$. The construction of such an accurate approximate solution is essential for the following reason: the control of the configuration over such long time scales requires an extremely accurate knowledge of the vortex positions, as any positional error would accumulate in time and eventually invalidate an estimate like in \cref{bd:mainNLintro}.  
In addition, the controls are obtained in more refined norms. We observe that for the particular choice of $\gamma$ used in \cite{Donati_2025_Crystal} our approximate solution starts at order $\eps^3(t)$ rather than the standard $\eps^2(t)$ found in all previous constructions \cite{dolce2024long,zhang2025long}, see discussion in the next subsection.

\begin{remark}[Bounds on the time scale]
Note that, in this vortex crystal configuration, it is expected that vortices merge abruptly at a time of order $\cO( T_{\rm adv} \delta^{-1} )=\cO( T_{\rm diff} )$, which corresponds to the regime where the aspect ratio $\eps(t)$ becomes of order one. Therefore, approximating $\omega$ by $\omega_{\mathrm{PV}_\nu}$ is only physically meaningful up to a time of at most $t_0 T_{\rm adv}\delta^{-1}$ for some sufficiently small constant $t_0>0$. By taking $\sigma \to 1^-$, we almost reach this  limit. However, our method relies on an asymptotic expansion in $\eps(t)$, but  one cannot completely decouple the parameters $\eps(t)$ and $\delta$, as the construction introduces remainder terms of size $\cO(\delta^{-2}\eps^{M+1}(t))$. On time scales of order $\cO(T_{\rm adv}\delta^{-\sigma})$, these errors can be made arbitrarily small by choosing $M$ large enough depending on $\sigma$, but this inherently limits the analysis to the regime $\eps(t) \ll \delta^s$ for some small but strictly positive $s$. Therefore, even though we control the solution up to a time that is arbitrarily close to the maximal time at which vortices remain well separated, we are mathematically far from reaching exactly $\sigma = 1$. This limitation is also present in the more structurally stable configuration of a vortex dipole \cite{dolce2024long}, where no abrupt merging is expected and the dynamics seem to evolve smoothly for all times \cite{DelRos2009}. Hence, we believe that reaching the diffusive time scale is not merely a technical artifact of the proof. Instead, reaching the diffusive time scale requires the introduction of fundamentally new ideas, as the nonlinearity and the diffusion become of the same order of magnitude. In essence, prior to the diffusive time, diffusion is treated purely as a perturbative effect that smoothly changes the shape and position of the vortices.
\end{remark}

\begin{remark}[On the circular vortex sheet limit]
\label{rem:sheet}
If we rescale our initial data in \cref{eq:omegain} as
\begin{equation}
\label{eq:ominN}
    \omega^{in}_N = \frac{\Gamma}{N}\sum_{j=1}^N\delta_{z_j}+\gamma \Gamma\delta_0,
\end{equation}
we observe that in the $N\to \infty$ limit, this sequence of measures converges to the circular vortex sheet with a central point vortex, given by 
\begin{equation}
    \label{eq:omininfty}
    \omega^{in}_{\infty} := \frac{\Gamma}{2\pi \mathsf{r}}\delta_{|x|=\mathsf{r}} + \gamma\Gamma \delta_0.
\end{equation}
The study of the evolution and stability of vortex sheets is a classical problem in the inviscid setting, where Kelvin-Helmholtz instabilities are commonly observed for both flat and circular sheets \cite{krasny1986study,MajdaBertozzi,moore1974stability,wilcox2024modern}. Moreover, a desingularization procedure for more general geometries of these singular structures was established only recently in \cite{enciso2025desingularization}.  A classical numerical approach to simulate the dynamics of vortex sheets relies on approximating them with a collection of point vortices \cite{krasny1986study}. In our setting, due to the specific choice of the initial data in \cref{eq:ominN} and the fact that the Navier-Stokes equations preserve the $N$-fold symmetry, one should not see any instability in a suitable limiting regime. Rather, one expects to effectively carry out the $N\to \infty$ limit for any fixed $t>0$ on an appropriate time scale. Indeed, with the initial data given by \cref{eq:omininfty}, the unique solution to \cref{eq:NS} is simply the radial function $\omega_\infty(t)= e^{\nu t \Delta}\omega_{\infty}^{in}$. This coincides with the solution to the 2D heat equation with data \cref{eq:omininfty}, which can be approximated by $\widetilde{\omega}_N (t):= e^{\nu t \Delta}\omega_{N}^{in}$ for any fixed $t>0$ as $N\to \infty$.

To formalize the limit $N\to \infty$ in our framework, recall that $L_N/\mathsf{r} = \cO(N^{-1})$, which implies that the dimensionless parameters defined in \cref{def:adpar} (with $\Gamma$ replaced by $\Gamma/N$)  now scale as \[\eps(t) = \cO(N\sqrt{\nu t}/\mathsf{r}),\quad T_{\rm adv} = \cO(\mathsf{r}^2/(\Gamma N)), \quad T_{\rm diff} = \cO(\mathsf{r}^2/(\nu N^2)).\]
Our analysis requires the time scale separation $T_{\rm adv}\ll T_{\rm diff}$, which is equivalent to imposing $\delta \ll 1/N$. To reach a dynamically relevant conclusion, we must ensure that the upper bound of our time interval, $T_{\rm adv}\delta^{-\sigma}$, does not vanish as $N\to \infty$, so we require $\delta\lesssim N^{-1/\sigma}$. Under these hypotheses, a careful tracking of the $N$-dependence of the constants suggests that it is feasible to take the $N\to \infty$ limit for any fixed $t\in (0,T_{\rm adv}\delta^{-\sigma})$ in Theorem~\ref{theo:main}. 
Specifically, since $\eps^2(t) \lesssim N \delta^{1-\sigma}$, if $\delta$ decays sufficiently fast with $N$ (for instance, $\delta=\min \{N^{-\frac{1}{\sigma}^{-}},N^{-\frac{1}{1-\sigma}^{-}}\}$), all the higher-order corrections in \cref{eq:appM} will vanish as $N\to \infty$. While one must carefully track the $N$-dependent constants in the definition of the approximate solution and the bound in Theorem \ref{theo:main}, we observe for example that the profile $\Omega_{2}$ defined in Proposition~\ref{prop:Om2} contains $N$-dependent constants that precisely balance out to yield a term of order $\cO(1)$ (see Remark \ref{rem:Om2Ninfinity}), whereas $\widehat{\Omega}_m=0$ for $m=2,\dots, N-1$ by the $N$-fold symmetry.
Furthermore, since the angular speed satisfies $\mathsf{a}\to \Gamma(1+2\gamma)/(4\pi \mathsf{r}^2)$ as $N\to \infty$, Theorem~\ref{theo:main} implies that the solution is well-approximated by $N$ Lamb-Oseen vortices rotating with a constant angular speed. This describes the evolution of $N$ Dirac masses under the heat equation in a rotating frame. The only quantitative difference between the Navier-Stokes and the heat equation evolutions then lies in this rotation of the discrete configuration. However, as $N \to \infty$, the limiting measure \cref{eq:omininfty} is purely radial and thus invariant under continuous rotations. This rotational effect therefore becomes irrelevant in the limiting process, smoothly reconciling the discrete dynamics with the purely diffusive evolution of the continuous vortex sheet. 
Consequently, we expect that our analysis can provide a vanishing viscosity justification for the point-vortex approximation of a circular vortex sheet, and we emphasize that reaching time scales of order $\cO(T_{\rm adv}\delta^{-\sigma})$ is a crucial ingredient to carry out this limiting process. 
\end{remark}

\subsection{Non radial viscous corrections}\label{sec:corrections}

The original proof of Proposition~\ref{prop:gallay1} in \cite{Gallay2011} already requires constructing an approximate solution that takes into account a non-radial deformation of the vortices due to viscous effects, specifically corresponding to the case $M=4$. This phenomenon of vortex deformation is well known by physicists (see for instance \cite{DiVer2002,Dritschel1985}), and was also investigated mathematically in more detail in \cite{donati2025fast}. Let us discuss and illustrate it briefly in our particular case. There exists a critical value for $\gamma$, which we denote by 
\begin{equation}
\label{eq:gammaN*}
    \gamma_N^* := \frac{(N-1)(N-5)}{12},
\end{equation}
which is a threshold in the behavior of the \emph{outer} vortices (meaning all except the central one). When $\gamma > \gamma_N^*$, the outer vortices deform in an elliptical shape with their larger radius pointing towards the central vortex. When $\gamma < \gamma_N^*$, the outer vortices deform elliptically with their smaller radius pointing towards the central vortex. The intensity of this deformation is proportional to $\eps^2(t)(\gamma-\gamma_N^*)$, and in the critical case $\gamma = \gamma_N^*$, there is no elliptical deformation of the outer vortices. This can be observed in the numerical simulations presented in Figure~\ref{fig:deformation_sim}. One can still mathematically compute a non-radial deformation at order $\eps^3(t)$ in that case, which happens to be 3-fold symmetric. Please note that this value $\gamma_N^*$ is the precise choice of $\gamma$ that leads to the long-time confinement in the non-viscous case as studied in \cite{Donati_2025_Crystal}.

The central vortex also deforms due to viscous effects. However, the solution $\omega$ of the Navier-Stokes equations~\cref{eq:NS} with initial data~\cref{eq:omegain} is $N$-fold symmetric at all times. The central vortex thus can only undergo an $N$-fold symmetric deformation, at order at least $\eps^N(t)$. Therefore, unless $N=2$, these effects are very weak, and the instabilities inherent to the numerical simulations will destroy the symmetries before one can hope to see such a subtle deformation become relevant.

\begin{figure}[hbtp]
\centering
\begin{subfigure}{0.25\linewidth}
  \includegraphics[width=\linewidth]{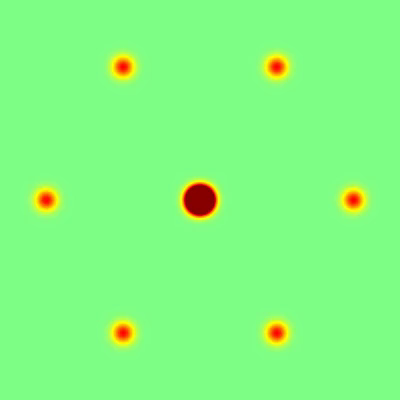}
\end{subfigure}\hspace{6mm}
\begin{subfigure}{0.25\linewidth}
  \includegraphics[width=\linewidth]{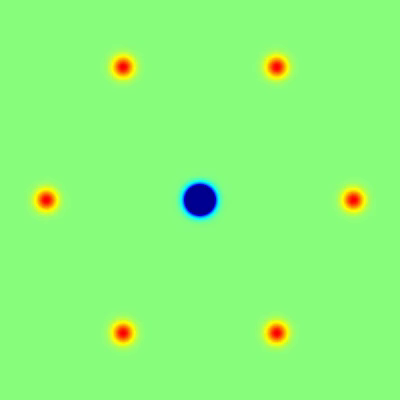}
\end{subfigure}\hspace{6mm}
\begin{subfigure}{0.25\linewidth}
  \includegraphics[width=\linewidth]{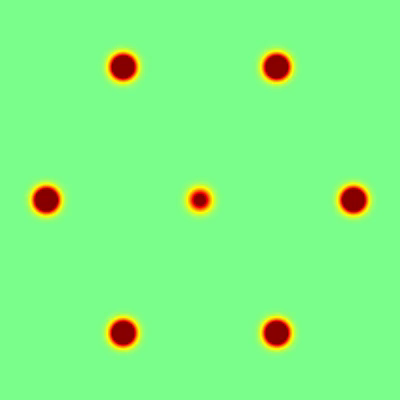}
\end{subfigure}

\vspace{1mm}

\begin{subfigure}{0.25\linewidth}
  \includegraphics[width=\linewidth]{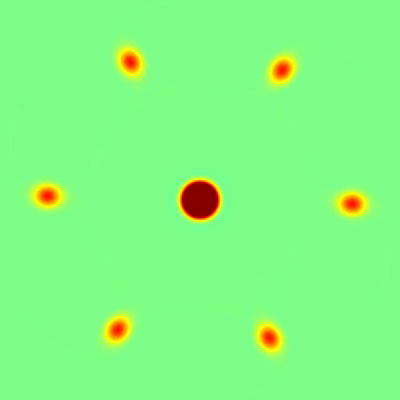}
\end{subfigure}\hspace{6mm}
\begin{subfigure}{0.25\linewidth}
  \includegraphics[width=\linewidth]{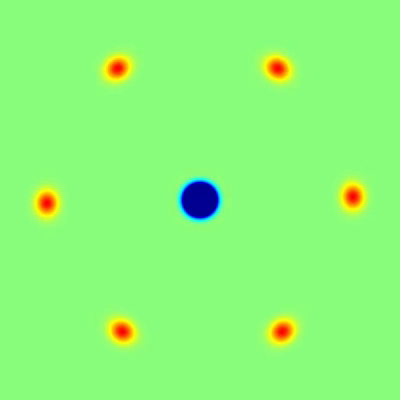}
\end{subfigure}\hspace{6mm}
\begin{subfigure}{0.25\linewidth}
  \includegraphics[width=\linewidth]{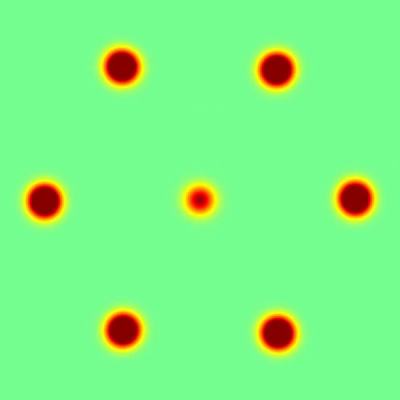}
\end{subfigure}
\caption{Initially radially symmetric concentrated vortices (top) evolve in time (bottom), becoming more elliptical with different orientations in the case $\gamma = 10$ (left) and in the case $\gamma = -10$ (middle). In the case $\gamma= \gamma_6^* = \frac{5}{12}$ (right), the vortices do not display any clear 2-fold symmetric deformation. The 3-fold deformation happening in that case, as well as the 6-fold deformation of the central vortices in each case, are too subtle to be seen here. The bottom pictures are taken in all three cases after the whole configuration has completed a rotation of angle $2\pi/3$.}
\label{fig:deformation_sim}
\end{figure}

These facts are mostly direct consequences of Theorem~\ref{theo:main} completed with the explicit computation of the second-order term of the approximate solution that we construct. More examples, with some details and numerical simulations, with various choices of parameters $(N,\gamma)$ can be found in Section~\ref{sec:sim}.

\subsection{Analogies and differences with previous related results}
From a technical point of view, the proof of Theorem \ref{theo:main} relies on the techniques developed to obtain the analogous result in the simpler configuration of a point vortex dipole in \cite{dolce2024long}.  That result itself builds upon the recent revisitation of \emph{Arnold's variational viewpoint} by Gallay and \v{S}ver\'ak \cite{GaSarnold}, which they successfully applied to the more geometrically involved configuration of a single vortex ring in the 3D axisymmetric Navier--Stokes equations \cite{GaSring}. The point vortex dipole is arguably the simplest nontrivial configuration to study, in the sense that it enjoys rigid symmetries that prevent violent instability mechanisms, such as the merging commonly observed for like-signed vortices \cite{Meunier2005}. From a mathematical perspective, these symmetries are reflected in key properties of the operators appearing in the equation governing the dipole evolution. These properties lead to almost exact symmetries in the approximate solution and crucial cancellations in the nonlinear problem, which play a fundamental role in justifying the validity of the approximate solution \cite{dolce2024long}.

In the case of two point vortices with general circulations recently studied by P. Zhang and Y. Zhang  in \cite{zhang2025long} (corresponding to the case $N=1$, $\gamma\neq -1$ in our setting\footnote{Note that $N=2$ and $\gamma=0$ is equivalent to $N=1$ and $\gamma=1$ up to redefining the center and the distance.}), one can still construct a very accurate approximate solution by adapting the methods of \cite{dolce2024long}.  However, it is no longer possible to reduce the analysis to a single scalar equation (unless $\gamma=1$), and the operators in the equations exhibit a more complicated dependence on the modulation parameters, such as the radius and angular speed. Although the configuration still enjoys some key structural symmetries, it is less clear how to exploit them effectively at certain steps. This becomes a serious obstacle in proving that the full solution remains close to the approximate one on very long time scales, as in Theorem \ref{theo:main}. Indeed, a key point in \cite{dolce2024long} is the use of the vertical position as a modulation parameter to fix the centers of vorticity, whereas the horizontal position is automatically fixed by symmetry. For two vortices with general circulations, one can control the vortex centers by passing to a rotating frame and modulating both the angular speed and the mutual distance. However, unlike the dipole case, neither of these parameters is automatically fixed by symmetry. As a consequence, potentially dangerous terms appear in the full nonlinear problem that are not easily controlled by the symmetry arguments. In \cite{zhang2025long}, the authors introduced a new strategy to bypass the problems given by some of these terms at the cost of constructing \emph{pseudo-momenta}, which are needed to replace standard first-order moments commonly used to control the vortex centers. We note that these pseudo-momenta are related to the presence of elements in the kernel of the leading-order operator arising from the linearization around the full approximate solution. These kernel elements are themselves linked to the symmetries of the configuration (especially the invariance with respect to a rotation around the origin of the whole configuration). Hence, symmetries still play a key role, but their use is  more subtle in \cite{zhang2025long}.

For the more general configuration of vortex crystals with a central vortex, we are in a situation that is structurally close to that of \cite{zhang2025long}, but the equations are further complicated by nontrivial interactions among all exterior vortices.  Nevertheless, we are able to fully exploit the only robust symmetries available, namely the $N$-fold symmetry and the rotational invariance of the total configuration. This allows us to avoid the introduction of pseudo-momenta and instead rely on a key quantity that is directly linked to these symmetries: the \emph{total angular momentum}. This quantity varies significantly only on diffusive time scales (see \cref{eq:angmom}), and it can be used to measure viscous changes to the radius of the polygon. By constructing an approximate radius $R_{\rm app}(t)$ that precisely follows the vortex centers of the approximate solution, see \cref{def:Rapp}, and modulating the angular speed of the full configuration, see \cref{def:alphaNL}, we use the almost-conservation of the total angular momentum to show that the centers of the exterior vortices can only deviate by a factor of size $\cO(\delta \eps^3)$ at most. This is related to a highly nontrivial property of the \emph{Eulerian} part of the approximate solution (namely, the part that approximately solves the inviscid problem), and it relies on suitable cancellations related to the exact conservation of the total angular momentum in the inviscid case, see Lemma \ref{lem:keyLEuler}. This result is a fundamental building block of our analysis that allows us to fix the radius $R(t)$ as the approximate one $R_{\rm app}(t)$. It highlights how robust symmetries can be used directly and effectively. At the level of the full nonlinear problem, symmetries continue to play an essential role at several stages of the argument. One key fact is the self-adjointness of an operator used to define an energy functional. Another one is that a potentially dangerous term generated by the modulation of the angular speed belongs (approximately) to the kernel of the operator naturally arising from the linearization around the approximate solution. Both properties are in fact linked to the rotational invariance of the whole configuration.

\subsection{Plan of the paper}

The paper is organized as follows. In Section~\ref{sec:eq}, we write the equations in self-similar variables and start exploiting the symmetries of the problem to reduce some degrees of freedom. We also introduce the relevant operators and write the equations in a compact form. In Section~\ref{sec:Func}, we discuss the important properties of these operators and their associated functional spaces that will be of use later, setting up the analytic framework for our proof. In Section~\ref{sec:App}, we construct the approximate solution $\omega_\app$ and compute explicitly the first few viscous corrections. Then, Section~\ref{sec:nonlinear} concludes the proof of Theorem~\ref{theo:main} by obtaining the control of the difference $|\omega-\omega_{\app}|$ in the relevant norms, from which we deduce the desired behavior of $\omega$. In Section \ref{sec:sim} we present a comparison of our results with numerical simulations made with \href{http://basilisk.fr/}{Basilisk}. 
In the appendices, we give the details of some proofs that are non-trivial but follow by explicit computations related  to symmetries, and  we also compute terms for which obtaining an explicit expression is interesting but not necessary to the proofs.

\section{Reformulation of the problem}\label{sec:eq}
In this section we reformulate the problem and we set up the notation and framework used in the rest of the paper. 
\subsection{Notations and conventions}

With the adimensional parameters introduced in \cref{def:adpar}, we will work in the asymptotic regime $\delta \to 0$ and, fixing $\Gamma$ once and for all, this is equivalent to the standard vanishing viscosity regime $\nu \to 0$. However, it will be very convenient to perform expansions in terms of $\eps$. Since $\eps^2 = \delta \frac{t}{T_\adv}$, we have that $\eps(t) \ll 1 \Longleftrightarrow \delta \ll 1$ uniformly on the time interval $t \in T_\adv \delta^{-\sigma}$ for any $\sigma < 1$. This leads us to the following notation.
\begin{definition}
     Given a map $t\mapsto f(t)$ evolving in a space $X$, we denote by $f = \cO_X(\eps)$ the fact that for any given $\sigma < 1$, there exists a constant $C$ not depending on $\delta$ and $t$, and a constant $\delta_0 > 0 $ such that for every $\delta \in (0,\delta_0)$ and for every $t \le T_\adv \delta^{-\sigma}$, there holds that $\| f(\cdot,t) \|_X \le C \eps(t)$. When $f$ depends on $t$ only through $\eps(t)$, this definition conveniently coincides with considering $\eps$ as an abstract parameter and taking the asymptotics $\eps \to 0$, hence the use of the notation $\cO(\eps)$.
\end{definition}
   Note also that, in the definition of $\delta,T_\adv, T_\diff$, we could  replace $\Gamma$ with $\Gamma\max\{1,|\gamma|\}$. However, it would simply change how the parameter $\gamma$ scales in the subsequent analysis; therefore, we prefer to use the same conventions that have already appeared in \cite{dolce2024long}.

\medskip

Relative equilibria in the point vortex system are intimately related to conserved quantities in \cref{eq:HK}. For the dynamics of \cref{eq:NS}, two important quantities are the \textit{mean} $\mrM(\omega)$ and the \textit{linear momentum} $\mrm^1(\omega)=(\mrm^1_1(\omega),\mrm^1_2(\omega))$ defined as 
\begin{equation}\label{def:premiers_moments}
    \mrM(\omega)=\int_{\RR^2}\omega(x)\dd x, \qquad \mrm^1_j(\omega)=\int_{\RR^2}x_j\omega(x)\dd x, \quad i=1,2.
\end{equation}
We will perform operations on vectors in $\RR^2$ as if they were complex numbers. For instance, given $a=(a_1,a_2)\leftrightarrow a_1+ia_2 $ and $ b=(b_1,b_2)\leftrightarrow b_1+ib_2$ we identify 
\begin{equation}
\label{def:multiplication}
 ab\leftrightarrow (a_1b_1-a_2b_2, a_1b_2+a_2b_1)\qquad \text{or} \qquad a^{-1}b\leftrightarrow \frac{1}{|a|^2}(a\cdot b, a^\perp \cdot b).
\end{equation}
We also identify the rotation matrices 
\begin{equation}
    \label{def:Qj}
Q_j=\begin{pmatrix}
    \cos(\frac{2\pi}{N}j) & -\sin(\frac{2\pi}{N}j)\\
    \sin(\frac{2\pi}{N}j) &\cos(\frac{2\pi}{N}j) 
\end{pmatrix}, \qquad \text{ for } j=1,\dots, N.
\end{equation}
with the complex number $\e^{i\frac{2\pi}{N}j}$.
For a vector $a=(a_1,a_2)\in \RR^2$ we write $\Re(a)=a_1, \, \Im(a)=a_2$. Then, with this notation we define the $k$-th moment vectors as 
\begin{equation}
    \label{def:mk} {\rm m}^k(f)\coloneqq \int_{\RR^2}\xi^k f(\xi)\dd \xi, \qquad \text{for } k\in \NN.
\end{equation}
As an example, we compute that ${\rm m}^0(f)=({\rm M}(f),0)$ which we will identify with $\mrM(f)$ and
\begin{equation}
    {\rm m}^1(f)=\int_{\RR^2}(\xi_1,\xi_2) f(\xi)\dd \xi, \qquad {\rm m}^2(f)=\int_{\RR^2}(\xi_1^2-\xi_2^2,2\xi_1\xi_2) f(\xi)\dd \xi.
\end{equation}
Then, we need to make use of the angular momentum that we define through the operator 
\begin{align}
\label{def:L}
    \mathrm{L}(f)\coloneqq \int_{\RR^2}|\xi|^2f(\xi)\dd \xi.
\end{align}
We denote the $L^2(\RR^2)$ inner product as 
\begin{align}
    \langle f,g\rangle=\int_{\RR^2} f(\xi)g(\xi)\dd \xi.
\end{align}

We conclude this section with the introduction of streamfunction and Poisson brackets. The incompressibility condition $\nabla \cdot u = 0$ implies the existence of a streamfunction $\psi$ such that $u = \nabla^\perp \psi$ and $\Delta \psi = \omega$. To get the expression of $\psi$ in terms of $\omega$, and thus recover the Biot-Savart law in~\cref{eq:NS}, one can invert the Laplacian with  the following formula:
\begin{equation}
\label{def:psiom}
    \psi(x,t)= [\Delta^{-1} \omega(\cdot , t)](x) := \frac{1}{2\pi}\int_{\RR^2}\ln|x-y|\omega(y,t) \dd y.
\end{equation}
We denote by $\{f,g\}=\nabla^\perp f\cdot \nabla g=(\de_1 f)(\de_2 g)-(\de_2f)(\de_1 g)$ the standard Poisson bracket of $f$ and $g$ and then notice, since we will make extensive use of that notation, that the first two equations in the Navier-Stokes system~\cref{eq:NS} can be changed to \cref{def:psiom} and
\begin{equation}
    \partial_t \omega + \{\psi,\omega\} = \nu \Delta \omega.
\end{equation}

\subsection{Symmetries and self-similar variables}
For the 2D Navier-Stokes equation \cref{eq:NS}, it is well known that the total angular momentum $\mathrm{L}(\omega(t))$, as defined in \cref{def:L}, can be computed explicitly and changes significantly only after diffusive times. This is related to the fact that the equation is rotationally invariant and that $\mathrm{L}(\omega(t))$ is exactly conserved when $\nu=0$. We then recall the following basic fact.
\begin{lemma}
    \label{lem:NSmom} Let $\omega$ be the solution to \cref{eq:NS} with initial data \cref{eq:omegain}. Then
\begin{align}
\label{eq:angmom}
\mathrm{L}(\omega(t))=\mathrm{L}(\omega^{in})+4\nu t\mrM(\omega^{in})=\Gamma N\big(\mathsf{r}^2+4\nu t \big(1+\frac{\gamma}{N}\big)\big).
\end{align}
\end{lemma}
\begin{proof}
By testing the equation \cref{eq:NS} with $|x|^2$, we get
\begin{equation}
    \de_t \mathrm{L}(\omega)=-\langle\{\psi,\omega\},|x|^2\rangle+\nu \langle \Delta \omega,|x|^2\rangle. 
\end{equation}
Exploiting the identities
\begin{align}
    &\langle \{\psi,\omega\},|x|^2\rangle=\langle\psi,\{\omega,|x|^2\}\rangle=-2\langle \psi,\de_\theta \omega\rangle=-2\langle \psi,\de_\theta \Delta\psi\rangle=0\\
    &\nu \langle \Delta\omega,|x|^2\rangle=4\nu\mrM(\omega(t))=4\nu\Gamma (N+\gamma),
\end{align}
the result follows by integrating in time.
\end{proof}
The initial data \cref{eq:omegain} is $N$-fold symmetric and it is well known that the Navier-Stokes equations \cref{eq:NS} preserve this symmetry. 
We can thus decompose the solution of equations \cref{eq:NS} as
\begin{equation}
    \omega(x,t)= \sum_{j=1}^N \tilde{\omega}\big(Q_{-j}x,t) + \widehat{\omega}(x,t)
\end{equation}
where $\tilde{\omega}$ is the vorticity profile of each exterior vortices and $\widehat{\omega}$ is the vorticity profile of the central vortex.
Recalling that the mean $M(\omega)$ and the linear momentum $\mrm^1(\omega)$ are conserved, the $N$-fold symmetry has the following consequence on the central blob of vorticity.
\begin{lemma}\label{lem:centre0}
    For every $t\ge 0$, we have that $\mrm^1(\widehat{\omega}(t)) = 0$.
\end{lemma}

\begin{proof} By conservation of $\mrm^1(\omega(t))$, we have that
\begin{equation}
    0 = \int_{\R^2}x \omega(x,t)\dd x = \sum_{j=1}^N\int_{\R^2}x \tilde{\omega}(Q_{-j}x,t)\dd x + \int_{\R^2}x \widehat{\omega}(x,t)\dd x.
\end{equation}
Then, we compute by change of variables that
\begin{equation}
    \sum_{j=1}^N\int_{\R^2}x \tilde{\omega}(Q_{-j}x,t)\dd x = \left(\sum_{j=1}^N Q_j\right)\int_{\R^2}x \tilde{\omega}(x,t)\dd x = 0,
\end{equation}
hence,
\begin{equation}\label{eq:center_vortex_centered}
    \int_{\R^2}x \widehat{\omega}(x,t)\dd x = 0.
\end{equation}
\end{proof}
As is usually done in the context of studying solutions arising from Dirac masses in the Navier-Stokes equations (see \cite{Gallay2011,dolce2024long,donati2025fast,zhang2025long}), we introduce the rescaled vorticity functions
\begin{equation}
    \tilde{\omega}(x,t) = \frac{\Gamma}{\nu t} \Omega\left( \frac{x-Z(t)}{\sqrt{\nu t}},t\right)
\end{equation}
and
\begin{equation}
    \widehat{\omega}(x,t) = \gamma \frac{\Gamma}{\nu t} \widehat{\Omega}\left( \frac{x}{\sqrt{\nu t}},t\right),
\end{equation}
where
\begin{equation}
    \label{def:Zalpha}
    Z(t)=(Z_1(t),Z_2(t))
\end{equation}
is a quantity to choose that acts as the position of the $N$-th exterior vortex. By $N$-fold symmetry, the positions of the other exterior vortices are $(Q_j Z(t))_j$, and from \cref{lem:centre0}, we choose to define the position of the central vortex to be 0.

Denoting by $\xi$ and $\widehat{\xi}$ the new variables, namely letting
\begin{equation}\label{def:xi}
    \xi = \frac{x-Z(t)}{\sqrt{\nu t}}
\end{equation}
and 
\begin{equation}\label{def:xi0}
    \widehat{\xi} = \frac{x}{\sqrt{\nu t}},
\end{equation}
and letting for $1 \le j \le N$
\begin{equation}\label{def:xi_j}
    \xi_j = \frac{Q_{-j}x-Z(t)}{\sqrt{\nu t}} = Q_{-j} \xi + \frac{Q_{-j}Z(t)-Z(t)}{\sqrt{\nu t}},
\end{equation}
we arrive in the end to the following ansatz for $\omega$ in the self-similar variables:
\begin{equation}\label{eq:omSS}
    \omega(x,t)=\frac{\Gamma}{\nu t}\sum_{j=1}^N \Omega(\xi_j,t)+\frac{\gamma \Gamma}{\nu t} \widehat{\Omega}(\widehat{\xi},t).
\end{equation}
Then, we also decompose the streamfunction as 
\begin{equation}
    \psi(x,t)=\Gamma \sum_{j=1}^N \Psi(\xi_j,t)+\gamma \Gamma \widehat{\Psi}(\widehat{\xi},t),
\end{equation}
where
\begin{equation}
    \Psi=\Delta^{-1}\Omega \; \text{ and } \; \widehat{\Psi}=\Delta^{-1} \widehat{\Omega}, \quad \text{recalling that}\quad \Delta^{-1}f(\xi):=\frac{1}{2\pi}\int_{\RR^2}\ln|\xi-\eta| f(\eta)\dd \eta.
\end{equation}

For later needs, we observe that we have the following identities to express $\xi_l$ in terms of $\xi_j$ for $1 \le j,l \le N$:
\begin{align}\label{eq:xi_l_in_xi_j}
    \xi_l & = Q_{-l+j} \xi_j + \frac{Q_{-l+j} Z(t) - Z(t)}{\sqrt{\nu t}} \\
    \xi_l & = Q_{-l} \widehat{\xi} - \frac{Z(t)}{\nu t} \\
    \widehat{\xi} & = Q_j\xi_j + \frac{Q_j Z(t)}{\nu t}.
\end{align}
We then have the following symmetry result for the central vortex.
\begin{lemma}\label{lem:Nfold}
    The functions $\widehat{\Omega}$ and $\widehat{\Psi}$ are $N$-fold symmetric.
\end{lemma}
\begin{proof}
    We know that $\omega$ is $N$-fold symmetric, therefore, for every $x \in \R^2$ and every $l \in \{1,\ldots,N\}$,
    \begin{equation}
        \omega(x,t) = \omega(Q_j x,t)
    \end{equation}
    and thus
    \begin{align}
        \frac{\Gamma}{\nu t}\sum_{j=1}^N \Omega(\xi_j,t)+\frac{\gamma \Gamma}{\nu t} \widehat{\Omega}(\widehat{\xi},t) &  = \frac{\Gamma}{\nu t}\sum_{j=1}^N \Omega\bigg(\frac{Q_{-j+l}x-Z(t)}{\sqrt{\nu t}},t\bigg)+\frac{\gamma \Gamma}{\nu t} \widehat{\Omega}(Q_l \widehat{\xi},t) \\
        & = \frac{\Gamma}{\nu t}\sum_{j=1}^N \Omega(\xi_{j-l},t)+\frac{\gamma \Gamma}{\nu t} \widehat{\Omega}(Q_l \widehat{\xi},t)
    \end{align}
    with the convention that $\xi_{-k} = \xi_{N-k}$ for $k \in \{0,\ldots,N-1\}$. By reorganizing the terms we get
\begin{equation}
    \frac{\Gamma}{\nu t}\sum_{j=1}^N \Omega(\xi_j,t)+\frac{\gamma \Gamma}{\nu t} \widehat{\Omega}(\widehat{\xi},t) = \frac{\Gamma}{\nu t}\sum_{j=1}^N \Omega(\xi_j,t)+\frac{\gamma \Gamma}{\nu t} \widehat{\Omega}(Q_l \widehat{\xi},t)
\end{equation}
namely
\begin{equation}
    \widehat{\Omega}(\widehat{\xi},t) = \widehat{\Omega}(Q_l \widehat{\xi},t).
\end{equation}
Thus $\widehat{\Omega}$ is $N$-fold symmetric, and therefore $\widehat{\Psi}$ is too.
\end{proof}

\subsection{Expression of the system in self-similar variables and rotating frame}
Since we expect the system to be rigidly rotating with the unknown angular speed $\alpha'(t)$, it is convenient to first reformulate \cref{eq:NS} in a rotating frame. Thus, with a slight abuse of notation we keep the same notation for $\omega(x,t)$ and $\omega(Q_{-N\alpha(t)}x,t)$ (with $Q_{-N\alpha(t)}$ defined as $Q_j$ in \cref{def:Qj} with $j=-N\alpha(t)$) and we obtain the Navier-Stokes equations on a rotating frame with angular speed $\alpha'$ which are given by 
\begin{equation}
    \label{eq:NSrot}
    \de_t\omega+\big\{\psi-\frac{\alpha'}{2}|x|^2,\omega\big\}=\nu\Delta\omega.
\end{equation}
We will refer to the extra term involving $\alpha'$ as the \textit{Coriolis term}.
\begin{remark}
    Note that the identity in Lemma \ref{lem:NSmom} is not affected by the presence of the Coriolis term, since this term does not contribute to the evolution of $\mathrm{L}(\omega(t))$ because $\langle \{\psi,|x|^2\},|x|^2\rangle=0$.
\end{remark}
To write down the system satisfied by $\Omega,\widehat{\Omega}$, we need to introduce some operators: 
\begin{itemize}
    \item The (rescaled) diffusion operator 
    \begin{equation}
    \label{def:cL}
    \cL:=\Delta_\xi+\frac12 \xi \cdot \nabla_\xi +1.
\end{equation}
\item The translation operator 
\begin{equation}
\label{def:Teps}
(\cT_{\eps(t),Z(t),d} f)(\xi)\coloneqq(\cT_{\eps,Z} f)(\xi)=f\Big(\xi+\frac{1}{\eps(t) d}Z(t)\Big)
\end{equation}

\item The \emph{polygon-center} operator
\begin{equation}
\label{def:Veps}
    (\cV_{\eps(t),Z(t),d,N} f)(\xi)\coloneqq(\cV_{\eps,Z} f)(\xi)=\sum_{l=1}^Nf\Big(Q_{l}\xi-\frac{1}{\eps(t) d}Z(t)\Big)
\end{equation}
\item The \emph{polygon-polygon} operator 
\begin{equation}
\label{def:Qeps}
    (\cQ_{\eps(t),Z(t),d,N} f)(\xi)\coloneqq(\cQ_{\eps,Z} f)(\xi)=\sum_{l=1}^{N-1}f\Big(Q_l\xi+\frac{1}{\eps(t) d}(Q_l-I)Z(t)\Big).
\end{equation}
\end{itemize}
We then have the following.
\begin{lemma}
\label{lem:derivationequation}
    The vorticity $\omega$ defined in \cref{eq:omSS} satisfies \cref{eq:NSrot} if $\Omega,\widehat{\Omega}$ solve the following system
\begin{align}
    \label{eq:Omega}&t\de_t \Omega=\cL\Omega-\frac{1}{\delta}\Big\{(I+\cQ_{\eps,Z})\Psi+\gamma \cT_{\eps,Z}\widehat{\Psi}-\frac{\delta t}{ \eps d}\xi^\perp\cdot Z'- \frac{\delta t}{2}\alpha'\big|\xi+\frac{1}{\eps d}Z(t)\big|^2 \, , \, \Omega\Big\}\\
    \label{eq:P}&t\de_t \widehat{\Omega}=\cL \widehat{\Omega}-\frac{1}{\delta}\Big\{\gamma \widehat{\Psi}+\cV_{\eps,Z} \Psi-\frac{\delta t}{2} \alpha'|\xi|^2 \, , \, \widehat{\Omega}\Big\} \\
    \label{eq:Init}& \Omega(\xi,0) = \widehat{\Omega}(\xi,0) = \frac{1}{4\pi} e^{-|\xi|^2/4}, \quad \forall \xi \in \R^2.
\end{align}
\end{lemma}
\begin{proof}
We first compute the terms associated to the heat-equation part in \cref{eq:NSrot}. Hence, we observe that
\begin{align}
    \de_t\omega(x,t)-\nu\Delta\omega(x,t)=\, &\frac{\Gamma}{\nu t^2}\sum_{j=1}^N\big(t\de_t \Omega-\cL \Omega-\sqrt{\frac{t}{\nu}} Z'\cdot \nabla \Omega\big)(\xi_j,t)\\
    &+\frac{\gamma \Gamma}{\nu t^2}\big(t\de_t \widehat{\Omega}-\cL \widehat{\Omega}\big)(\widehat{\xi},t).
\end{align}
For the Poisson bracket terms, we compute using relation~\cref{eq:omSS} that
\begin{equation}
    \{|\cdot|^2,\omega\}_x = \frac{\Gamma}{\nu t} \sum_{j= 1}^N \{ |\cdot|^2,\Omega(\xi_j(\cdot),t)\}_x + \frac{\gamma\Gamma}{\nu t}\{ |\cdot|^2,\widehat{\Omega}(\widehat{\xi}(\cdot),t)\}_x,
\end{equation}
where we interpreted the $\xi_j$ as functions of $x$. We now want to remove the map composition in the Poisson bracket, namely express it in terms of the maps $\Omega$ and $\widehat{\Omega}$ with their natural variable $\xi$, applied in $\xi_j$ and $\widehat{\xi}$. To this end, we express $|x|^2$ in terms of $\xi_j$ using relation~\cref{def:xi_j} to get that
\begin{equation}\label{eq:x_in_xi_j}
    x=\sqrt{\nu t}Q_j( \xi_j+\frac{1}{\eps d}Z(t)), \qquad \text{for } j=1,\dots,N.
\end{equation}
Then, to complete this change of variable, we compute using relations~\cref{def:xi0} and~\cref{def:xi_j}  that for $0\le j \le N$,
\begin{equation}
    \{(\xi_j)_1,(\xi_j)_2\}_x=\frac{1}{\nu t},
\end{equation}
and thus, recalling the well-known identity for changing variables inside a Poisson bracket
\begin{equation}\label{eq:change_var_Poisson}
    \{ f(\xi(\cdot)),g(\xi(\cdot)) \}_x(x) = \{ \xi_1,\xi_2 \}_x(x) \{f,g\}_\xi(\xi(x)),
\end{equation}
we obtain that
\begin{align}
    \frac{\alpha'}{2}\{|\cdot|^2,\omega\}_x = \, & \frac{\Gamma\alpha'(t)}{2\nu t}\sum_{j=1}^N\big\{\big| Q_j\big(\cdot + \frac{1}{\eps d}Z(t)\big)\big|^2,\Omega(\cdot,t)\big\}_\xi(\xi_j)\\ 
&+\frac{\gamma\Gamma\alpha'(t)}{2\nu t}\big\{|\cdot|^2,\widehat{\Omega}(\cdot,t)\big\}_{\xi} (\widehat{\xi}).
\end{align}
Similarly, keeping the notations compact, we have that 
\begin{align}
    \{\psi,\omega\}_x(x,t)=\,&\frac{\Gamma^2}{\nu^2 t^2}\sum_{j=1}^{N}\bigg\{\sum_{l=1}^N\Psi(\xi_l),\Omega\bigg\}_\xi(\xi_j,t)+
    \frac{\gamma^2\Gamma^2}{\nu^2 t^2}\{\widehat{\Psi},\widehat{\Omega}\}_\xi(\widehat{\xi},t)\\
    &+\frac{\gamma\Gamma^2}{\nu^2 t^2}\bigg\{\sum_{l=1}^N\Psi(\xi_l),\widehat{\Omega}\bigg\}_\xi(\widehat{\xi},t)+\frac{\gamma\Gamma^2}{\nu^2 t^2}\sum_{j=1}^N\{\widehat{\Psi}(\widehat{\xi}),\Omega\}_\xi(\xi_j,t).
\end{align}
One should remember that in this expression, $\xi_l$ is seen as a function of $\xi_j$ given by relation~\cref{eq:xi_l_in_xi_j}. In particular, this means that
\begin{align}
\sum_{l=1}^N\Psi(\xi_l,t) & =\Psi(\xi_j,t)+\sum_{\substack{l=1\\ l\neq j}}^N\Psi\Big(Q_{-l+j}\xi_j+\frac{1}{\sqrt{\nu t}}(Q_{-l+j}-I)Z(t),t\Big) \\
& = \Psi(\xi_j,t)+\sum_{l=1}^{N-1}\Psi\Big(Q_{-l}\xi_j+\frac{1}{\sqrt{\nu t}}(Q_{-l}-I)Z(t),t\Big)\\
& =((I+\cQ_{\eps,Z})\Psi)(\xi_j,t)
\end{align}
where we used the definition of the operator $\cQ_{\eps,Z}$ in \cref{def:Qeps} and simply reordered the terms in the second equality. Analogously,
\begin{align}
    &\sum_{l=1}^N\Psi(\xi_l,t) =   \sum_{l=1}^N\Psi\bigg(Q_{-l}\widehat{\xi}-\frac{1}{\sqrt{\nu t}}Z(t),t\bigg)=(\cV_{\eps,Z}\Psi)(\widehat{\xi},t),\\
    &\widehat{\Psi}(\widehat{\xi},t)=   \widehat{\Psi}\bigg(Q_j\xi_j+\frac{1}{\sqrt{\nu t}}Q_jZ(t),t\bigg)= \widehat{\Psi}\bigg(\xi_j+\frac{1}{\sqrt{\nu t}}Z(t),t\bigg)=(\cT_{\eps,Z}\widehat{\Psi})(\xi_j,t),
\end{align}
where we used that $\widehat{\Psi}$ is $N$-fold symmetric. Collecting all the identities above, since $\sqrt{t/\nu}=t/(\eps d)$ and $Z'\cdot \nabla \Omega = \{\xi^\perp \cdot Z , \Omega\}_\xi$, we find that we solve \cref{eq:NS} if
\begin{align}
   &\sum_{j=1}^N\bigg(t\de_t\Omega-\cL\Omega+\frac{1}{\delta}\Big\{(I+\cQ_{\eps,Z})\Psi+\gamma \cT_{\eps,Z}\widehat{\Psi}-\frac{\delta t}{ \eps d}\xi^\perp\cdot Z'- \frac{\delta t}{2}\alpha'|\xi+\frac{1}{\eps d}Z(t)|^2, \Omega\Big\}_\xi\bigg)(\xi_j,t)\\
   &\, +\gamma\bigg(t\de_t\widehat{\Omega}-\cL \widehat{\Omega}+\frac{1}{\delta}\Big\{\gamma \widehat{\Psi}+\cV_{\eps,Z} \Psi-\frac{\delta t}{2} \alpha'|\xi|^2, \widehat{\Omega}\Big\}_\xi\bigg)(\widehat{\xi},t)=0.
\end{align}
Therefore, when $\Omega,\widehat{\Omega}$ satisfy \cref{eq:Omega}-\cref{eq:P} we clearly have the identity above.
To assign an initial data to the equations \cref{eq:Omega}-\cref{eq:P}, we argue as in \cite{dolce2024long}. By the uniqueness result in \cite{GaGa2005} we deduce that 
\begin{equation}
\label{def:G}
\Omega_0(\xi)=\widehat{\Omega}_0(\xi)=G(\xi), \qquad \text{ with } \qquad G(\xi)=\frac{1}{4\pi}\e^{-|\xi|^2/4}.
\end{equation}
\end{proof}
We  introduce the notation for the initial velocity and streamfunction as
\begin{equation}\label{def:UG}
  U_0(\xi) \, \,:=\, \frac{1}{2\pi}
  \frac{\xi^\perp}{|\xi|^2}\,\Bigl(1-e^{-|\xi|^2/4}\Bigr)\,, \qquad
  \Psi_0(\xi) \, := \,\frac{1}{4\pi}\big(\Ein\bigl(|\xi|^2/4\bigr) -
  \gamma_E\big)\,,
\end{equation}
where $\gamma_E\approx 0,5772$ is the Euler-Mascheroni constant and ${\rm Ein}(x)=\int_0^xs^{-1}(1-\e^{-s}) \dd s$ is the exponential integral function. In the sequel, we will need to use also the function 
\begin{equation}
    \label{def:A}
    A(\xi)=-\frac{\nabla \Psi_0(\xi)}{\nabla \Omega_0(\xi)}=\frac{4}{|\xi|^2}(\e^{|\xi|^2/4}-1).
\end{equation}

\subsection{Choice of position and angular speed}
In the system~\cref{eq:Omega}-\cref{eq:P}, and already in the definition of the self-similar variables, we introduced $Z$ and $\alpha$ as given quantities to be chosen later. 
The key point in the choice of these modulation parameters is to ensure that the center of vorticity of the blobs remains centered in our new coordinates, namely that
\begin{equation}
    \mrm^1(\Omega(t))=\mrm^1(\widehat{\Omega}(t))=(0,0), \qquad \text{for all } t\geq0.
\end{equation}
This specific choice is necessary for technical needs, but for very concentrated vortices, the center of vorticity is also a good estimation of the location of most of the mass of the vortex\footnote{In fact, to determine the position of the vortices, one could decide to follow the maximum of vorticity or the stagnation point. None of these choices are equivalent in general and we prefer to fix the center of vorticity blobs for later convenience.}. Recall that for the central vortex $\widehat{\Omega}$, the condition above is automatically guaranteed from the $N$-fold symmetry with $N\geq2$, see \cref{lem:centre0}. 

We are thus left with the freedom of choosing three parameters $Z_1(t),Z_2(t),\alpha(t)$ to fix only two conditions on $\Omega$. This is simply because those three variables are linked by a triangular relation. Therefore, we choose to set
\begin{equation}
    Z_2(t)=0 \qquad \text{for all } t\geq0,
\end{equation}
which we can do without loss of generality by adjusting the choice of $\alpha$.
The quantity $Z_1(t)$ is now  the distance to the center of the $N$-th vortex, meaning that it is the radius of our regular polygon and we thus call it
\begin{align}
\label{eq:choiceZ1}
    Z_1(t)=R(t), \qquad R(0)=\mathsf{r}.
\end{align}
Then, in analogy with \cite{dolce2024long,GaSring}, the choice that fixes the center of mass is given in the following lemma.
\begin{lemma}\label{lem:alpha'gives_moments}
    Let $\Omega,\widehat{\Omega}$ be the solutions to \cref{eq:Omega,eq:P,eq:Init}. Define the radius of the vortex polygon and the angular speed of the solution to be
    \begin{align}
    \label{def:R} tR'(t)&\coloneqq-\frac{d\eps(t)}{\delta }\int_{\RR^2}\big(\de_2(\cQ_{\eps,(R,0)}\Psi+\gamma \cT_{\eps,(R,0)}\widehat{\Psi})\Omega\big)(\xi,t)\dd \xi,\\
    \label{def:alpha}
t\alpha'(t)&\coloneqq\frac{d\eps(t)}{\delta R(t)}\int_{\RR^2}\big(\de_1(\cQ_{\eps,(R,0)}\Psi+\gamma \cT_{\eps,(R,0)}\widehat{\Psi})\Omega\big)(\xi,t)\dd \xi,
    \end{align}
    with $\alpha(0)=0$ and  $R(0)=\mathsf{r}$.
    Then, for all $t\geq 0$ the following moment identities hold true:
    \begin{align}
&\mrM(\Omega(t))=\mrM(\widehat{\Omega}(t))=1, \quad \mrm^1(\Omega(t))=\mrm^1(\widehat{\Omega}(t))=(0,0).
\end{align}
\end{lemma}
\begin{proof}
    The proof for the mean is straightforward. In the sequel,  we need the Poisson bracket identities $\{f,g\}=-\{g,f\}$ and
    \begin{equation} 
    \langle \{f,g\},h\rangle =\langle f,\{g,h\}\rangle.
    \end{equation}
    Moreover, we have the following identity involving the operator $\cL$
    \begin{align}
    \label{eq:Lmom}
    \langle \cL (f), \xi_j\rangle=-\frac12\mrm_j^1(f).
    \end{align}
    Since $$\de_1f=\{f,\xi_{2}\}, \quad \de_2f=-\{f,\xi_1\} $$ 
    and $\de_i$ commute with  $\Delta^{-1}$, it is not hard to verify also that 
    \begin{align}
    \label{eq:Pbracket}
        \langle \{\Delta^{-1}f,f\},\xi_i\rangle =0, \qquad \text{ for } i=1,2.
    \end{align}
    Then, since  $\mrM(\Omega)=1$ and $Z(t)=(R(t),0)$, by testing \cref{eq:Omega} with $\xi_1$ and $\xi_2$ we deduce that 
    \begin{align}
    \label{eq:m21}
                &t\de_t \mrm_1^1(\Omega)=-\frac12\mrm_1^1(\Omega)
            -\frac{1}{\delta}\int_{\RR^2}\de_2(\cQ_{\eps,Z} \Psi+\gamma \cT_{\eps,Z}\widehat{\Psi})\Omega \, \dd \xi
            - \frac{tR'(t)}{\eps d}+t\alpha'(t)\mrm_2^1(\Omega),
\\
            &t\de_t \mrm_2^1(\Omega)=-\frac12\mrm_2^1(\Omega)
            +\frac{1}{\delta}\int_{\RR^2}\de_1(\cQ_{\eps,Z} \Psi+\gamma \cT_{\eps,Z}\widehat{\Psi})\Omega \, \dd \xi
            - t\alpha'(t)\big(\mrm_1^1(\Omega)+\frac{R}{\eps d }\big).
    \end{align}
    By the choice of $R',\alpha'$ in \cref{def:R,def:alpha}, the equation above simplifies to 
        \begin{align}
        \label{eq:simpm12}
            t\de_t \mrm^1(\Omega)=-\frac12\mrm^1(\Omega)
        -t\alpha'(t)(\mrm^1(\Omega))^\perp.
    \end{align}
    The identity above readily implies that 
        \begin{align}
        t\de_t (t|\mrm^1(\Omega)|^2)=0\qquad \Longrightarrow   \qquad |\mrm^1(\Omega(t))|^2=0, \text{ for all } t>0.
    \end{align}
    Since clearly $\mrm^1(\Omega(0))=\mrm^1(G)=(0,0)$ the proof of the lemma is over.
\end{proof}
\begin{remark}
\label{rem:commentRalpha}

The choice of the modulation parameters $R$ and $\alpha$ as in Lemma \ref{lem:alpha'gives_moments} is fundamental in the construction of the approximate solution. Indeed, we avoid the iterative construction in \cite{dolce2024long,zhang2025long} by showing that the choices in \cref{def:R,def:alpha} can be used to fix the first-order moments of the approximate solution. However, in the full nonlinear problem, we do not choose the radius as in \cref{def:R}, but instead fix it to be the one obtained from the approximate solution. This is because the operators depend in a complicated way on the radius, and a fully nonlinear choice as in \cref{def:R} introduces several technical difficulties (an issue also present in \cite{zhang2025long}). Thus, we will drop the requirement that $\mrm^1_1(\Omega(t))=0$. From the proof of Lemma \ref{lem:alpha'gives_moments}, it is straightforward to see that one can still fix $\mrm_2^1(\Omega(t))=0$ with the choice
    \begin{align}
    t\alpha'(t)&\coloneqq\frac{d\eps(t)}{\delta (R+\eps d \mrm^1_1(\Omega))(t)}\int_{\RR^2}\big(\de_1(\cQ_{\eps,(R,0)}\Psi+\gamma \cT_{\eps,(R,0)}\widehat{\Psi})\Omega\big)(\xi,t)\dd \xi,
    \end{align}
    for any $R(t)$. On the other hand, we now have to guarantee that $\mrm^1_1(\Omega(t))$ remains sufficiently small on the time scale of interest. To this end, we fully exploit a property of the approximate solution related to the evolution of the angular momentum in Lemma \ref{lem:NSmom} (see Lemma \ref{lem:keyLEuler}).
    \end{remark}
\subsection{Equations in reduced form}
In view of the previous sections, we introduce the operator
\begin{equation}\label{def:cA}
    \cB_{\eps,Z}(\varphi,\widehat{\varphi},\widetilde{\varphi}) := \int_{\R^2} \nabla^\perp(\cQ_{\eps,Z}\varphi+\gamma\cT_{\eps,Z}\widehat{\varphi})\Delta\widetilde{\varphi} \dd \xi,
\end{equation}
and by abuse of notation, define
\begin{equation}
 \cB_{\eps,Z}(\varphi,\widehat{\varphi}) := \mathcal{B}_{\eps,Z}(\varphi,\widehat{\varphi},\varphi).
\end{equation}

We make once and for all the choice that the parameter $Z(t)$ of the change of variable~\cref{def:xi} is given by $$Z(t) = (R(t),0), $$ and we identify $(R(t),0) \simeq R$ for notational convenience. 

Then, we introduce
\begin{align}
    \label{def:Xieps}&\Phi_{\eps,R}(\Psi,\widehat{\Psi}) =\Psi+\cQ_{\eps,R}\Psi+\gamma\cT_{\eps,R} \widehat{\Psi}+\xi_2\Re(\cB_{\eps,R}(\Psi,\widehat{\Psi}))-\frac{d \eps}{2 R}\Im(\cB_{\eps,R}(\Psi,\widehat{\Psi}))\big|\cT_{\eps,R}\xi\big|^2\\
\label{def:Thetaeps}&\widehat{\Phi}_{\eps,R}(\Psi,\widehat{\Psi}) =\gamma\widehat{\Psi}+\cV_{\eps,R} \Psi-\frac{d \eps}{2 R}\Im(\cB_{\eps,Z}(\Psi,\widehat{\Psi}))|\xi|^2.
\end{align}
When the parameters $R,\alpha$ from the rotating frame are defined by relations~\cref{def:R,def:alpha}, observe that
\begin{equation}
    tR'=\frac{\eps d}{\delta } \Re(\cB_{\eps,R}(\Psi,\widehat{\Psi})), \qquad t \alpha'(t) = \frac{\eps d}{\delta R} \Im(\cB_{\eps,R}(\Psi,\widehat{\Psi}))
\end{equation}
with $\Psi$ and $\widehat{\Psi}$ the streamfunctions of the solutions to equations~\cref{eq:Omega,eq:P,eq:Init}.
Equipped with all the previous notations, and with this specific choice of parameters, we obtain that $\Omega$ and $\widehat{\Omega}$ are solutions of the system~\cref{eq:Omega,eq:P} if and only if
\begin{equation}
\label{eq:OmPhiR}
    t\de_t\Omega=\cL\Omega-\frac{1}{\delta}\Big\{\Phi_{\eps,R}(\Psi,\widehat{\Psi}),\Omega \Big\},\qquad t\de_t\widehat{\Omega}=\cL \widehat{\Omega}-\frac{1}{\delta}\Big\{\widehat{\Phi}_{\eps,R}(\Psi,\widehat{\Psi}),\widehat{\Omega}\Big\},
\end{equation}
and in that case, by Lemma~\ref{lem:alpha'gives_moments}, we have that $\mrM(\Omega(t))=\mrM(\widehat{\Omega}(t))=1$ and  $\mrm^1(\Omega(t))=\mrm^1(\widehat{\Omega}(t))=(0,0)$.
Moreover, when there will be no ambiguity, we even use the notation $\Phi = \Phi_{\eps,R}(\Psi,\widehat{\Psi})$ and $\widehat{\Phi} = \widehat{\Phi}_{\eps,R}(\Psi,\widehat{\Psi})$, so that we arrive to the compact form
\begin{equation}
\label{eq:simplified}
t\de_t\Omega=\cL\Omega-\frac{1}{\delta}\{\Phi,\Omega\},\qquad t\de_t\widehat{\Omega}=\cL \widehat{\Omega}-\frac{1}{\delta}\{\widehat{\Phi},\widehat{\Omega}\},
\end{equation}
for the system~\cref{eq:Omega,eq:P}. The use of such similar notation in the equations for the vortex at one vertex $\Omega$ and the central vortex $\widehat{\Omega}$ is made to stress the structural analogies. Indeed, the construction of the approximate solution follows a rather general asymptotic expansion scheme, in which the differences are all encoded in the different expansions for $\Phi$ and $\widehat{\Phi}$. 
Therefore, we will often present detailed computations for $\Omega$, which satisfies the most complicated equation, while only highlighting the interesting differences for $\widehat{\Omega}$ when needed. One has to remember that equations~\cref{eq:simplified} are coupled since $\Phi$ depends on $\widehat{\Omega}$ and vice versa.

\section{Function spaces, operators and their asymptotic expansion}
\label{sec:Func}In this section, following \cite{GalWay2005,Gallay2011,dolce2024long}, we introduce the function spaces needed to construct the approximate solution, and state some relevant properties of the operators introduced above. As some facts are either well-known, or simply coming from explicit computations, some details and proofs are delayed to Appendix~\ref{app:functional}.

\subsection{Function spaces}

We first define the weighted $L^2$ spaces 
\begin{align}
\label{def:cY}
  \cY \,&=\, \biggl\{f\in L^2(\RR^2)\,:\, \int_{\RR^2} |f(\xi)|^2
  \,\e^{|\xi|^2/4}\,\dd \xi < \infty\biggr\}\,,\\
  \label{def:calZ}
  \cZ \,&=\, \bigl\{f:\RR^2\to \RR \,:\, \xi \mapsto \e^{|\xi|^2/4}f(\xi) \in
  \cS_*(\RR^2)\bigr\}\,,
\end{align}
where $\cS_*(\RR^2)$ denotes the space of smooth functions with at most
polynomial growth at infinity. Note that $\cZ$ is a dense subset of $\cY$. The space $\cY$ is a Hilbert spaces equipped with the inner product 
\begin{equation}
  \langle{f,g}\rangle_{\cY} \,=\, \int_{\RR^2}f(\xi)\,g(\xi)\,\e^{|\xi|^2/4}\,\dd \xi\,, \qquad
  \forall\,f,g \in \cY.
\end{equation}
In view of the Fourier transform in the angular variable in polar coordinates $\xi=(r\cos(\theta),r\sin(\theta))$,  we also have the orthogonal decomposition 
\begin{equation}\label{Ydecomp}
  \cY \,= \,\bigoplus_{n=0}^{\infty}\,\cY_n\,,
\end{equation}
where $\cY_n = \bigl\{f\in \cY\,:\, f(\xi) = a(r)\cos(n\theta) + b(r)\sin(n\theta) \text{ with }  a,b:
\RR_+\to\RR\bigr\}$. We also denote with $\cP_n$ the $L^2$-projection operator onto the $n$-th angular Fourier mode.

Then, we recall the $\cO_{\cS_*}$ and $\cO_\cZ$ notation introduced in \cite[Definition 3.4]{dolce2024long}.
\begin{definition}\label{def:topo}
Let $M \in \NN$ be a positive integer. 

\smallskip\noindent 1)
If $g_\eps \in \cS_*(\RR^2)$ depends on a small parameter $\eps > 0$,
we say that $g_\eps \,=\, \cO_{\cS_*}\bigl(\eps^M\bigr)$ if, for all $\alpha \in \NN^2 $ and all $m\in \{0,\dots, M\}$ there exists $C>0$ and $\widetilde{N}\in \NN$ such that
\[
  |\partial^{\alpha}_\xi\de_\eps^mg_\eps(\xi)| \le C(1+|\xi|)^{\widetilde{N}}\eps^{M-m}
  \quad \forall\,\xi \in \RR^2\,.
\]

\noindent 2) Similarly, if $f_\eps \in \cZ$, we write $f_\eps = \cO_{\cZ}
\bigl(\eps^M\bigr)$ if $\e^{|\xi|^2/4}f_\eps = \cO_{\cS_*}\bigl(\eps^M\bigr)$. 
\end{definition}
\begin{remark}
    In the sequel, we will not explicitly verify the conditions on the derivatives in Definition \ref{def:topo}. Indeed, in the asymptotic expansions  of the operators proved in Appendix \ref{app:functional}, we either have errors involving tails of explicit integrals or errors arising by truncating a Taylor series. In both cases, it is straightforward to check the conditions for the derivatives in \cref{def:topo}.
\end{remark}
\subsection{Diffusion and advection operators}
We now recall properties of the difussion operator $\cL$ in \cref{def:cL} and the \emph{advection} operator arising from the linearization of the Euler equations at the Lamb-Oseen vortex.
\subsubsection*{{\rm i)} \textbf{The diffusion operator}} The diffusion operator $\cL$ defined by \cref{def:cL} is
a linear operator in $\cY$ with (maximal) domain
\begin{equation}\label{def:cLdom}
  D(\cL) \,=\, \bigl\{f\in \cY \, : \,\Delta f\in \cY, \, 
  \xi\cdot \nabla f \in \cY\bigr\}\,.
\end{equation}
It is well known, see \cite[Appendix~A]{GalWay2002}, that $\cL$ is {\em self-adjoint} in
$\cY$ with compact resolvent and purely discrete spectrum   $\sigma(\cL) \,=\, \{-n/2 \, : \, n\in \NN\}\,$. We have $\mathrm{Ker}(\cL)=\mathrm{span}(G)$ where $G$ is the Gaussian function defined in \cref{def:G}. The eigenspace corresponding to the eigenvalue $-1/2$ is $\mathrm{span}(\de_1G,\de_2G)$. Moreover, it is straightforward to verify that if $f \in \cY_n \cap D(\cL)$, then $\cL f \in \cY_n$. We also recall the result \cite[Lemma 3.5]{dolce2024long}, telling us the following.
\begin{lemma}\label{lem:Linvert}
For any $\kappa > 0$ and any $f \in \cZ$ one has $(\kappa - \cL)^{-1}f
\in \cZ$.   
\end{lemma}

\subsubsection*{{\rm ii)} 
\textbf{The advection operator}}
We also need the operator associated to the linearization of the $2D$ Euler equation at the Lamb-Oseen vortex. This is denoted by  $\Lambda:D(\Lambda)\to \cY$, and is defined as 
\begin{equation}\label{def:Lambda2}
  \Lambda f \,=\, \bigl\{\Psi_0\,,\,f\bigr\} + \bigl\{\Delta^{-1}f\,,\,\Omega_0\}\,,
  \qquad f \in D(\Lambda)\,, 
\end{equation}
where $\Omega_0,\Psi_0$ are  defined in \cref{def:G}, \cref{def:UG}. 
The operator $\Lambda$ is considered in its maximal domain
$  D(\Lambda) \,=\, \bigl\{f\in \cY \, : \, \{\Psi_0, f\} \in \cY\bigr\}\,$ and we also have that  $\Lambda$ is invariant under rotations
about the origin, so that it commutes with the  direct sum decomposition
\cref{Ydecomp}. Moreover, it is clear that $\Lambda f \in \cZ$ if $f \in \cZ$. In view of the definition of $A$ in \cref{def:A}, we also note that we can write 
\begin{align}
\label{eq:LambdaA}
    \Lambda f=\{(A+\Delta^{-1})f,\Omega_0\}.
\end{align}
We then recall the following properties, that are well known by now, e.g. \cite{Gallay2011}.
\begin{proposition}\label{prop:Lambda}
The following properties holds true:
\begin{enumerate}
    \item The operator $\Lambda$ is skew-adjoint in the Hilbert space $\cY$ with
kernel 
\begin{equation}\label{kerLam}
  \mathrm{Ker}(\Lambda) \,=\, \cY_0 \,\oplus\, \bigl\{\beta_1\partial_1 G +
  \beta_2\partial_2 G  \,:\, \beta_1, \beta_2 \in \RR\bigr\}\,.
\end{equation}
\item If $h \in \mathrm{Ker}(\Lambda)^\perp \cap \cZ$ or $h\in \cZ$ with $\cP_0(h)=\mrm^1_1(h)=\mrm^1_2(h)=0$ the equation $\Lambda f = h$
has a unique solution $f \in \mathrm{Ker}(\Lambda)^\perp \cap \cZ$, and $f$ is an
even (odd) function of the variable $\xi_2$ if $h$ is an odd (even) function of $\xi_2$.
\end{enumerate}
\end{proposition}
We remark that, in fact, to determine the leading order terms in the construction of the approximate solution one has to invert $\Lambda$ on functions of the type $h=\{\varphi,\Omega_0\}$ with $\varphi\in \cS_*$. Since $\Omega_0$ is radial, there is the degree of freedom of adding any radial function to $\varphi$, and we therefore fix $\cP_0\varphi=0$ to remove this issue. Then, thanks to the representation \cref{eq:LambdaA} and Proposition \ref{prop:Lambda}, we know that 
\begin{align}
\label{def:ADelta}
  \Lambda f=\{\varphi,\Omega_0\}\quad \iff \quad  (A+\Delta^{-1})f=\varphi.
\end{align}
Properties of the operator $A+\Delta^{-1}$ on the space 
\begin{align}
\label{def:XLambda}
    \cX=\Big\{ f\in L^2(\RR^2) \, :\, \int_{\RR^2}|f(\xi)|^2A(\xi)\, \dd \xi<\infty\Big\}
\end{align}
have been carefully studied in \cite[Section 2]{GaSarnold}. For our purposes, observe that $\cP_0(\{\varphi,\Omega_0\})=0$ and therefore 
\begin{align}
\label{eq:KerA}
    \{\varphi,\Omega_0\}\in \mathrm{Ker}(\Lambda)^\perp \quad \iff \quad \langle \varphi,\de_i \Omega_0\rangle=0 \quad \text{for }i=1,2,
\end{align}
namely $\varphi$ is orthogonal to $\de_i \Omega_0$ for $i=1,2$ with respect to the standard $L^2(\RR^2)$ inner product. Moreover, solving  the second equation in \cref{def:ADelta} in the space $\cX$ is  equivalent to the following problem in $L^2(\RR^2)$
\begin{align}
\label{eq:newLambda}
    \big({\rm Id}-A^{-\frac12}(-\Delta)^{-1}A^{-\frac12}\big)F=\phi, \quad \text{with } F=A^{\frac12}f \text{ and } \phi=A^{-\frac12}\varphi.
\end{align}
Having that $A$ is radially symmetric, we can as well introduce the angular Fourier decomposition, meaning that we can solve equation \cref{eq:newLambda} mode by mode. The problem is then reduced in proving a bound for the operator $T_n:=\cP_n(A^{-\frac12}(-\Delta)^{-1}A^{-\frac12})\cP_n : L^2(\RR^2)\to L^2(\RR^2)$, which can be represented as 
\begin{align}
\label{def:Tn}
 T_n[g]=\frac{1}{2n}\int_{0}^\infty \frac{\min\big(\frac{r}{s},\frac{s}{r}\big)^n}{\sqrt{A(r)A(s)}}g(s) s \, \dd s,    \quad \text{ with } g\in L^2(r\dd r).
\end{align}
In \cite{GaSarnold}, it was proved that there exists $0<c<1$ such that 
\begin{equation}
\label{bd:Tn}
    \|T_n\|_{X_n\to X_n}\leq \frac{c}{n}, \qquad X_n=\begin{cases}
        L^2(r\dd r) \qquad &n\geq 2\\
       \{g\in L^2(r\dd r)\, : \, \langle g, G' A^\frac12\rangle_{L^2(r\dd r)}=0\}\qquad &n=1.
    \end{cases}
\end{equation}
When $n=1$, the orthogonality condition in \cref{bd:Tn} is equivalent to \cref{eq:KerA}. This is  related to the fact that $T_1[G'A^{\frac12}]=G'A^{\frac12}$ \footnote{To prove this idendity, recall that in polar coordinates $A=-\Psi_0'/\Omega_0'$ with $(\de_{rr}+r^{-1}\de_r)\Psi_0=\Omega_0$. Thus $\Delta_1\Psi_0'=\Omega_0'$ and therefore $(-\Delta_1)^{-1}\Omega_0'=-\Psi_0'=A \Omega_0'$, which proves the desired identity after dividing by $A^{-\frac12}$.}, meaning that $G'A^\frac12\in \mathrm{Ker}({\rm Id}-T_1)$. Moreover,  one can deduce that $1$ is the largest eigenvalue of $T_1$ on $L^2(r\dd r)$ by Sturm-Liouville theory \cite{GaSarnold}, that is the key point to justify the bound \cref{bd:Tn} above. By taking the projection onto the $n$-th angular Fourier mode in \cref{eq:newLambda}, we see that we can define
\begin{align}
    \cP_n F=\big({\rm Id}-T_n\big)^{-1}\cP_n\phi=\sum_{j=0}^\infty T_{n}^j(\cP_n\phi).
\end{align}
With these observations, we have actually proved the following reinterpretation of Proposition \ref{prop:Lambda} to invert $\Lambda$ on some particular functions.
\begin{proposition}
\label{prop:explicitsol}
    Let $h=\{\varphi,G\}\in \mathrm{Ker}(\Lambda)^\perp\cap \cZ$ with $\varphi\in \cS_*$ and $\cP_0\varphi=0$. Then, the projections on the $n$-th angular Fourier mode of the unique solution to $\Lambda f=h$ can be written as
    \begin{equation}
        \cP_nf=A^{-\frac12}\sum_{j=0}^\infty T_n^j[A^{-\frac12}\cP_n\varphi], \qquad \text{for all } n\in \NN.
    \end{equation}
\end{proposition}

\subsection{Properties and expansions of the configurational operators} \label{sec:propExpAdj}

As explained in Section~\ref{sec:eq}, we are only interested in computing the operators $\cT_{\eps,Z}$, $\cV_{\eps,Z}$, $\cQ_{\eps,Z}$ in $Z = (R,0)$ so we simply write $\cT_{\eps,R},\cV_{\eps,R},\cQ_{\eps,R},\cB_{\eps,R}$ instead of $\cT_{\eps,(R,0)},\cV_{\eps,(R,0)},\cQ_{\eps,(R,0)},\cB_{\eps,(R,0)}$.

\subsubsection*{{\rm i)} \textbf{Expansions}}

We start by introducing the constants 
\begin{equation}
\label{def:cnk}
    c_{0,k}=(-1)^{k-1}, \qquad  c_{n,k}\coloneqq \begin{pmatrix}
    n\\ k
\end{pmatrix}(-1)^{n+k-1}\frac{ d^n}{2\pi n} \qquad \text{for } n\in \NN\setminus\{0\}, \, k\in \NN,
\end{equation}
and
\begin{equation}
\label{def:Snk}
    {\rm S}_{n,k}=\sum_{l=1}^{N-1}\frac{Q_l^k}{(Q_l-1)^n}, \qquad Q_l\leftrightarrow \e^{\frac{2\pi il}{N}}, \qquad 0\leq k\leq n.
\end{equation}
We provide in Appendix~\ref{app:constants} details about the values of the constants ${\rm S}_{n,k}$, and we only note that ${\rm S}_{n,k}\in \RR$ as proved in Lemma \ref{lem:Snk}.
The expansion of the operators is encoded in the following.
\begin{lemma}\label{lem:expansionsR}
    Assume that $f \in \cZ$, let $\varphi = \Delta^{-1}f$ and $M\in \NN$. Then, the translation operator admits the expansion
\begin{equation}\label{Texpansion}
  \cT_{\eps,R}\varphi(\xi) \,=\, \frac{{\rm M}(f)}{2\pi}\log\Big(\frac{R}{\eps d}\Big)
+\sum_{n=1}^M\sum_{k=0}^{n} \frac{\eps^n}{R^n}c_{n,k} \Re\Big(\xi^{n-k}{\rm m}^k(f)\Big) + \cO_{\cS_*}\bigl(\epsilon^{M+1}\bigr)\,.
\end{equation}
The polygon-center operator admits the expansion
\begin{equation}\label{Vexpansion}
  \begin{split}
      \cV_{\eps,R}\varphi(\xi) \,=\,& \mrM(f)\frac{N}{2\pi}\log\Big(\frac{R}{\eps d}\Big)
  +N\sum_{n=1}^M (-1)^n\frac{\eps^n}{R^n}c_{n,n} \Re\big({\rm m}^n(f)\big)\\
&+N\sum_{n=N}^M\sum_{\substack{\ell \in \mathbb{N}\setminus \{0\}\\
\ell N\leq M}}(-1)^n\frac{\eps^n}{R^n} c_{n,n-\ell N} \Re\Big(\xi^{\ell N}\mrm^{n-\ell N}(f)\Big)+\cO_{\cS_*}\bigl(\epsilon^{M+1}\bigr)\,.
        \end{split}
\end{equation}
The polygon-polygon operator admits the expansion

\begin{equation}\label{Qexpansion}
  \begin{split}
\cQ_{\eps,R}\varphi(\xi) \,=&\, \frac{\mrM(f)}{2\pi}\sum_{l=1}^{N-1}\log\Big(\frac{R}{\eps d}\Big)
  +\sum_{n=1}^M \sum_{k=0}^n\frac{\epsilon^n}{R^n}c_{n,k}{\rm S}_{n,n-k}\Re\Big(\xi^{n-k}\mrm^k(f)\Big)+\cO_{\cS_*}(\eps^{M+1}).
        \end{split}
\end{equation}
\end{lemma}
These expansions are proven in Appendix~\ref{app:functional}, by applying Lemma~\ref{lem:Expansions} in the case $Z=(R,0)$.
Let us now gather some direct consequences of these expansions that will be of use later.
\begin{proposition}\label{prop:small_operator}
    For any $\varphi \in \Delta^{-1}\cZ$, the gradient of the operators is always at least of order $\eps$:
        \begin{equation}\label{eq:small_gradients}
            \nabla Q_{\eps,R} \varphi = \cO_{\cS_*}(\eps), \qquad \nabla \cT_{\eps,R} \varphi = \cO_{\cS_*}(\eps), \qquad \nabla \cV_{\eps,R} \varphi = \cO_{\cS_*}(\eps),
        \end{equation}
    meaning in particular that for any $\varphi,\widehat{\varphi} \in \Delta^{-1}\cZ$,
        \begin{equation}\label{eq:smallA}
            \cB_{\eps,R}(\varphi,\widehat{\varphi}) = \cO(\eps),
        \end{equation}
    and thus that for any $f \in \cZ$,
    \begin{equation}\label{eq:devPhi}
        \big\{\Phi_{\eps,R}(\varphi,\widehat{\varphi}),f \big\} = \{ \varphi, f\} + \cO_\cZ(\eps).
    \end{equation}
    Moreover, let $X_\eps,Y_\eps \in \mathbb{C}$ be such that $|X_\eps|,|Y_\eps|\geq c>0$ for a constant $c$ independent of $\eps$ and $|X_\eps - Y_\eps| = \cO(\eps^M)$ for some $M \in \NN$. Then
        \begin{equation}
        \label{eq:TXY}
            (\cT_{\eps,X_\eps}-\cT_{\eps,Y_\eps})\varphi = \cO_{\cS_*}(\eps^{M+1}).
        \end{equation}
        The same holds true with $\cT_{\eps,\cdot}$ replaced by $\cV_{\eps,\cdot}$ and $\cQ_{\eps,\cdot}$
\end{proposition}
\begin{proof}
    Taking the gradient in the expansions of Lemma~\ref{lem:expansionsR} immediately gives~\cref{eq:small_gradients} and thus~\cref{eq:smallA}. Applying those relations to the definition of $\Phi_{\eps,R}$ gives in turn relation~\cref{eq:devPhi}. Instead, the relation \cref{eq:TXY} follows by comparing term by term the expansions in Lemma \ref{lem:expansionsR}, where we gain at least a factor of $\eps|X_{\eps}-Y_{\eps}|$ since the constant term vanishes and $|X_{\eps}|,|Y_{\eps}|\geq c>0$.
\end{proof}
As a consequence of Proposition \ref{prop:small_operator}, we have the following.
\begin{corollary}\label{lem:expA}
    Let $\varphi_1,\widehat{\varphi}_1,\varphi_2,\widehat{\varphi}_2\in \Delta^{-1}\cZ$ and let $n \in \NN$. Then
    \begin{equation}
        \cB_{\eps,R}(\varphi_1 + \eps^n \varphi_2 , \widehat{\varphi}_1 + \eps^n \widehat{\varphi}_2) = \cB_{\eps,R}(\varphi_1,\widehat{\varphi}_1) + \cO(\eps^{n+1}).
    \end{equation}
\end{corollary}
\begin{proof}
    It directly follows by the definition of $\cB_{\eps,R}$ in \cref{def:cA} and Proposition \ref{prop:small_operator}.
\end{proof}

\subsubsection*{{\rm ii)} \textbf{
Symmetries, commutators and adjoint}}
We finally characterize useful properties of our operators. \begin{lemma}\label{lem:adjoints}
    Let  $\cT_{\eps,R}^*$, $\cV_{\eps,R}^*$ and $\cQ_{\eps,R}^*$ be the $L^2(\RR^2)$-adjoint  of $\cT_{\eps,R}$, $\cV_{\eps,R}$ and $\cQ_{\eps,R}$ respectively.

Then the following holds.
\begin{enumerate}
    \item $\cT_{\eps,R}^*=\cT_{-{\eps,R}}$ and $[\cT_{\eps,R},\Delta]=[\cT^*_{\eps,R},\Delta]=0$. Moreover, if $f$ is even in $\xi_2$ then $\cT_{\eps,R}(f)$ is also even in $\xi_2$.
    \item The adjoint $\cV_{\eps,R}^*$ is given by
        \begin{equation}
            \label{def:Veps*}
            (\cV_{\eps,R}^*f)(\xi)=\sum_{l=1}^Nf\big(Q_l(\xi+\frac{(R,0)}{\eps d})\big), 
        \end{equation}
        and satisfies $[\cV_{\eps,R},\Delta]=[\cV_{\eps,R}^*,\Delta]=0$.
        Moreover, if $f$ is $N$-fold symmetric, then $$\cV_{\eps,R}^* f=N\cT_{\eps,R} f.$$
    \item The operator $\cQ_{\eps,R}$ is self-adjoint, namely $\cQ_{\eps,R}^*=\cQ_{\eps,R}$, and $[\cQ_{\eps,R},\Delta]=0.$ Moreover, if $f$ is even in $\xi_2$ then $\cQ_{\eps,R}(f)$ is also even in $\xi_2$.
    \item The following identities holds true
    \begin{align}
    \label{key:Q}
        &\cQ_{\eps,R}\{|\cT_{\eps,R}\xi|^2,f\}=\{|\cT_{\eps,R}\xi|^2,\cQ_{\eps,R}f\},\\
                \label{key:V}&\cV_{\eps,R}\{|\cT_{\eps,R}\xi|^2,f\}=\{|\xi|^2,\cV_{\eps,R}f\},\\
        \label{key:V*}&\cV^*_{\eps,R}\{|\xi|^2,f\}=\{|\cT_{\eps,R}\xi|^2,\cV^*_{\eps,R}f\}.\\
    \end{align}
\end{enumerate}
\end{lemma}
\begin{remark}
The identities in \cref{key:Q,key:V*,key:V} reflect the invariance of the vortex polygon with respect to global rotations. In the shifted frame used for the exterior vortices, the function $|\cT_{\eps,R}\xi|^2$ acts as the generating function of rotations around the global origin. Consequently, the first identity states that the polygon-polygon operator $\cQ_{\eps,R}$, describing the interactions between the vortices at the vertices of the polygon, commutes with the generator of global rotations. Analogously, the second and third identities show that the operators $\cV_{\eps,R}$ and $\cV_{\eps,R}^*$ intertwine the rotation of the central vortex with the global rotation of the exterior vortices.
\end{remark}
The proof of Lemma \ref{lem:adjoints} follows by direct computations, and we thus postpone it to Appendix~\ref{app:proofs_adjoints}. Finally, we record a key property of the operator $\cB_{\eps,R}(\cdot,\cdot)$ defined in \cref{def:cA}.
\begin{lemma}
\label{lem:eveneven}
    If $\varphi,\widehat{\varphi}$ are even in $\xi_2$, then 
    \begin{align}
    \label{eq:Reeven}
        \Re(\cB_{\eps,R}(\varphi,\widehat{\varphi}))=0.
    \end{align}
    In general, for any   $\varphi,\widehat{\varphi},\widetilde{\varphi}\in \cS_*$, denoting $w=\Delta\varphi, \widehat{w}=\Delta\widehat{\varphi} $ and $\widetilde{w}=\Delta\widetilde{\varphi}$, we have the expansion       \begin{equation}
    \label{eq:expansionBeps}
    \cB_{\eps,R}(\varphi,\widehat{\varphi},\widetilde{\varphi}) = i\sum_{n=1}^M\left(\frac{\eps}{R}\right)^n\sum_{k=0}^{n-1}(n-k)c_{n,k}\left({\rm S}_{n,n-k}\overline{\mrm^k(w)} + \gamma\overline{\mrm^k(\widehat{w})}\right)\overline{\mrm^{n-k-1}(\widetilde{w})}+\cO(\eps^{M+1}).
\end{equation}

\end{lemma}
The identity in \cref{eq:Reeven} is telling us that, with the radius defined as in \cref{def:R}, changes on the radius are only due to the odd-in-$\xi_2$ part of the solution. In particular, in the approximate solution constructed in the next section, we will show that the even-in-$\xi_2$ symmetry is lost only through viscous effects, meaning that one is able to effectively gain a factor $\cO(\delta)$ on the right-hand side of \cref{def:R}.
\begin{proof}
    From the definition in \cref{def:cA}, we know that 
    \begin{align}
        \Re(\cB_{\eps,R}(\Psi,\widehat{\Psi})) =-\langle \de_2(\cQ_{\eps,R}\Psi+\gamma \cT_{\eps,R}\widehat{\Psi}),\Delta\Psi\rangle.
    \end{align}
    Appealing to Lemma \ref{lem:adjoints}, since $\Psi,\widehat{\Psi}$ are even in $\xi_2$, we know that $\de_2(\cQ_{\eps,R}\Psi)$ and $\de_2(\cT_{\eps,R}\widehat{\Psi})$ are odd in $\xi_2$ whereas $\Delta\Psi$ is even in $\xi_2,$ meaning that the integral above is zero.

    The expansion in \cref{eq:expansionBeps} is a consequence of Lemma \ref{lem:expansionsR}. In fact, we show in Lemma \ref{lem:expVel} a detailed proof that holds for any position parameter $Z.$
\end{proof}


\section{The approximate solutions}\label{sec:App}

The objective of this section is to construct an approximate solution $(\Mapp{\Omega}{M},\Mapp{\widehat{\Omega}}{M})$ to the system \cref{eq:Omega}--\cref{eq:P} at any desired order of precision $M$. Our approach follows the systematic asymptotic framework initiated in \cite{Gallay2011} and subsequently refined in \cite{GaSring,dolce2024long,donati2025fast,zhang2025long}. The new aspect of our derivation is the self-consistent coupling of the modulation parameters: we fix the approximate radius and angular velocity to be those defined in \cref{def:R,def:alpha}, evaluated directly at the $M$-th order profiles $(\Mapp{\Omega}{M},\Mapp{\widehat{\Omega}}{M})$. This technical choice streamlines the construction of higher-order profiles and, more crucially, leads to improved error bounds in the subsequent nonlinear analysis. Since this specific coupling modifies the recursive hierarchy of the expansion, previous results cannot be applied directly as a black box. Consequently, while we follow the general strategy established \cite{dolce2024long}, we provide a self-contained construction to clarify the precise impact of the choice of the modulation parameters in the full expansion.

\subsection{Setting up the construction}

We look for approximate solutions of the equations in Lemma \ref{lem:derivationequation}, using the simplified notation in \cref{eq:OmPhiR}. At every order $M \ge 0$, we denote the approximate solutions by $\Mapp{\Omega}{M},\Mapp{\widehat{\Omega}}{M}$ and we assume the asymptotic expansion 
\begin{align}
    \label{def:Omapp}&\Mapp{\Omega}{M}(\xi,t)=\sum_{k=0}^M\eps^k(t)\Omega_{k}(\xi), \qquad \Mapp{\widehat{\Omega}}{M}(\xi,t)=\sum_{k=0}^M\eps^k(t)\widehat{\Omega}_{k}(\xi)
\end{align}
where the families of maps $(\Omega_k)_{k> 0}$ and $(\widehat{\Omega}_k)_{k> 0}$ must be determined. Recall that $\Omega_0$ and $\widehat{\Omega}_0$ are already defined at relation~\cref{def:G}, a choice that ensures in particular that for any $M$, $$\Mapp{\Omega}{M}(\xi,0) = \Mapp{\widehat{\Omega}}{M}(\xi,0) = \Omega_0(\xi),$$ which coincides with the true solution at time 0, see relation~\cref{eq:Init}.

We then denote by
\begin{equation}
    \Psi_k := \Delta^{-1}\Omega_k , \quad \widehat{\Psi}_k := \Delta^{-1}\widehat{\Omega}_k \quad \text{ and } \quad  \Mapp{\Psi}{M} = \sum_{k=0}^M\eps^k \Psi_k, \quad \Mapp{\widehat{\Psi}}{M} = \sum_{k=0}^M \eps^k\widehat{\Psi}_k,
\end{equation}
which, by linearity, satisfy $\Mapp{\Psi}{M} =  \Delta^{-1}\Mapp{\Omega}{M}$ and $\Mapp{\widehat{\Psi}}{M} =  \Delta^{-1}\Mapp{\widehat{\Omega}}{M}$. 
It is convenient to expand also in $\delta$ as to handle separately viscous effects, since we will be working in the large Reynolds number regime $\delta\ll1$. We will then expand each approximation at $k$-th order with the Eulerian and Navier-Stokes component as 
\begin{align}
\label{def:OmkexpNS}
\Omega_{k}=\Omega_{k,\mathrm{E}}+\delta\Omega_{k,\mathrm{NS}}, \qquad \widehat{\Omega}_k=\widehat{\Omega}_{k,\mathrm{E}}+\delta \widehat{\Omega}_{k,\mathrm{NS}}
\end{align}
    and the same expansion for the associated streamfunctions.
Accordingly, we also split 
\begin{align}
\Mapp{\Omega}{M}=\Mappp{\Omega}{M}{E}+\delta\Mappp{\Omega}{M}{NS}, \qquad \Mapp{\widehat{\Omega}}{M}=\Mappp{\widehat{\Omega}}{M}{E}+\delta\Mappp{\widehat{\Omega}}{M}{NS}.
\end{align}

Then, we fix the radius and angular speed of the rotating frame associated with each approximate solution as the ones given in Lemma \ref{lem:alpha'gives_moments}. Namely, we define $\Mapp{R}{M},\Mapp{\alpha}{M}$ through the relation \cref{def:R,def:alpha} as
    \begin{align}
            \label{def:Rapp} & t (\Mapp{R}{M})'(t) = \frac{d \eps(t)}{\delta } \Re( \cB_{\eps,\Mapp{R}{M}}(\Mapp{\Psi}{M},\Mapp{\widehat{\Psi}}{M}))\\
            \label{def:alphaapp}&t (\Mapp{\alpha}{M})'(t) = \frac{d \eps(t)}{\delta \Mapp{R}{M}}  \Im(\cB_{\eps,\Mapp{R}{M}}(\Mapp{\Psi}{M},\Mapp{\widehat{\Psi}}{M})) \\
            &\Mapp{R}{M}(0)=\mathsf{r}, \quad  \Mapp{\alpha}{M}(0) = 0.
    \end{align}
This choice of $\Mapp{R}{M},\Mapp{\alpha}{M}$ differs from what was done in previous works~\cite{dolce2024long,zhang2025long}, where the modulation parameters are constructed iteratively with an asymptotic expansion. One then needs to justify \emph{a posteriori} that the formulas \cref{def:Rapp,def:alphaapp} hold up to small enough errors. We believe it is more intuitive to fix the approximate modulation parameter from the beginning and verify that it is sufficient to proceed with the inductive construction of the approximate solution. From a practical point of view, it also matches better the technical difficulties that we face. \medskip 

To measure how good of an approximate solution we have constructed, we introduce the remainders
\begin{equation}
\label{def:RM}
    \cR_{M}=\delta(t\de_t-\cL)\Mapp{\Omega}{M}+\{\Mapp{\Phi}{M},\Mapp{\Omega}{M}\}, \qquad \widehat{\cR}_{M}=\delta(t\de_t-\cL)\Mapp{\widehat{\Omega}}{M}+\{\Mapp{\widehat{\Phi}}{M},\Mapp{\widehat{\Omega}}{M}\},
\end{equation}   
where $\Mapp{\Phi}{M},\Mapp{\widehat{\Phi}}{M}$ are an abuse of notation for the use of operators~\cref{def:Xieps} and~\cref{def:Thetaeps} applied at the approximate solutions, namely
\begin{equation}\label{def:Phiapp}
    \Mapp{\Phi}{M} := \Phi_{\eps,\Mapp{R}{M}} (\Mapp{\Psi}{M},\Mapp{\widehat{\Psi}}{M}), \qquad \Mapp{\widehat{\Phi}}{M} := \widehat{\Phi}_{\eps,\Mapp{R}{M}} (\Mapp{\Psi}{M},\Mapp{\widehat{\Psi}}{M}).
\end{equation}

We now state the existence of a suitable choice of approximate solution.
\begin{proposition}\label{prop:construction}
    There exists two families $(\Omega_k)_{k\ge 0}$ and $(\widehat{\Omega}_k)_{k\ge 0}$ of functions belonging to the space $\cZ$ such that for every $k \ge 0$:
    \begin{enumerate}[label=\textbf{C\arabic*)}]

        \item \label{constraint:1}$\mrm^1(\Omega_k)=\mrm^1(\widehat{\Omega}_k)=(0,0)$ and, for $k>0$, $\mrM(\Omega_k)=\mrM(\widehat{\Omega}_k)=0$,
        \item\label{constraint:2}the functions $\Omega_{k,E}$ and $\widehat{\Omega}_{k,E}$ are 
        even in $\xi_2$,
        \item\label{constraint:3} $\widehat{\Omega}_k$ is $N$-fold symmetric.
    \end{enumerate}
    
        Moreover, for every $M \in \NN$, the associated approximate solution defined by~\cref{def:Omapp} satisfies:
    \begin{enumerate}[label=\textbf{C\arabic*)},resume]
            \item\label{constraint:4}there exists a constant $C$ independent of $\delta$ such that for every $t \le C \, T_\adv\delta^{-1}$,
        $$\frac{1}{2} \mathsf{r} \le \Mapp{R}{M}(t) \le \frac{3}{2} \mathsf{r},$$ where $\Mapp{R}{M}$ is given by \cref{def:Rapp},
        \item\label{constraint:5} the remainders $\cR_{M},\widehat{\cR}_M$ defined by~\cref{def:RM} are of order $\cO_{\cZ}(\eps^{M+1}+\delta^2\eps^2)$ and satisfy $\mrm^j(\cR_M)=\mrm^j(\widehat{\cR}_M)=(0,0)$ for $j=0,1$.
    \end{enumerate}
\end{proposition}
The rest of Section~\ref{sec:App} is devoted to the proof of Proposition~\ref{prop:construction}. The constraints \ref{constraint:1} to \ref{constraint:3} are technical constraints ensuring that we can manipulate correctly the functions we construct, and that the approximate solution has important properties related to symmetries. The constraint \ref{constraint:4} is a consequence of the even symmetry of the Eulerian part, and it is necessary to ensure that we can safely expand our operators. The constraint \ref{constraint:5} is the quantitative control of how close to the approximate solution is  the real solution, which is the whole point of the construction.

\begin{remark}\label{rem:dividebygamma}
    In the case $\gamma =0$, the choice of the family $(\widehat{\Omega}_k)_{k \ge 0}$ is irrelevant. Therefore, while computing the $\widehat{\Omega}_k$, we will assume $\gamma \neq 0$. This assumption does not need to be made to construct the family $(\Omega_k)_{k \ge 0}$, hence our construction is perfectly valid for $\gamma =0$, even if the expression $1/\gamma$ appears in the formula for $\widehat{\Omega}_k$.
\end{remark}

The construction will be done iteratively, meaning that after having constructed $\Omega_0,\ldots,\Omega_M$, and $\widehat{\Omega}_0,\ldots,\widehat{\Omega}_M$, satisfying constraints~\ref{constraint:1} to~\ref{constraint:3}, and such that the associated approximate solution $\Mapp{\Omega}{M}$ and $\Mapp{\widehat{\Omega}}{M}$ satisfy constraints~\ref{constraint:4}, \ref{constraint:5} at order $M$, then we construct $\Omega_{M+1}$ and $\widehat{\Omega}_{M+1}$, satisfying constraints~\ref{constraint:1} to~\ref{constraint:3} such that the next order approximate solution  $\Mapp{\Omega}{M+1}$ and $\Mapp{\widehat{\Omega}}{M+1}$ satisfy constraints~\ref{constraint:4}, \ref{constraint:5} at order $M+1$.

The rest of the present section is organized as follows. In Section~\ref{sec:first_orders} and~\ref{sec:second_order} we present the computations of the first orders, giving both insights and numerically exploitable results on the actual deformation of vortices. Then, in Section~\ref{sec:iterative}, we construct iteratively the whole family $(\Omega_k)_{k\in \NN}$, leading to the construction of an approximate solution at every order. Then we compute the first corrections to the point-vortex dynamics in Section~\ref{sec:first_corrections}. In Sections~\ref{sec:func_rel} and~\ref{sec:angmom}, we establish some important properties of the Eulerian part of the approximate solution such as the existence of an approximate functional relationship and remarkable relations for its angular momentum.

\subsection{The zero-th and first order}\label{sec:first_orders}
We recall that the choice of the zero-th order is dictated by the initial data, which is the Lamb-Oseen vortex
\begin{equation}
    \Omega_0= \widehat{\Omega}_0 = G,
\end{equation}
Therefore, $\Mapp{\Omega}{0} = \Mapp{\widehat{\Omega}}{0} = G$. Since $\mrm^1(G) = 0$, we have that $\Omega_0$ and $\widehat{\Omega}_0$ satisfy constraints \ref{constraint:1} to \ref{constraint:3}. Moreover, since $G$ is radially symmetric, $\Psi_0,\widehat{\Psi}_0$ are radially symmetric as well and thus even in $\xi_2$. By Lemma \ref{lem:eveneven}, we readily deduce that $t(\Mapp{R}{0})'=0$ and therefore 
\[
\Mapp{R}{0}(t)=\mathsf{r}.
\]
With those expressions in hand, we can compute and expand the zero-th order remainders.
\begin{lemma}
\label{lem:R0}
    For any $M\geq 2$ we have $\cR_0,\widehat{\cR}_0=\cO_{\cZ}(\eps^2)$ with explicit expansion
    \begin{align}
\label{eq:RG}\cR_0&=\sum_{n=2}^M\frac{\eps^nc_{n,0}}{\mathsf{r}^n}({\rm S}_{n,n}+\gamma)\{\Re(\xi^n),G\}+\cO_{\cZ}(\eps^{M+1}), \\
\label{eq:JG}\widehat{\cR}_0&=N\sum_{\substack{\ell \leq M/N, \\ \ell \in \mathbb{N}\setminus\{0\}}}\frac{\eps^{\ell N}c_{\ell N,0}}{\mathsf{r}^{\ell N}}\big\{\Re(\xi^{\ell N}),G\big\}+\cO_{\cZ}(\eps^{M+1}).
    \end{align}
\end{lemma}
\begin{proof}
    We start by observing that since $G$ is a self-similar solution of the heat equation, we have $(t\de_t-\cL)(G)=0$. Thus, for $\cR_0$ we see that
\begin{align}
 \cR_0=\{\Mapp{\Phi}{0},\Omega_0\}=\Big\{\cQ_{\eps,\Mapp{R}{0}}\Psi_0+\gamma\cT_{\eps,\Mapp{R}{0}}\widehat{\Psi}_0-\frac{\eps d}{2\Mapp{R}{0}}\Im(\cB_{\eps,\Mapp{R}{0}}(\Psi_0,\widehat{\Psi}_0))|\cT_\eps\xi|^2,\Omega_0\Big\}.
\end{align}
Since $\mrm^k(G)=0$ for any $k\geq1$, and $\Mapp{R}{0}=\mathsf{r}$,
applying \cref{lem:expansionsR}, we have
\begin{align}
\{\cQ_{\eps,\Mapp{R}{0}}\Psi_{0},\Omega_0\}=\frac{\eps c_{1,0}}{\mathsf{r}} \mathrm{S}_{1,1}\{\Re(\xi),G\}+\sum_{n=2}^M\frac{\eps^nc_{n,0}}{\mathsf{r}^n}{\rm S}_{n,n}\{\Re(\xi^n),G\}+\cO_{\cZ}(\eps^{M+1}),
\end{align}
and
\begin{align}
\gamma\{\cT_{\eps,\Mapp{R}{0}}\widehat{\Psi}_{0},\Omega_0\}=\gamma\frac{\eps c_{1,0}}{\mathsf{r}}\{\Re(\xi),G\}+\gamma\sum_{n=2}^M\frac{\eps^nc_{n,0}}{\mathsf{r}^n}\{\Re(\xi^n),G\}+\cO_{\cZ}(\eps^{M+1}).
\end{align}
We now turn to the term involving $\cB_{\eps,\Mapp{R}{0}}$. Using Lemma~\ref{lem:expansionsR} combined with the fact that $\mrm^k(G) = 0$ for every $k > 0$, we get that
\begin{equation}
     \frac{\eps d}{2 R} \Im(\cB_\eps(\Psi_0,\widehat{\Psi}_0))
      = \frac{\eps^2 d}{2 \mathsf{r}^2} c_{1,0}(S_{1,1}+\gamma) + o(\eps^\infty),
\end{equation}
and since
\begin{align}
\{|\cT_{\eps,\Mapp{R}{0}}\xi|^2,\Omega_0\}=\{|\xi+\frac{1}{\eps d}(\mathsf{r},0)|^2,G\}=\frac{2\mathsf{r}}{\eps d}\{\xi_1,G\}=\frac{2\mathsf{r}}{\eps d}\{\Re(\xi),G\},
\end{align}
we obtain that
\begin{equation}
    \cR_0=\sum_{n=2}^M\frac{\eps^nc_{n,0}}{\mathsf{r}^n}({\rm S}_{n,n}+\gamma)\{\Re(\xi^n),G\}+\cO_{\cZ}(\eps^{M+1}),
\end{equation}
where the term of order $\eps$  has canceled out.

Turning our attention to $\widehat{\cR}_0$, the formula \cref{eq:JG} directly follows by recalling the definition \cref{def:Thetaeps}, using \cref{lem:expansionsR} and the radial symmetry of $G$.
\end{proof}

Since, $\cR_0, \widehat{\cR}_0 = \cO(\eps^2)$, one can choose, as in \cite{Gallay2011,dolce2024long,donati2025fast,zhang2025long}, the first order to vanish exactly:
\begin{equation}
    \Omega_1 = \widehat{\Omega}_1 = 0,
\end{equation}
leading in particular to $\Mapp{\Omega}{1} = \Mapp{\widehat{\Omega}}{1} = \Mapp{\Omega}{0} = \Mapp{\widehat{\Omega}}{0} = \Omega_0$. Indeed, this choice of $\Omega_1$ and $\widehat{\Omega}_1$ satisfies constraints~\ref{constraint:1}~to~\ref{constraint:3}. In addition, it leads to identical formulas for the remainders: $\cR_1 = \cR_0$ and $\widehat{\cR}_1 = \widehat{\cR}_0$, so by Lemma~\ref{lem:R0}, constraints~ \ref{constraint:4}, \ref{constraint:5} are satisfied at order $M=1$, which concludes the construction of the first order.

\begin{remark}
\label{rem:special}
    Note that if $\gamma=-{\rm S}_{2,2}$ then by Lemma~\ref{lem:R0}, $\cR_0 = \cO_{\cZ}(\eps^3)$, meaning that, as we did for the first order, one can take $\Omega_2=0$ and satisfy all the constraints at order $2$. This explains the fact observed in Section~\ref{sec:corrections} that for the special value $\gamma_N^* = -S_{2,2} = \frac{(N-1)(N-5)}{12}$, the exterior vortices do not show any $2$-fold symmetric deformation. By the same argument, if $N>2$, for any value of $\gamma$, the choice $\widehat{\Omega}_k=0$ for $k=2,\dots,N-1$ satisfies all constraints since in that case $\widehat{\cR}_{N-1} = \ldots = \widehat{\cR}_0 = \cO_{\cZ}(\eps^N)$.
\end{remark}

\subsection{Construction of the second order}\label{sec:second_order}

As a propaedeutical way to understanding the iterative scheme, and in order to get the first non-trivial correction explicitly, we now give in detail the construction of the second order. The procedure is clearly analogous to the one first proposed in \cite{Gallay2011} and extended in \cite{dolce2024long,zhang2025long}. However, here  we first need to guarantee some \emph{a priori} properties of the approximate position in \cref{def:Rapp}. We  begin by computing an expression of $\Mapp{R}{2}$.
\begin{lemma}\label{lem:computeR2}
    Assume that $\Omega_2,\widehat{\Omega}_2$ satisfy constraints~\ref{constraint:1} to \ref{constraint:3}. Let $T$ be the first time such that \ref{constraint:4} fails. Then, for every $t \in [0,T]$ we have that,
    \begin{equation}
        t(\Mapp{R}{2})'(t) = \cO( \eps^6(t)), \qquad \Mapp{R}{2} = \mathsf{r}(1+ \cO(\eps^6(t))), \qquad \delta t (\Mapp{\alpha}{2} - \alpha_0)'(t) =  \cO(\eps^6(t)).
    \end{equation}
\end{lemma}
\begin{proof}
We start by decomposing $\Omega_2 = \Omega_{2,\mathrm{E}} + \delta\, \Omega_{2,\mathrm{NS}}$, and the approximate solution as \begin{equation}
    \Mapp{\Omega}{2} = \underbrace{\Omega_0 + \eps^2 \Omega_{2,\mathrm{E}}}_{\Mappp{\Omega}{2}{E}} + \eps^2\delta \, \Omega_{2,\mathrm{NS}}
\end{equation}
and similary, $\Mapp{\widehat{\Omega}}{2} = \Mappp{\widehat{\Omega}}{2}{E} + \delta\widehat{\Omega}_{2,\mathrm{NS}}$.
Recalling that $\cB_{\eps,\Mapp{R}{2}}$ is actually a map in three variables (see its full definition~\cref{def:cA}), linear with respect to the third variable, but affine with respect to the other two, we compute that
\begin{align}
    &\cB_{\eps,\Mapp{R}{2}}(\Mapp{\Psi}{2},\Mapp{\widehat{\Psi}}{2}) = \cB_{\eps,\Mapp{R}{2}}(\Mappp{\Psi}{2}{E},\Mappp{\widehat{\Psi}}{2}{E}) \\
     &\qquad + \eps^2\delta \Big( \cB_{\eps,\Mapp{R}{2}}(\Psi_{2,\mathrm{NS}},\widehat{\Psi}_{2,\mathrm{NS}},\Mappp{\Psi}{2}{E}) + \cB_{\eps,\Mapp{R}{2}}(\Mappp{\Psi}{2}{E},\Mappp{\widehat{\Psi}}{2}{E},\Psi_{2,\mathrm{NS}})\Big)  \\
    &\qquad + \eps^4\delta^2\cB_{\eps,\Mapp{R}{2}}(\Psi_{2,{\rm NS}},\Psi_{2,{\rm NS}}).
\end{align}
Using \cref{eq:Reeven} in Lemma~\ref{lem:eveneven}, by the $\xi_2$ symmetry of the purely Eulerian part we get that
\begin{equation}
    \Re(\cB_{\eps,\Mapp{R}{2}}(\Mappp{\Psi}{2}{E},\Mappp{\widehat{\Psi}}{2}{E})) = 0.
\end{equation}
For the remaining terms we use \cref{eq:expansionBeps} and we observe the following: since $\mrm^k(\Omega_{2,NS})=(0,0)$ for $k\geq 0$ and $\mrm^1(\Omega_{\rm app,E})=(0,0),$ the associated sum in \cref{eq:expansionBeps} starts from $n=3$. Thus, applying that expansion with $M=3$, we see that
\begin{equation}
    \Re(\cB_{\eps,\Mapp{R}{2}}(\Mapp{\Omega}{2},\Mapp{\widehat{\Omega}}{2})) = \cO(\delta\eps^5)
\end{equation}
from which we infer  
\begin{equation}
    t (\Mapp{R}{2})'(t) = \frac{d \eps}{\delta} \Re(\cB_{\eps,\Mapp{R}{2}}(\Mapp{\Omega}{2},\Mapp{\widehat{\Omega}}{2})) = \cO(\eps^6).
\end{equation}
Integrating over time, it yields that
\begin{equation}
    \Mapp{R}{2}(t) = \mathsf{r} + \cO(\eps^6).
\end{equation}
To control the angular velocity, we do not have the same cancellations for the Eulerian part, but we observe that 
\begin{align}
    & \cB_{\eps,\Mapp{R}{2}}(\Mappp{\Psi}{2}{E},\Mappp{\widehat{\Psi}}{2}{E})=\cB_{\eps,\Mapp{R}{2}}(\Psi_0,\widehat{\Psi}_0) \\
     &\qquad + \eps^2\Big( \cB_{\eps,\Mapp{R}{2}}(\Psi_{2,\mathrm{E}},\widehat{\Psi}_{2,\mathrm{E}},\Psi_0) + \cB_{\eps,\Mapp{R}{2}}(\Psi_0,\widehat{\Psi}_0,\Psi_{2,\mathrm{E}})\Big)  \\
    &\qquad + \eps^4\cB_{\eps,\Mapp{R}{2}}(\Psi_{2,{\rm E}},\Psi_{2,{\rm E}}).
\end{align}
Thus, by using again \cref{eq:expansionBeps} in Lemma \ref{lem:eveneven}, we observe that the first term is of order $\eps$, whereas for the other ones  we again start from the order $n=3$, meaning that 
\begin{equation}
    \cB_{\eps,\Mapp{R}{2}}(\Mapp{\Omega}{2},\Mapp{\widehat{\Omega}}{2}) = i \frac{\eps}{\Mapp{R}{2}} c_{1,0} (S_{1,1}+\gamma) + \cO(\eps^5).
\end{equation}
Thus, using the information on $\Mapp{R}{2}$, we infer 
\begin{equation}
    \delta t (\Mapp{\alpha}{2})'(t) = \frac{d \eps}{\Mapp{R}{2}(t)} \Im(\cB_{\eps,\Mapp{R}{2}}(\Mapp{\Omega}{2},\Mapp{\widehat{\Omega}}{2})) = \delta t \alpha_0'(t) + \cO(\eps^6).
\end{equation}
\end{proof}
Please note that the fact that we cannot recover a $\delta$ factor in $(\Mapp{\alpha}{2})'$ means that the difference $|\Mapp{\alpha}{2} - \alpha_0|$ may become of order 1 after a time $t \ll T_\adv \delta^{-1}$. The explicit computation of this difference is delayed in Section~\ref{sec:first_corrections}, since it is not relevant to the construction of the approximate solution. We now infer from Lemma~\ref{lem:computeR2} that constraint \ref{constraint:4} is satisfied at the second order regardless of the choice of $\Omega_2,\widehat{\Omega}_2$ as long as it satisfies the first three constraints.
\begin{corollary}
    Any $\Omega_2,\widehat{\Omega}_2$ satisfying constraints \ref{constraint:1} to \ref{constraint:3} immediately satisfy constraint \ref{constraint:4}.
\end{corollary}
\begin{proof}
     By continuity of $\Mapp{R}{2}$, Lemma~\ref{lem:computeR2} implies that there exists a constant $C$ independent of $\delta$ such that
\begin{equation}
    \frac{R}{2} = |\Mapp{R}{2}(T) - R| \le C\eps^6(T) = C(\nu T)^{3}
\end{equation}
hence 
\begin{equation}
    T \geq \frac{1}{\Gamma\delta}\left( \frac{R}{2C}\right)^{1/3}.
\end{equation}
So there exists a constant $C'$ independent of $\delta$ such that for every $t \le C' \delta^{-1}$,
\begin{equation}
    \frac{1}{2} R \le \Mapp{R}{2}(t) \le \frac{3}{2} R.
\end{equation}
\end{proof}
Let us compute $\cR_2$ in terms of $\Omega_2$ and $\widehat{\Omega}_2$ to understand how constraint \ref{constraint:5} will dictate the construction. Many terms are involved, and the important thing is to isolate the main contributions as follows.
\begin{lemma}\label{lem:computecR2}
We have that
    \begin{equation}
        \cR_2 = \eps^2 \Big( \delta ( 1-\cL) \Omega_2 + \Lambda \Omega_2 + \cH_2\Big) + \cO_\cZ(\eps^3),
    \end{equation}
    where
    \begin{equation}
        \cH_2 := \frac{c_{2,0}}{\mathsf{r}^2} (S_{2,2} + \gamma) \{ \Re(\xi^2), G\}
    \end{equation}
    is the error at order $\cO_\cZ(\eps^2)$ coming from $\cR_0$. Similarly,
    \begin{equation}
        \widehat{\cR}_2 = \eps^2 \Big( \delta ( 1-\cL) \widehat{\Omega}_2 + \gamma\Lambda \widehat{\Omega}_2 + \cH_2\Big) + \cO_\cZ(\eps^3),
    \end{equation}
    where
    \begin{equation}
    \widehat{\cH}_2 := \begin{cases}
        2\frac{c_{2,0}}{\mathsf{r}^2}  \{ \Re(\xi^2), G\} & \text{ if } N = 2 \\
         0 & \text{ if } N > 2,
    \end{cases}
\end{equation}
is the error at order $\cO_\cZ(\eps^2)$ coming from $\widehat{\cR}_0$.
\end{lemma}
\begin{proof}
    We compute the remainder $\cR_2$ associated with the second-order approximate function given by the expression $\Mapp{\Omega}{2}(\xi,t) = \Omega_0(\xi) + \eps^2(t)\Omega_2(\xi)$:
    \begin{align}
        \cR_2 = \delta(t \partial_t - \cL)\Mapp{\Omega}{2} + \{ \Mapp{\Phi}{2},\Mapp{\Omega}{2} \}.
    \end{align}
    We look at the first term. Recalling that $t\partial_t \eps^2 = \eps^2$, we get that 
    \begin{equation}
        \delta(t \partial_t - \cL)(\Omega_0 + \eps^2 \Omega_2)  = \eps^2\delta(1-\cL)\Omega_2.
    \end{equation}
    For the second term, we isolate the terms already computed (the ones present in the construction of the previous order) and two different kinds of errors:
    \begin{align}
        \{ \Mapp{\Phi}{2},\Mapp{\Omega}{2} \} = \{ \Mapp{\Phi}{0},\Mapp{\Omega}{0} \} + \eps^2\{ \Mapp{\Phi}{0},\Omega_2 \} + \{ \Mapp{\Phi}{2} - \Mapp{\Phi}{0},\Mapp{\Omega}{2} \}.
    \end{align}
    We now observe that
    \begin{equation}
        \{ \Mapp{\Phi}{0},\Mapp{\Omega}{0} \} = \cR_0 = \eps^2\cH_2 +  \cO_\cZ(\eps^3)
    \end{equation} 
    as obtained in Section~\ref{sec:first_orders}, then, using Proposition~\ref{prop:small_operator} we obtain that
    \begin{equation}
        \eps^2\{ \Mapp{\Phi}{0},\Omega_2 \} = \eps^2 \{ \Psi_0, \Omega_2 \} + \cO_\cZ(\eps^3).
    \end{equation}    
    The last term involves $\Mapp{\Phi}{2} - \Mapp{\Phi}{0}$, which we recall that properly written is
    \begin{equation}
        \Mapp{\Phi}{2} - \Mapp{\Phi}{0} := \Phi_{\eps,\Mapp{R}{2}}(\Mapp{\Psi}{2},\Mapp{\widehat{\Psi}}{2}) - \Phi_{\eps,\Mapp{R}{0}}(\Mapp{\Psi}{0},\Mapp{\widehat{\Psi}}{0}).
    \end{equation}
        Not only are those maps applied in different functions, but they are also computed at different approximate positions. Beware that $\Phi$ is not a linear operator (but an affine one), and the following identity holds:
    \begin{align}
        \Mapp{\Phi}{2} - \Mapp{\Phi}{0} = \,&\eps^2\Big[ ( I +\cQ_{\eps,\Mapp{R}{2}}) \Psi_2 + \gamma \cT_{\eps,\Mapp{R}{2}} \widehat{\Psi}_2\Big] \\
      \label{eq:Phi2Phi01}  &+\xi_2\Re\big( \cB_{\eps,\Mapp{R}{2}}(\Mapp{\Psi}{2},\Mapp{\widehat{\Psi}}{2})- \cB_{\eps,\Mapp{R}{2}}(\Psi_0,\widehat{\Psi}_0) \big)\\
        \label{eq:Phi2Phi02}&-\frac{d \eps}{2 \Mapp{R}{2}} \big| \cT_{\eps,\Mapp{R}{2}} \xi\big|^2\Im\big( \cB_{\eps,\Mapp{R}{2}}(\Mapp{\Psi}{2},\Mapp{\widehat{\Psi}}{2})- \cB_{\eps,\Mapp{R}{2}}(\Psi_0,\widehat{\Psi}_0) \big)\\
        &+ (\Phi_{\eps,\Mapp{R}{2}} - \Phi_{\eps,\Mapp{R}{0}})(\Psi_0,\widehat{\Psi}_0).
    \end{align}
    From Proposition~\ref{prop:small_operator} we have that
    \begin{equation}
        \eps^2\big\{( I + \cQ_{\eps,\Mapp{R}{2}}) \Psi_2 + \gamma \cT_{\eps,\Mapp{R}{2}} \widehat{\Psi}_2\, , \, \Mapp{\Omega}{2} \big\}= \eps^2 \{ \Psi_2 , \Omega_0\} + \cO_\cZ(\eps^3). 
    \end{equation}
    Then, we compute the last term coming from changing the approximate position:
    \begin{align}
        &(\Phi_{\eps,\Mapp{R}{2}} - \Phi_{\eps,\Mapp{R}{0}})(\Psi_0,\widehat{\Psi}_0) = (\cQ_{\eps,\Mapp{R}{2}} - \cQ_{\eps,\Mapp{R}{0}} \big)\Psi_0 + \gamma(\cT_{\eps,\Mapp{R}{2}} - \cT_{\eps,\Mapp{R}{0}})\widehat{\Psi}_0 \\
        &\qquad+ \xi_2\Re(\cB_{\eps,\Mapp{R}{2}}(\Psi_0,\widehat{\Psi}_0) - \cB_{\eps,\Mapp{R}{0}}(\Psi_0,\widehat{\Psi}_0) )\\
        &\qquad- \Im\Big( \frac{d\eps}{2\Mapp{R}{2}}\big| \cT_{\eps,\Mapp{R}{2}} \xi \big|^2\cB_{\eps,\Mapp{R}{2}}(\Psi_0,\widehat{\Psi}_0) - \frac{d\eps}{2\Mapp{R}{0}}\big| \cT_{\eps,\Mapp{R}{0}} \xi \big|^2\cB_{\eps,\Mapp{R}{0}}(\Psi_0,\widehat{\Psi}_0)\Big).
    \end{align}
    Combining Lemma~\ref{lem:computeR2} and Proposition~\ref{prop:small_operator}, we get that all the terms involved in the last expression are of order $\cO(\eps^3)$ at least. Moreover, for the terms in \cref{eq:Phi2Phi01,eq:Phi2Phi02} we can apply Corollary \ref{lem:expA}  and see that indeed  \begin{equation}
        \Mapp{\Phi}{2} - \Mapp{\Phi}{0} = \eps^2 \Psi_2  + \cO_{\cS_*}(\eps^3).
    \end{equation}
Therefore, we conclude that
    \begin{equation}
        \cR_2 = \eps^2\Big(\delta(1-\cL)\Omega_2 + \{\Psi_0,\Omega_2 \} + \{\Psi_2,\Omega_0\} + \cH_2 \Big) + \cO_\cZ(\eps^3),
    \end{equation}
    which gives the desired result by recalling that $\Lambda \Omega = \{ \Psi_0,\Omega\} + \{ \Delta^{-1} \Omega , \Omega_0\}$.

    Similar computations, left to the reader, yield the result regarding $\widehat{\cR}_2$.
\end{proof}
With Lemma \ref{lem:computecR2}, we have a very natural (and, remarkably, somewhat \emph{unique}) choice of $\Omega_2$ and $\widehat{\Omega}_2$ satisfying constraint \ref{constraint:5}. This choice is made so that all $\cO(\eps^2)$ terms in $\cR_2$ and $\widehat{\cR}_2$ vanish. We can now formulate the construction of the approximate solution up to the second order.
\begin{proposition}
\label{prop:Om2}
Let $\Omega_2 =  \Omega_{2,\mathrm{E}} + \delta\Omega_{2,\mathrm{NS}}$ be defined by the following formulas:
    \begin{equation}
        \Omega_{2,\mathrm{E}} := -\Lambda^{-1} \cH_2, \qquad \Omega_{2,\mathrm{NS}} := -\Lambda^{-1} (1-\cL)\Omega_{2,\mathrm{E}}.
    \end{equation}
     Similarly, let $\widehat{\Omega}_2 =  \widehat{\Omega}_{2,\mathrm{E}} + \delta\widehat{\Omega}_{2,\mathrm{NS}}$ be defined by
    \begin{equation}
        \widehat{\Omega}_{2,\mathrm{E}} := -\frac{1}{\gamma}\Lambda^{-1} \widehat{\cH}_2, \qquad \widehat{\Omega}_{2,\mathrm{E}} :=  -\frac{1}{\gamma}\Lambda^{-1} (1-\cL)\widehat{\Omega}_{2,\mathrm{E}}.
    \end{equation}
Then, $\Omega_2$, $\widehat{\Omega}_2$ and the associated second order approximate solution satisfy all constraints \ref{constraint:1} to \ref{constraint:5}.
\end{proposition}
\begin{remark}
\label{rem:Om2Ninfinity} Note that the dependence on $N$ in the definition of $\Omega_2$ is only through the constant $c_{2,0}({\rm S}_{2,2}+\gamma)$. As observed in Remark \ref{rem:sheet}, we know that $d/\mathsf{r}=\cO(N^{-1}),$ and therefore, by \cref{def:cnk}, we have $c_{2,0}/\mathsf{r}^2=\cO(N^{-2})$. Moreover ${\rm S}_{2,2}=\cO(N^2)$, meaning that  the $N$-dependent constant hidden in the definition of $\Omega_2$ is indeed of order $\cO(1)$ as announced in Remark \ref{rem:sheet}. 
\end{remark}
\begin{proof}
As explained in Remark~\ref{rem:special}, if $\gamma = -S_{2,2}$ then setting $\Omega_2 = 0$ satisfies all the constraints, and so does setting $\widehat{\Omega}_2 = 0$ in the case $N \ge 3$.

Let $h(\xi) := \{ \Re(\xi^2),G\}(\xi) = -\xi_1\xi_2 G(\xi)$. Since $h$ satisfies that
\begin{itemize}
    \item $h \in \mathrm{Ker}(\Lambda)^{\perp}\cap \cZ$,
    \item $h$ is odd in $\xi_2$,
    \item $h$ is also two-fold symmetric (hence $\widehat{\cH}_2$ is always $N$-fold symmetric)
    \item $M(h) = 0$ and $m^1(h) = (0,0)$,
\end{itemize}
then, using the properties of $\Lambda$ gathered in Proposition \ref{prop:Lambda}, we have that if $\gamma \neq - S_{2,2}$, then $\Omega_{2,\mathrm{E}}$ is well defined, and if $N =2$, then $\widehat{\Omega}_{2,\mathrm{E}}$ is well defined (else, they are well defined by taking them vanishing everywhere). Moreover, we also get that so are $\Omega_{2,\mathrm{NS}}$ and $\widehat{\Omega}_{2,\mathrm{NS}}$, and that $\Omega_2$ and $\widehat{\Omega}_2$ satisfy constraints \ref{constraint:1} to \ref{constraint:3}. With these definitions, we compute using Lemma~\ref{lem:computecR2} that
\begin{align}
    \cR_2 &  = \eps^2 \Big( \delta ( 1-\cL) \Omega_2 + \Lambda \Omega_2 + \cH_2\Big) + \cO_\cZ(\eps^3) \\
    & =  \eps^2 \Big( \delta ( 1-\cL) (\Omega_{2,\mathrm{E}} + \delta\Omega_{2,\mathrm{NS}}) + \Lambda \Omega_{2,\mathrm{NS}}\Big) + \cO_\cZ(\eps^3) \\
    & = \eps^2 \delta^2 (1-\cL) \Omega_{2,\mathrm{NS}} +  \cO_\cZ(\eps^3)  \\
    & = \cO_\cZ(\eps^3 + \eps^2 \delta^2).
\end{align}
Similarly, one can compute that $\widehat{\cR}_2 = \cO_\cZ(\eps^3 + \eps^2 \delta^2)$, and thus constraint \ref{constraint:5} is now fully satisfied at order $M=2$. Note that these estimates still hold if one (or both) of $\Omega_2,\widehat{\Omega}_2$ is vanishing.

This concludes the proof and therefore the construction of the approximate solution at second order.
\end{proof}
\begin{remark}
\label{rem:NSdelta}
Observe that, up to an error of order $\cO(\eps^3 + \eps^2\delta^2)$, Lemma~\ref{lem:computecR2} ensures the existence of a unique choice of $\Omega_2$ satisfying constraint~\ref{constraint:5}. Furthermore, by defining $\Omega_2$ as a higher order polynomial in $\delta$, namely
\begin{equation}
    \Omega_2 = \Omega_{2,\mathrm{E}} + \sum_{k=1}^K \delta^k \Omega_{2,\mathrm{NS},k},
\end{equation}
for some $K \ge 1$, one can uniquely choose $\Omega_{2,\mathrm{NS},1},\ldots,\Omega_{2,\mathrm{NS},K}$ such that $\cR_{2} = \cO_\cZ(\eps^3 + \eps^2 \delta^{K+1})$. Indeed, to define the approximate solution, one could write $(\Lambda+\delta(1-\cL))^{-1}=\Lambda^{-1}(\mathrm{Id}+\delta(1-\cL)\Lambda^{-1})^{-1}$ and then use a truncation of the Neumann series associated with the operator $(\mathrm{Id}+\delta(1-\cL)\Lambda^{-1})^{-1}$, which essentially corresponds to our approach. However, in our functional setting, rigorously inverting $\mathrm{Id}+\delta(1-\cL)\Lambda^{-1}$ presents significant difficulties due to an accumulation of losses in powers of $|\xi|$ as $|\xi|\to \infty.$
\end{remark}

Finally, a few other remarks on the construction of the third-order approximate solution can be found in Section~\ref{sec:first_corrections}. However, instead of performing the complete computations of the third order, we now provide the general procedure to construct every order inductively.

\subsection{The inductive construction}\label{sec:iterative}

Recall from Section~\ref{sec:Func} that $\cP_0$ is the orthogonal projection in $\cY$ onto $\cY_0$ (the average in the angular direction).
We now state how to construct iteratively the approximate solutions at any order.

\begin{proposition}\label{prop:iterative}
    Assume that for some $M \ge 1$, all the maps $\Omega_0,\ldots,\Omega_{M-1}$ and $\widehat{\Omega}_0,\ldots,\widehat{\Omega}_{M-1}$ are constructed satisfying constraints~\ref{constraint:1} to~\ref{constraint:5}. Then by hypothesis, there exists $\cH_{M} = \cH_{M,\mathrm{E}} + \delta\cH_{M,\mathrm{NS}}$ such that
    \begin{equation}
        \cR_{M-1} (\xi,t) = \eps^{M}(t) \cH_{M}(\xi) + \cO_\cZ(\eps^{M+1}(t) + \delta^2 \eps^2(t)),
    \end{equation}
    and $\widehat{\cH}_{M} = \widehat{\cH}_{M,\mathrm{E}} + \delta\widehat{\cH}_{M,\mathrm{NS}}$ such that
    \begin{equation}
        \widehat{\cR}_{M-1}(t,\xi) = \eps^{M}(t) \widehat{\cH}_{M}(\xi) + \cO_\cZ(\eps^{M+1}(t) + \delta^2 \eps^2(t)).
    \end{equation}
    Let $\Omega_{M} = \Omega_{M,\mathrm{E}} + \Omega_{M,\mathrm{E}_r}  + \delta\Omega_{M,\mathrm{NS}}$ be defined by
    \begin{align}
        & \Omega_{M,\mathrm{E}} := - \Lambda^{-1} \cH_{M,\mathrm{E}} \\
        & \Omega_{M,\mathrm{E}_r} := \left(\cL - \frac{M}{2} \right)^{-1} \cP_0(\cH_{M,\mathrm{NS}}),\\
        & \Omega_{M,\mathrm{NS}} := - \Lambda^{-1} \left( (1-\cP_0)\cH_{M,\mathrm{NS}} + \left(\frac{M}{2}-\cL\right) \Omega_{M,\mathrm{E}}\right).
    \end{align}
    Similarly, let $\widehat{\Omega}_{M} =  \widehat{\Omega}_{M,\mathrm{E}} + \widehat{\Omega}_{M,\mathrm{E}_r} + \delta\widehat{\Omega}_{M,\mathrm{NS}}$ with
    \begin{align}
        & \widehat{\Omega}_{M,\mathrm{E}} := - \frac{1}{\gamma}\Lambda^{-1} \widehat{\cH}_{M,\mathrm{E}} \\
        & \widehat{\Omega}_{M,\mathrm{E}_r} := -\left(\frac{M}{2} - \cL\right)^{-1} \cP_0(\widehat{\cH}_{M,\mathrm{NS}}),\\
        & \widehat{\Omega}_{M,\mathrm{NS}} := - \Lambda^{-1} \left( (1-\cP_0)\widehat{\cH}_{M,\mathrm{NS}} + \left(\frac{M}{2}-\cL\right) \widehat{\Omega}_{M,\mathrm{E}}\right).
    \end{align}
    Then, $\Omega_{M}$, $\widehat{\Omega}_{M}$ and the associated approximate solution at order $M$ satisfy constraints~\ref{constraint:1} to~\ref{constraint:5}.
\end{proposition}
The rest of Section~\ref{sec:iterative} is devoted to the proof of Proposition~\ref{prop:iterative} and to the proof that it implies Proposition~\ref{prop:construction}. It follows the same plan as the construction of the second order, except for one major difference which is the need to define a new Eulerian part $\Omega_{M,\mathrm{E}_r}$ which is radially symmetric, since $\cH_{M}$ does not necessarily lie within $(\mathrm{Ker} \Lambda)^\perp$, a problem that only appears from the fourth order approximation onward.

Let us now begin with the proof of Proposition~\ref{prop:iterative}. By definition of $\cR_{M-1}$ and $\widehat{\cR}_{M-1}$ \cref{def:RM}, since all the involved operators and maps can be expanded as series in $\eps$ and $\delta$, then $\cR_{M-1}$ and $\widehat{\cR}_{M-1}$ can themselves be expanded as series in $\eps$ and $\delta$. Therefore, we have indeed the existence of $\cH_{M-1}$ and $\widehat{\cH}_{M-1}$ as stated is Proposition~\ref{prop:iterative}.

Now let us prove that $\Omega_{M}$ and $\widehat{\Omega}_{M}$ given in Proposition~\ref{prop:iterative} are well defined and satisfy constraints~\ref{constraint:1} to~\ref{constraint:3}.

    \begin{lemma}\label{lem:mrmOfHm}
        We have that $M(\cH_{M,\mathrm{E}}) = M(\widehat{\cH}_{M,\mathrm{E}}) = 0$ and $\mrm^1(\cH_{M,\mathrm{E}}) =  \mrm^1(\widehat{\cH}_{M,\mathrm{E}}) =0$.
    \end{lemma}
    \begin{proof}
        By definition,
        \begin{equation}
            \mrm^1(\cR_{M-1}) = \delta \mrm^1\big((t\partial_t - \cL) \Mapp{\Omega}{M-1} \big)+ \mrm^1 \big( \{ \Mapp{\Phi}{M-1},\Mapp{\Omega}{M-1} \} \big).
        \end{equation}
       By hypothesis $\Mapp{\Omega}{M-1}$ satisfies constraint~\ref{constraint:1}, hence $\mrm^1(\Mapp{\Omega}{M-1} )= (0,0)$. Then, by the definition~\cref{def:Phiapp} of the operator $\Phi$ and the choice of $\Mapp{R}{M-1}, \Mapp{\alpha}{M-1}$, it is not hard to check that
        \begin{equation}
             \mrm^1 \big( \{ \Mapp{\Phi}{M-1},\Mapp{\Omega}{M-1} \} \big)=(0,0).
        \end{equation}
        Note that this identity is exactly what we used to define $R(t),\alpha(t)$ in Lemma \ref{lem:alpha'gives_moments}. Thus, we conclude that
        \begin{equation}
            \mrm^1(\cR_{M-1}) = (0,0)
        \end{equation}
By the separation of orders in $\eps,\delta$, the statement of the lemma follows by the identity above.   \end{proof}
Now, we check that the Eulerian part of the remainders are indeed in the kernel of $\Lambda.$
    \begin{lemma}
    \label{lem:Hmodd}
        We have that $\cH_{M,\mathrm{E}},\widehat{\cH}_{M,\mathrm{E}} \in (\mathrm{Ker} \Lambda)^\perp \cap \cZ$ and are odd in $\xi_2$.
    \end{lemma}
    \begin{proof}
        From the definition of the operators $\Phi,\widehat{\Phi}$ in \cref{def:Xieps,def:Thetaeps} and Lemma \ref{lem:adjoints}, we know that the only term that does not respect the even symmetry in $\xi_2$ is the one involving $\Re(\cB_{\eps,R}(\Psi,\widehat{\Psi}))$. However, this term vanishishes on the purely Eulerian part of the approximate solution thanks to \cref{eq:Reeven}. Then,  when we set $\delta=0$ to define $\cH_{M,\mathrm{E}},\widehat{\cH}_{M,\mathrm{E}}$ we see that we need to only look at $\Mapp{\Phi}{M-1}, \Mapp{\widehat{\Phi}}{M-1}$ in \cref{def:Phiapp} when computed on $\Mappp{\Psi}{M-1}{E}, \Mappp{\widehat{\Psi}}{M-1}{E}$, and then take the Poisson bracket with $\Mappp{\Omega}{M-1}{E}, \Mappp{\widehat{\Omega}}{M-1}{E}$. Since by the constraint  \cref{constraint:2} the Eulerian part is even in $\xi_2$, we get that $\cH_{M,\mathrm{E}},\widehat{\cH}_{M,\mathrm{E}}$ are odd in $\xi_2$. This automatically tells us that $\cP_0(\cH_{M,\mathrm{E}})=\cP_0(\widehat{\cH}_{M,\mathrm{E}})=0$ and we can conclude that $\cH_{M,\mathrm{E}},\widehat{\cH}_{M,\mathrm{E}} \in (\mathrm{Ker} \Lambda)^\perp$ by using Lemma \ref{lem:mrmOfHm}. The fact that they are in $\cZ$ is a trivial consequence of properties of the approximate solution.    \end{proof}
    By combining Lemmas~\ref{lem:mrmOfHm} and \ref{lem:Hmodd} with the definitions of $\Omega_M$ and $\widehat{\Omega}_M$, the properties of the operators $\Lambda$ and $\cL$ in Proposition \ref{prop:Lambda} and Lemma \ref{lem:Linvert}, we conclude that $\Omega_M$ and $\widehat{\Omega}_M$ satisfy constraint~\ref{constraint:1} to~\ref{constraint:3}. We now compute the remainder.
\begin{lemma} With this choice of $\Omega_{M+1}$ and $\widehat{\Omega}_{M+1}$ one has that
    \begin{equation}
        \cR_{M} = \cO_\cZ(\eps^{M+1} + \eps^2 \delta^2), \qquad \widehat{\cR}_{M} = \cO_\cZ(\eps^{M+1} + \eps^2\delta^2).
    \end{equation}
\end{lemma}
\begin{proof}
By the definition of $\cR_M$, we have
    \begin{align}
        \cR_{M} & =\delta(t\de_t-\cL)\Mapp{\Omega}{M}+\{\Mapp{\Phi}{M},\Mapp{\Omega}{M}\} \\
        & = \cR_{M-1} + \delta(t\de_t-\cL) (\eps^{M}\Omega_{M}) + \{\Mapp{\Phi}{M},\Mapp{\Omega}{M}\} - \{\Mapp{\Phi}{M-1},\Mapp{\Omega}{M-1}\} \\
        & = \cR_{M-1} + \eps^{M}\Big(\delta(t\de_t+\frac{M}{2}-\cL) \Omega_{M} + \{\Mapp{\Phi}{M},\Omega_{M}\}\Big) + \{\Mapp{\Phi}{M}-\Mapp{\Phi}{M-1},\Mapp{\Omega}{M-1}\}.
    \end{align}
    We now compute term by term, only at the required precision. By hypothesis, \begin{equation}
        \cR_{M-1} = \eps^{M} \cH_{M} +\cO_\cZ(\eps^{M+1} + \eps^2\delta^2).
    \end{equation} 
    Using Proposition~\ref{prop:small_operator}, we have that 
    \begin{equation}
        \eps^{M} \{\Mapp{\Phi}{M},\Omega_{M}\} = \eps^{M} \{ \Mapp{\Psi}{M} , \Omega_{M} \} + \cO_\cZ(\eps^{M+1}) = \eps^{M} \{ \Psi_0 , \Omega_{M} \} + \cO_\cZ(\eps^{M+1}).
    \end{equation}
    Then, using again Proposition~\ref{prop:small_operator}, and Lemma \ref{lem:computeR2}, proceeding as done for $\cR_2$ (see the proof of Lemma \ref{lem:computecR2}), we deduce that 
    \begin{equation}
         \{\Mapp{\Phi}{M}-\Mapp{\Phi}{M-1},\Mapp{\Omega}{M-1}\}=\eps^{M}\{\Psi_M,\Omega_0\}+\cO_{\cZ}(\eps^{M+1}).
    \end{equation}
    Gathering all the previous relations leads to
    \begin{equation}
        \cR_{M} = \eps^{M} \Big( \cH_{M} + \delta \Big(t\partial_t + \frac{M}{2}- \cL\Big)\Omega_{M} + \Lambda \Omega_{M}\Big) + \cO_{\cZ}(\eps^{M+1}+\eps^2\delta^2),
    \end{equation}
    which is the expression that allows us to determine that the choice made in the statement of Proposition~\ref{prop:iterative} is correct, or if one is searching for such a choice, allows one to find it. Remarkably, this equation does not depend on $\widehat{\Omega}_{M}$, meaning that we can choose $\Omega_{M}$ independently of the choice of $\widehat{\Omega}_{M}$. The main Eulerian part is killed by the choice of $\Omega_{M,\mathrm{E}}$ as
    \begin{equation}
        \cH_{M,\mathrm{E}} + \Lambda (\Omega_{M,\mathrm{E}}+\Omega_{M,\mathrm{E}_r}) = 0.
    \end{equation}
    Then, since $$\eps^{M}\delta^2 (t \partial_t + \frac{M}{2} - \cL) \Omega_{M,\mathrm{NS}} = \cO_\cZ(\eps^2 \delta^2),$$ and since $\partial_t \Omega_{M,\mathrm{E}} = 0$, we only need to check that
    \begin{equation}\label{eq:27}
        \cH_{M,\mathrm{NS}}  + \left(\frac{M}{2} - \cL\right) (\Omega_{M,\mathrm{E}}+\Omega_{M,\mathrm{E}_r}) + \Lambda\Omega_{M,\mathrm{NS}} = 0.
    \end{equation}
By construction, we have that
\begin{equation}
    \cH_{M,\mathrm{NS}} + \left(\frac{M}{2} - \cL\right)\Omega_{M,\mathrm{E}_r} = (1-\cP_0)\cH_{M,\mathrm{NS}}
\end{equation}
and thus relation~\cref{eq:27} holds true.
\end{proof}

\subsection*{Conclusion of the proof of Proposition~\ref{prop:construction}.} 
\begin{proof}
    We have constructed the $M=0$ order functions $\Omega_0,\widehat{\Omega}_0$ (and actually up to the $M=2$ order), so there only remains to apply Proposition~\ref{prop:iterative} iteratively.
\end{proof}

\subsection{First non trivial corrections}\label{sec:first_corrections}

Although it is not necessary to compute the third order $\Omega_3,\widehat{\Omega}_3$ maps, as mentioned in Section~\ref{sec:corrections}, the third order can be the leading order correction; hence, let us give it explicitly.

Mostly due to the fact that the zero-th and first order are trivial, the computation of the third order is actually very similar to the computations of the second order, although estimating all the error terms is slightly more complicated. One can also directly apply Proposition~\ref{prop:construction} to obtain the expression of the third-order term $\Omega_3$ and $\widehat{\Omega}_3$, but to do so, one has to compute the $\cO_\cZ(\eps^3)$ term in $\cR_2$ and $\widehat{\cR}_2$. Going back to the computation of $\cR_2$ and making finer use of the properties of $\Omega_2$ and the operators $\cQ_{\eps,\cdot}$ and $\cT_{\eps,\cdot}$ one can show that $\cR_2$ contains no other term of order $\cO_\cZ(\eps^3)$ than the term coming from $\cR_0$, namely
    \begin{equation}
        \cH_3 := \frac{c_{3,0}}{R^3} ({\rm S}_{3,3} + \gamma) \{ \Re(\xi^3), G\} , \qquad 
    \widehat{\cH}_3 := \begin{cases}
        3\frac{c_{3,0}}{R^3}  \{ \Re(\xi^3), G\} & \text{ if } N = 3 \\
         0 & \text{ if } N \neq 3.
    \end{cases}
    \end{equation}
Therefore, Proposition~\ref{prop:construction} gives that 
\begin{equation}
    \Omega_{3,\mathrm{E}} := -\Lambda^{-1} \cH_3, \qquad \widehat{\Omega}_{3,\mathrm{E}} := -\frac{1}{\gamma}\Lambda^{-1} \widehat{\cH}_3
\end{equation}
and
\begin{equation}
    \Omega_{3,\mathrm{NS}} := -\Lambda^{-1} \left(\frac{3}{2}-\cL\right)\Omega_{3,\mathrm{E}}, \qquad \widehat{\Omega}_{3,\mathrm{E}} :=  -\frac{1}{\gamma}\Lambda^{-1} \left(\frac{3}{2}-\cL\right)\widehat{\Omega}_{3,\mathrm{E}}.
\end{equation}
Things get more complicated at the fourth order, as new non trivial error terms at order $\eps^4$ come from the nonlinear interactions of the second order term with itself.

One can now wonder what is the first non trivial term for the associated approximate radius $\Mapp{R}{M}$ and angular velocity $(\Mapp{\alpha}{M})'$ of the rotating frame.
\begin{proposition}\label{prop:deviation}
   There exists $r_6 \in \R$ and $\alpha_4 \in \R$, such that for any $M \ge 5$, and any $t \le C \delta^{-1} T_\adv$,
    \begin{equation}
        \Mapp{R}{M}(t) = \mathsf{r}(1 + r_6 \eps^6(t) + \cO(\eps^7(t)))
    \end{equation}
    and
    \begin{equation}
        (\Mapp{\alpha}{M})'(t) = \mathsf{a}\big(1 + 
        \alpha_4\eps^4(t)  + \cO(\eps^5(t) + \delta \eps^4(t))\big),
    \end{equation}
    where we recall that $\mathsf{a}=\frac{\Gamma}{4\pi \mathsf{r}^2}(N-1+2\gamma)$ is the angular velocity of the point-vortex polygonal vortex crystal configuration defined in~\cref{eq:zjt}. The constant $r_6$ and $\alpha_4$ are given by:
    \begin{equation}
        r_6 = -\frac{d}{3\mathsf{r}^4}\Big( 3c_{3,0} \big( {\rm S}_{3,3}  + \gamma\big) \mrm_2^{2}(\Omega_{2,\mathrm{NS}}) + c_{3,2} \big( {\rm S}_{3,1} \mrm_2^2( \Omega_{2,\mathrm{NS}}) + \gamma \mrm_2^2( \widehat{\Omega}_{2,\mathrm{NS}})\big) \Big)
    \end{equation}
    and
    \begin{equation}
      \label{def:alha4}  \alpha_4 =\frac{1}{\mathsf{r}^2c_{1,0}(S_{1,1}+\gamma)} \Big(3c_{3,0}({\rm S}_{3,3}+\gamma) \mrm_1^2(\Omega_2) + c_{3,2} ({\rm S}_{3,1} \mrm_1^2(\Omega_2)+\gamma \mrm_1^2(\widehat{\Omega}_2)\big) \Big).
    \end{equation}
\end{proposition}
\begin{proof}
        Plugging Lemma~\ref{lem:expansionBtot} into the definition~\cref{def:Rapp} of $\Mapp{R}{M}$, we have that
    \begin{equation}
        t (\Mapp{R}{M})'(t) = 3\mathsf{r}^4\frac{r_6\eps^6}{(\Mapp{R}{M})^3}  + \cO(\eps^7).
    \end{equation}
    Denoting by $T$ the maximal time for which $|\Mapp{R}{M}(t) - r| \le \eps$, we have that for $t \le T$,
    \begin{equation}
        t (\Mapp{R}{M})'(t) = 3r_6\eps^6  + \cO(\eps^7).
    \end{equation}
    Recalling that $\eps^2(t) = \frac{\nu t}{d^2}$, by dividing by $t$ and integrating,
    \begin{equation}
        (\Mapp{R}{M})(t) = \mathsf{r}(1 + r_6\eps^6  + \cO(\eps^7)).
    \end{equation}
    The bootstrap assumption is therefore satisfied for as long as $\eps$ is small enough, namely for $t \le c \delta^{-1} T_\adv$ for some small constant $c$. 
    Plugging this into Lemma~\ref{lem:expansionBtot} and the definition~\cref{def:alphaapp} of $\Mapp{\alpha}{M}$ gives that
    \begin{equation}
    \label{eq:alphaappMproof}
        \delta t(\Mapp{\alpha}{M})'(t) = d\frac{\eps^2}{\mathsf{r}^2} c_{1,0} (S_{1,1}+\gamma)\Big( 1 + \alpha_4 \eps^6 
        + \cO(\eps^7 + \delta \eps^6)\Big).
    \end{equation}
    Recalling that $\eps^2 = \delta \frac{t}{T_\adv}$, we conclude that
\begin{equation}
    (\Mapp{\alpha}{M})'(t) = \mathsf{a}(1 + \alpha_4 \eps^4 + \cO(\eps^5 + \delta \eps^4))
\end{equation}
by noticing that $\mathsf{a} =\frac{\Gamma}{4\pi \mathsf{r}^2}(N-1+2\gamma) = \frac{d}{T_\adv \mathsf{r}^2} c_{1,0}(S_{1,1}+\gamma).$ 
\end{proof}

\begin{remark}\label{rem:delta23}
    Please note that, by integrating in time the formula in Proposition \ref{prop:deviation} we get that
    $|\Mapp{\alpha}{M}(t) - \mathsf{a}t| \lesssim \mathsf{a} t \eps^4(t)$. In particular, defining
    \begin{equation}
        \zeta^{(M)}_{\app,j}(t) = \Mapp{R}{M}(t)\Big(\cos\Big(\frac{2\pi j}{N}+\Mapp{\alpha}{M}(t)\Big) ,\sin\Big(\frac{2\pi j}{N}+\Mapp{\alpha}{M}(t) \Big) \Big),
    \end{equation}
    we have that
    \begin{equation}
        \frac{1}{\mathsf{r}}|\zeta^{(M)}_{\app,j}(t) - z_j(t)| \lesssim \frac{1}{\mathsf{r}}|\Mapp{R}{M}(t)-\mathsf{r}|+ | \Mapp{\alpha}{M}(t) - \mathsf{a}t| \lesssim  \eps^6(t)+ \mathsf{a}t \eps^4(t).
    \end{equation}
    Then, since $\eps^6(t)\ll \mathsf{a}t\eps^4(t)$ whenever $\nu/d^2\ll \Gamma/\mathsf{r}^2$, for $\delta$ sufficiently small we arrive at
        \begin{equation}
        \frac{1}{\mathsf{r}}|\zeta^{(M)}_{\app,j}(t) - z_j(t)| =\cO\big(\frac{t}{T_{\rm adv}} \eps^4(t)\big),
    \end{equation}
where we also used that $\mathsf{a}$ is proportional to $1/T_{\rm adv}$.
    The error above is of order $\cO(1)$ as soon as $t$ becomes of order $\cO(T_\adv \delta^{-2/3})$. This means that although our solution remains a sum of concentrated vortices for a time of order $T_\adv\delta^{-\sigma}$, its angular position is completely out of phase with respect to the point vortex dynamics after a time $\cO(T_\adv \delta^{-2/3})$.

    Moreover, the constant $\alpha_4$ can be easily computed numerically\footnote{Besides  the definition of the constants, we only need to compute numerically $\mrm^2_1(\Omega_{2,\rm E})$ as in Proposition \ref{prop:Om2}. To define $\Omega_{2,\rm E}=-\Lambda^{-1}\cH_2$ we make use of the considerations leading to Proposition \ref{prop:explicitsol}. Namely, looking at $\cH_2$ in Lemma \ref{lem:computecR2}, it is enough to solve $(A+\Delta^{-1})f_2=\Re(\xi^2)$. To do so, we discretize the integral operator $T_2$ in \cref{def:Tn} and cut the integral in its definition, invert the matrix associated to ${\rm Id}-T_2$ and obtain an approximation for $f_2$.}. When $N=2$ and $\gamma=0$, we observe that $\alpha_4\approx 8.7362$ which is in agreement with the corresponding correction to the vertical speed for a dipole computed in \cite{HaNaFu2018} (and in agreement as well with the numerical value found in \cite{dolce2024long}). This is not surprising because the second order correction for the dipole and two equal co-rotating vortices are exactly the same except for a sign, meaning that the dipole slows down whereas the likely signed vortex pair \emph{increase} their angular speed. In the case $N=6$, $\gamma=10$ (and $\gamma=-10$) in Figure \ref{fig:deformation_sim} we get that $\alpha_4\approx 513.496$ (and $\alpha_4=-1011.1373$). 
    
For the special value of $\gamma=\gamma_N^*$ in \cref{eq:gammaN*} and $N>2$, we know that $\Omega_2=\widehat{\Omega}_2=0$ and therefore the correction to the Helmholtz-Kirchhoff angular speed starts at order $\eps^5$, meaning that $\alpha_4=0$ and the next coefficient $\alpha_5$ can be explicitly defined with the formulas provided in Lemma \ref{lem:expansionBtot}.

    Finally, when $N\to \infty$, we know that $c_{n,k}=\cO(L_N^n)=\cO(N^{-n})$ whereas ${\rm S}_{3,j}=\cO(N^2)$ and ${\rm S}_{1,1}=\cO(N)$, meaning that $\alpha_4=\cO(N^{-1})$ in view of Remark \ref{rem:Om2Ninfinity}. For general $N\geq 3$ and $\gamma=0$, see Figure \ref{fig:alpha4_combined}.
    \begin{figure}[htbp]
    \centering
    \begin{minipage}{0.34\textwidth}
        \centering
        \small 
        \vspace{-23pt}
        \begin{tabular}{ccc}
            \toprule
            $N$ & $d$ & $\alpha_4$ \\
            \midrule
            3 & 1 & 7.7655 \\
            4 & 1 & 2.9121 \\
            5 & 1 & 0 \\
            6 & 1 & 4.8535 \\
            7 & 0.8678 & 13.2101 \\
            8 & 0.7654 & 20.9846 \\
            9 & 0.6840 & 27.2031 \\
            10 & 0.6180 & 31.8650 \\
            \bottomrule
        \end{tabular}
    \end{minipage}
    \hfill
    \begin{minipage}{0.65\textwidth}
        \centering
        \includegraphics[width=.75\textwidth, valign=c]{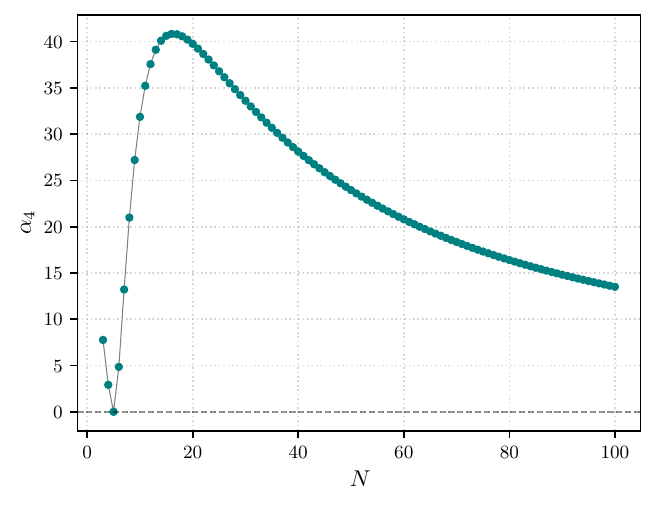}
    \end{minipage}
    \caption{Values of the distance parameter $d$ in \cref{def:d} and $\alpha_4$ in \cref{def:alha4} computed for $N=3,\dots,10$ (left) and plot of $\alpha_4$ as a function of $N$ (right). For $N=5$, we have $\gamma_5^*=0$ and note the behavior $\alpha_4=\cO(N^{-1})$ for large values of $N$.}
    \label{fig:alpha4_combined}
\end{figure}
\end{remark}\label{rem:direction_of_deformation}

\FloatBarrier
Finally, let us conclude this section with the following remark on the direction of deformation.
\begin{remark}
    As observed in Figure~\ref{fig:deformation_sim}, we expect that the direction of deformation changes depending on the sign of $\gamma - \gamma_N^* = \gamma + S_{2,2}$. Gathering the expressions of $\Omega_2$ and $\cH_2$ obtained in Section~\ref{sec:second_order}, one can compute (similarly to what was done in \cite{Gallay2011,dolce2024long,donati2025fast}) that there exists a map $w : \R_+ \to \R_+$ such that
    \begin{equation}
        \Omega_2(\xi) = \frac{c_{2,0}}{R^2} (S_{2,2} + \gamma) w(r)\cos(2\theta),
    \end{equation}
    where $(r,\theta)$ denotes the polar coordinates of $\xi$. With that expression, we see that $\Omega_2$ changes sign with the quantity $\gamma - \gamma_N^*$ and that it results in $\Omega_0 + \eps^2\Omega_2$ changing its direction of deformation with the sign of $\Omega_2$ (since it is a $\pi$-periodic map in $\theta$).
\end{remark}
\subsection{The functional relationship}\label{sec:func_rel}
A key point to justify the validity of the approximate solution, is the construction of an approximate functional relationship between vorticity and streamfunction for the Eulerian part of the solution, arising when setting $\delta=0$ in the expansion done in the previous sections. In particular, this will be a crucial property that allow us to apply the Arnold's energy method in the nonlinear problem. In order for that Eulerian part of the solution to make sense in itself, we consider in this section $\eps$ to be a parameter independent of $\delta$, and to define $\Mapp{\Omega}{M}$ and $\Mapp{\widehat{\Omega}}{M}$ through formulas~\cref{def:Omapp} as a series in this parameter $\eps$.

First of all, by construction (constraint~\ref{constraint:5}), we know that 
\begin{align}
\label{def:OmE}
    &\lim_{\delta \to 0}\Mapp{\Omega}{M}=\Mappp{\Omega}{M}{E}=\Omega_0+\sum_{k=2}^M\eps^2\Omega_{k,E}\\
    \label{def:OmhE}&\lim_{\delta \to 0}\Mapp{\widehat \Omega}{M}=\Mappp{\widehat \Omega}{M}{E}=\Omega_0+\sum_{k=2}^M\eps^2\widehat{\Omega}_{k,E}.
\end{align}
Then, we recall that by \cref{def:Xieps,def:Thetaeps} that
\begin{align}
    \Mapp{\Phi}{M}=\,&(I+\cQ_{\eps,\Mapp{R}{M}})\Mapp{\Psi}{M}+\gamma\cT_{\eps,\Mapp{R}{M}}\Mapp{\widehat \Psi}{M}+\xi_2\Re(\cB_{\eps,\Mapp{R}{M}}(\Mapp{\Psi}{M},\Mapp{\widehat{\Psi}}{M}))\\
    &-\frac{d \eps}{2 \Mapp{R}{M}}\Im(\cB_{\eps,\Mapp{R}{M}}(\Mapp{\Psi}{M},\Mapp{\widehat{\Psi}}{M}))|\cT_{\eps,\Mapp{R}{M}}\xi|^2\\
    \Mapp{\widehat \Phi}{M}=\,&\gamma \Mapp{\widehat{\Psi}}{M}+\cV_{\eps,\Mapp{R}{M}}\Mapp{\Psi}{M}-\frac{d \eps}{2 \Mapp{R}{M}}\Im(\cB_{\eps,\Mapp{R}{M}}(\Mapp{\Psi}{M},\Mapp{\widehat{\Psi}}{M}))|\xi|^2.
\end{align}
In the expression above, we now keep only the Eulerian part of the approximate solution but we keep the radius $\Mapp{R}{M}$ fixed. By Lemma \ref{lem:eveneven} and  the constraint \cref{constraint:2}, note that
\begin{equation}
    \Re(\cB_{\eps,\Mapp{R}{M}}(\Mappp{\Psi}{M}{E},\Mappp{\widehat{\Psi}}{M}{E}))=0.
\end{equation}
Hence, by also shifting the streamfunctions by a constant (that will not be seen in the Poisson brackets), we  define
\begin{align}
\label{def:PhiappE}\Mappp{\Phi}{M}{E}
    &\coloneqq(I+\cQ_{\eps,\Mapp{R}{M}})\Mappp{\Psi}{M}{E}+\gamma\cT_{\eps,\Mapp{R}{M}}\Mappp{\widehat \Psi}{M}{E}\\
   &\qquad -\frac{\eps d}{2 \Mapp{R}{M}} \Im(\cB_{\eps,\Mapp{R}{M}}(\Mappp{\Psi}{M}{E},\Mappp{\widehat \Psi}{M}{E}))|\cT_{\eps,\Mapp{R}{M}}\xi|^2-\frac{1}{2\pi}\log(\eps)\\
\label{def:PhihappE}\Mappp{\widehat \Phi}{M}{E}
&\coloneqq\gamma\Mappp{\widehat \Psi}{M}{E}+\cV_{\eps,\Mapp{R}{M}}\Mappp{\Psi}{M}{E}-\frac{\eps d}{2 \Mapp{R}{M}}\Im(\cB_{\eps,\Mapp{R}{M}}(\Mappp{\Psi}{M}{E},\Mappp{\widehat \Psi}{M}{E}))|\xi|^2-\frac{1}{2\pi}\log(\eps).
\end{align}
Then, by following the same proof of Proposition \ref{prop:iterative} with $\delta$ set to $0$ in the definition of the errors $\cR_{M},\widehat{\cR}_M$, it is not hard to show inductively that 
\begin{align}
    \cR_{M,{\rm E}}\coloneqq\{\Mappp{\Phi}{M}{E},\Mappp{\Omega}{M}{E}\}=\cO_{\cZ}(\eps^{M+1}), \qquad     \widehat{\cR}_{M,{\rm E}}\coloneqq\{\Mappp{\widehat \Phi}{M}{E},\Mappp{\widehat \Omega}{M}{E}\}=\cO_{\cZ}(\eps^{M+1}).
\end{align}
Hence, from the identities above we are very close to being an exact solution of the Euler equation in a moving frame, for which one can expect that there is (at least locally) a functional relationship between $\Mappp{\Omega}{M}{E}$ and $\Mappp{\Phi}{M}{E}$.
\begin{remark}
Once the even symmetry in $\xi_2$ is guaranteed, the modulation parameter $\Mapp{R}{M}$ is actually ``free", in the sense that  $\mrm_{1}^1(\cR_{M,{\rm E}})=0$ being $\cR_{M,{\rm E}}$ odd in $\xi_2$. Thus, $\Mapp{R}{M}$  enters in the construction of the Eulerian part only as a multiplicative constant. This reflects a scaling symmetry of the Eulerian configuration, where a change in the radius corresponds to a scaling of the angular speed.
\end{remark}

 Unfortunately, it is not possible to justify an exact functional relationship up to errors like $\Mappp{\Phi}{M}{E}=-F(\Mappp{\Omega}{M}{E})+\cO_{\cS_*}(\eps^{M+1})$. Nonetheless,  in \cite{dolce2024long} the authors introduced  a way of defining an approximate functional relationship that suffices for the scope of the result. The key idea is to define everything as a Taylor expansion around the Lamb-Oseen vortex $\Omega_0$. Indeed, at leading order $M=0$ we know that $\Mappp{\Omega}{0}{E} =\Mappp{\widehat{\Omega}}{0}{E} =\Omega_0$ and $\Mappp{\Phi}{0}{E}=\gamma^{-1}\Mappp{\widehat \Phi}{0}{E}=\Psi_0$ where $\Omega_0,\Psi_0$ are given in \cref{def:G}, \cref{def:UG} respectively. Moreover, by the definition in \cref{def:G} and \cref{def:UG} we also deduce that
\begin{equation}
    \Psi_0+F_0(\Omega_0)=0, \qquad \gamma\Psi_0+\widehat{F}_0(\Omega_0)=0
\end{equation}
with
\begin{equation}
\label{def:F0}
    F_0(s)=\frac{1}{4\pi}\Big(\gamma_E-{\rm Ein}(\log\big(\frac{1}{4\pi s}\big)\Big), \quad \widehat{F}_0(s)=\gamma F_0(s), \qquad 0< s\leq \frac{1}{4\pi},
\end{equation}
Note that $A(|\xi|)=F'_0(\Omega_0(|\xi|))$ with $A$ being defined in \cref{def:A}. The functional relationship at leading order is then our starting point to construct an approximate relationship inductively at order $M$. To accomplish this task, we proceed as in \cite{dolce2024long} and we introduce the following definition.
\begin{definition}
    We say that a smooth function $F:(0,+\infty)\to \RR$ belongs to the  class $\mathsf{K}$ if $F(\Omega_0)\in \cS_*(\RR^2)$.
\end{definition}
This class is introduced because it contains all the necessary information we can hope to propagate from $F_0$, since we clearly have $F_0\in \mathsf{K}$.
Then, a useful property that we need for functions in this class is contained in the following lemma, whose proof can be found in \cite[Lemma 3.19]{dolce2024long}.
\begin{lemma}
\label{lem:FK}
    Let $F:(0,+\infty)\to \RR$ be in the class $\mathsf{K}$. Then, for all $k\in \NN$, the $k$-th order derivative of $F$ has the property that $F^{(k)}(\Omega_0)\Omega_0^k\in\cS_*(\RR^2)$. Moreover, if $\cH\in \cS_{*}(\RR^2)$, there exists an $F\in \mathsf{K}$ such that $F(\Omega_0)=\cH$ if and only if $\{\cH,\Omega_0\}=0$, namely $\cH$ is radially symmetric.
\end{lemma}
Note that the solvability of $F(\Omega_0)=\cH$ is not contained in \cite[Lemma 3.19]{dolce2024long}, but this is a straightforward consequence of the change of variables $s=\Omega_0(r)$. In addition, since $\cH$ may grow polynomially at infinity, $F$ can have logarithmic singularity at $s=0$ (which is the case for $F_0$).

We finally introduce the Taylor polynomial up to order $M$ associated to a given smooth function $f(\eps,s):(0,1)\times \RR\to \RR$ as 
\begin{equation}
    \Pi_M f=\sum_{k=0}^M\frac{\eps^k}{k!}\frac{\dd^k f}{\dd \eps^k}\bigg|_{\eps=0}.
\end{equation}
This allow us to find a functional relations in the form $F^{(M)}=F_0+\eps^2F_2+\dots \eps^MF_M$, with $F_i\in \mathsf{K}$, for which the functional relationship holds for the associated Taylor polynomial at order $M$, safely computed at $\Omega_0$. Note that it is important to compute everything at $\Omega_0$ since it is not true that $\Omega_{\app,E}$ is always positive in the whole domain (but this is true locally), which is a condition required to work in class $\mathsf{K}$. We now have all the ingredients to conclude the construction of the functional relationship and obtain the desired properties.
\begin{proposition}
\label{prop:functional}
    Given any integer $M\geq2$, let $\Mappp{\Omega}{M}{E},\Mappp{\widehat{\Omega}}{M}{E} $ and $\Mappp{\Phi}{M}{E},\Mappp{\widehat \Phi}{M}{E}$ be given as in \cref{def:OmE}, \cref{def:OmhE} and \cref{def:PhiappE}, \cref{def:PhihappE} respectively. Then, there exists $\{F_i,\widehat{F}_i\}_{i=2}^M\in \mathsf{K}$ such that for $F^{(M)}=F_0+\sum_{k=2}^M\eps^kF_k$ and $\widehat{F}^{(M)}=\widehat{F}_0+\sum_{k=2}^M\eps^k\widehat{F}_k$ one has 
    \begin{align}
\Pi_{M}\big(\Mappp{\Phi}{M}{E}+F^{(M)}(\Mappp{\Omega}{M}{E})\big)=0, \qquad \Pi_{M}\big(\Mappp{\widehat \Phi}{M}{E}+\widehat{F}(\Mappp{\widehat \Omega}{M}{E})\big)=0
    \end{align}
    Moreover, there exists $\sigma_1\in (0,1)$, $C>0$ and $\widetilde{N}\in \NN$ depending on $M$ such that, for $\eps>0$ sufficiently small, the functions 
    \begin{equation}
        \Theta^{(M)}\coloneqq \Mappp{\Phi}{M}{E}+F^{(M)}(\Mappp{\Omega}{M}{E}), \qquad\widehat{\Theta}^{(M)}\coloneqq \Mappp{\widehat\Phi}{M}{E}+{\widehat F}^{(M)}(\Mappp{\widehat \Omega}{M}{E})
    \end{equation}
    are well defined for all $|\xi|\leq 2\eps^{-\sigma_1}$ and satisfy the bound
    \begin{align}
    \label{bd:keyapp}
        \big|\nabla \Theta^{(M)}(\xi)\big|+ \big|\nabla \widehat{\Theta}^{(M)}(\xi)\big|\leq C\eps^{M+1}(1+|\xi|)^{\widetilde{N}}, \qquad \text{for all }|\xi|\leq 2\eps^{-\sigma_1}.
    \end{align}
\end{proposition}
\begin{proof}
    The proof can be reduced to the same arguments done in \cite[Proof of Proposition 3.20]{dolce2024long}. Indeed, first of all we observe that the construction for the ``hats" $\widehat{\cdot}$ variables is the same and we avoid detailing this case. Then, we know that by Proposition \ref{prop:iterative} that we can expand
    \begin{equation}
\Mappp{\Omega}{M+1}{E}=\Mappp{\Omega}{M}{E}+\eps^{M+1}{\Omega}_{M+1,{\rm E}},
    \end{equation}
    Thus, in view of the definitions in \cref{def:PhiappE} and \cref{def:PhihappE}, we can expand the Eulerian streamfunction as 
    \begin{equation}
        \Mappp{\Phi}{M+1}{E}=\Mappp{\Phi}{M}{E}+\eps^{M+1}\Phi_{M+1,{\rm E}}.
    \end{equation}
  Dropping from now on all the subscripts ``$\rm app, E$'', by construction  we know that 
    \begin{equation}
        \{\Phi^{(M+1)},\Omega^{(M+1)}\}=\cO_{\cZ}(\eps^{M+2}).
    \end{equation}
    We can then proceed by induction on $M$ as done in \cite[Proof of Proposition 3.20]{dolce2024long}. The case $M=0$ is covered, $M=1$ is trivial and we can assume by induction that we are given $$F^{(M)}=F_0+\sum_{k=2}^M\eps^k F_i\in \mathsf{K}$$ for some fixed $M\geq 1$. To construct the function $F^{(M+1)}=F^{(M)}+\eps^{M+1}F_{M+1}$, we take $\Omega^{(M+1)}$ and $\Phi^{(M+1)}$  as defined before. We can then run the same arguments in \cite{dolce2024long} (just based on Poisson bracket's identities and properties of Taylor polynomials) to deduce that there exists a function $\cH_{M+1}\in \cS_{*}(\RR^2)$ such that the following holds true:
    \begin{equation}
        \Pi_{M+1}(\Phi^{(M+1)}+F^{(M+1)}(\Omega^{(M+1)}))=\eps^{M+1}\big(\underbrace{\cH_{M+1}+\Phi_{M+1}+F_0'(\Omega_0)\Omega_{M+1}}_{\cE_{M+1}}+F_{M+1}(\Omega_0)\big).
    \end{equation}
    and
    \begin{align}
    \label{eq:solvability}
        0=\{\cH_{M+1},\Omega_0\}+\{\Phi_{M+1},\Omega_0\}+\{\Phi_0,\Omega_{M+1}\}=\{\cE_{M+1},\Omega_0\},
    \end{align}
    where in the last identity we used that $\Phi_0=-F_0(\Omega_0)$ and $\{F_0'(\Omega_0)\Omega_{M+1},\Omega_0\}=\{\Omega_{M+1},F_0(\Omega_0)\}.$
    The identity \cref{eq:solvability} guarantees that we satisfy the solvability condition in Lemma \ref{lem:FK} to define $F_{M+1}\in\mathsf{K}$ such that $F_{M+1}(\Omega_0)=-\cE_{M+1}$, so that we obtain $\Pi_{M+1}(\Phi^{(M+1)}+F^{(M+1)}(\Omega^{(M+1)}))=0$ as desired.

    The proof of the property \cref{bd:keyapp} is the same of \cite{dolce2024long}. This directly follows by a Taylor expansion in $\eps$ of $\Theta^{(M)}$ at $\Omega_0$, since by definition we know that $\Pi_M \Theta^{(M)}=0$. Thus, cutting the Taylor expansion with the integral remainder it is not hard to prove \cref{bd:keyapp}.
\end{proof}
\subsection{Eulerian approximation and total angular momentum}\label{sec:angmom}
To end our analysis of the approximate solution, we need to deduce a property of the Eulerian part of the approximate solution related to evolution of the {total angular momentum}, which is defined through the operator \cref{def:L}.
First of all, in analogy with \cref{eq:omSS} we define 
\begin{align}
\label{def:omappM}
\omega_{\rm app}^{(M)}(x,t)=\frac{\Gamma}{\nu t}\Big(\sum_{j=1}^N\Omega_{\rm app}^{(M)}(\xi_j(x),t)+\gamma\widehat{\Omega}_{\rm app}^{(M)}(\widehat{\xi}(x),t)\Big).
\end{align}
Then, by the definition of $\cR_{M},\widehat{\cR}_M$ in \eqref{def:RM}, following the derivation of the equations for $\Omega,\widehat{\Omega}$ in Lemma \ref{lem:derivationequation} it is not hard to see that $\omega_{\rm app}^{(M)}$ solves the following forced $2D$ Navier-Stokes equations
\begin{align}
\label{eq:NSappx}
    \begin{cases}
        \de_t\omega_{\rm app}^{(M)}+\{\psi_{\rm app}^{(M)}-\frac12(\alpha_{\rm app}^{(M)})'|x|^2,\omega_{\rm app}^{(M)}\}=\nu \Delta\omega+f_M,\\
    \omega_{\rm app}^{(M)}|_{t=0}=\omega^{in}
    \end{cases},
\end{align}
where $\omega^{in}$ as in \cref{eq:omegain} and
\begin{align}\label{def:fM}
    f_M(x,t)=\frac{\Gamma}{\nu t^2}\delta^{-1}\Big(\sum_{j=1}^N\cR_M(\xi_j(x),t)+\gamma\widehat{\cR}_M(\widehat{\xi}(x),t)\Big).
\end{align}
Then, as in the proof of Lemma \ref{lem:NSmom}, a direct computation shows that 
\begin{align}
\label{eq:angmomf}\mathrm{L}(\omega_{\rm app}^{(M)}(t))
=\Gamma N\big(\mathsf{r}^2+4\nu t \big(1+\frac{\gamma}{N}\big)\big)+\int_0^t \mathrm{L}(f_M(\tau))\, \dd \tau.
\end{align}
Using this exact evolution, we are able to deduce the following property that will be crucial in the nonlinear section.
\begin{lemma}
\label{lem:keyLEuler}
    Let $\Omega_{\rm app}^{(M)},\widehat{\Omega}_{\rm app}^{(M)}$ be the ones constructed in Proposition \ref{prop:iterative}. Then
    \begin{align}
    \label{eq:keyEuler}
      \frac{(R_{\rm app}^{(M)})^2-\mathsf{r}^2}{d^2}+\eps^2 \big(\mathrm{L}(\Omega_{\rm app,E}^{(M)}-\Omega_0)+\frac{\gamma}{N}\mathrm{L}(\widehat{\Omega}_{\rm app,E}^{(M)}-\Omega_0)\big)=\cO(\delta^{-1}\eps^{M+3}+\delta \eps^4).
    \end{align}
\end{lemma}
\begin{remark}
Note that, by Proposition~\ref{prop:deviation}, $(R_{\mathrm{app}}^{(M)})^2 = \mathsf{r}^2(1+\cO(\eps^6))$. Moreover, since $\mathrm{L}(g) = \mathrm{L}(\cP_0 g)$ for any function $g$, where $\cP_0$ denotes the angular average, the explicit form of the second and third-order approximations defined in Sections~\ref{sec:second_order} and~\ref{sec:first_corrections} implies that $\mathrm{L}(\cP_0(\Omega^{(M)}_{\mathrm{app,E}}-\Omega_0))$ and $\mathrm{L}(\cP_0(\widehat{\Omega}^{(M)}_{\mathrm{app,E}}-\Omega_0))$ are of order $\cO(\eps^4)$. Thus, the left-hand side of~\cref{eq:keyEuler} consists of terms of order $\cO(\eps^6)$ and higher, indicating exact cancellations at each order up the ones needed to validate \cref{eq:keyEuler}. One could directly verify this condition by hand, but this would be  tedious given the number of terms involved in the iterative construction. Instead, we deduce the identity~\cref{eq:keyEuler} from the exact evolution of the total angular momentum in the 2D forced Navier-Stokes equations and the fact that our approximate radius fixes $\mrm^1_1(\Omega_{\mathrm{app}}^{(M)}) = 0$.
\end{remark}
\begin{proof}
        We omit the superscript $(M)$ in this proof. By \cref{def:omappM} we know that
    \begin{align}
        \mathrm{L}(\omega_{\rm app}(t))=\frac{\Gamma}{\nu t}\Big(\sum_{j=1}^N\int_{\RR^2} |x|^2\Omega_{\rm app}(\xi_j(x),t)\dd x+\gamma \int_{\RR^2}|x|^2\widehat{\Omega}_{\rm app}(\widehat{\xi}(x),t)\Big)\dd x.
    \end{align}
    Then, by the definition of $\xi_j$ in \cref{def:xi_j} and $\widehat{\xi}$ in \cref{def:xi0} we can change variables in each integral as 
    \begin{align}
        x=Q_j\big(\sqrt{\nu t}\xi_j+(R_{\rm app}(t),0)\big), \qquad x=\sqrt{\nu t}\, \widehat{\xi}.
    \end{align}
    Since $|Q_jv|=|v|$ for any vector $v$, the integrals of each term in $\xi_j$ are equal and we therefore rename the dummy variables $\xi_j,\widehat{\xi}$ as $\xi$, to obtain that
    \begin{align}
    \label{eq:Ang1}
        \mathrm{L}(\omega_{\rm app}(t))=\Gamma N \nu t\Big( \int_{\RR^2}|\xi+\frac{1}{\sqrt{\nu t}}(R_{\rm app},0)|^2\Omega_{\rm app}(\xi,t)\dd \xi+\frac{\gamma}{N}\mathrm{L}(\widehat{\Omega}_{\rm app}(t))\Big).
    \end{align}
    Note that
\begin{align}
    \int_{\RR^2}&|\xi+\frac{1}{\sqrt{\nu t}}(R_{\rm app},0)|^2\Omega_{\rm app}(\xi,t)\dd \xi
\\
\label{eq:AngOmapp}
&=\mathrm{L}(\Omega_{\rm app}(t))+2\frac{1}{\sqrt{\nu t}}R_{\rm app}(t)\, \mrm^1_1(\Omega_{\rm app}(t))+\frac{R_{\rm app}^2}{\nu t}\mrM(\Omega_{\rm app}(t)).
\end{align}
Hence, since $\mrM(\Omega_{\rm app})=1$ and $\mrm^1_1(\Omega_{\rm app})=0$, combining the identity above with  \cref{eq:Ang1,eq:angmomf}, we see that
\begin{align}
\label{eq:Ang2}
    \mathsf{r}^2+4\nu t\big(1+\frac{\gamma}{N}\big)+\frac{1}{\Gamma N}\int_0^t \mathrm{L}(f_M)\dd \tau=\nu t \big(\mathrm{L}(\Omega_{\rm app}(t))+\frac{\gamma}{N}\mathrm{L}(\widehat{\Omega}_{\rm app}(t))\big)+R_{\rm app}^2.
\end{align}
 Considering only the Eulerian part, since $\Omega_0=\frac{1}{4\pi}\e^{-|\xi|^2/4}$, we have
\begin{align}
    \mathrm{L}(\Omega_{\rm app,E})=\mathrm{L}(\Omega_0)+\mathrm{L}(\Omega_{\rm app,E}-\Omega_0)=4+\mathrm{L}(\Omega_{\rm app,E}-\Omega_0)
\end{align}
and analogously for $\widehat{\Omega}_{\rm app,E}$.
Thus, the term $4\nu t(1+\gamma/N)$ cancels out with the Lamb-Oseen part of the approximate solution and  we  arrive at 
\begin{align}
    &R_{\rm app}^2-\mathsf{r}^2+(\eps d)^2\big(\mathrm{L}(\Omega_{\rm app,E}-\Omega_0)+\frac{\gamma}{N}\mathrm{L}(\widehat{\Omega}_{\rm app,E}-\Omega_0)\big)\\
    &\qquad=\frac{1}{\Gamma N}\int_0^t\mathrm{L}(f_M)\dd \tau -\delta(\eps d)^2\big(\mathrm{L}( \Omega_{\rm app,NS})+\frac{\gamma}{N}\mathrm{L}( \widehat{\Omega}_{\rm app,NS})\big).
\end{align}
By the definition~\cref{def:fM} of $f_M$, since $\mrM(\cR_{M})=\mrM(\widehat{\cR}_{M})=\mrm_1^1(\cR_{M})=\mrm_1^1(\widehat{\cR}_{M})=0$, following the same computations in \cref{eq:Ang1,eq:AngOmapp} we deduce that 
\begin{align}
    \frac{1}{\Gamma N}\int_0^t\mathrm{L}(f_M)\dd \tau =\nu\delta^{-1}\int_0^t\Big(\mathrm{L}(\cR_M)+\frac{\gamma}{N}\mathrm{L}(\widehat{\cR}_M)\Big)\dd \tau=(\eps d)^2\cO(\delta^{-1}\eps^{M+1}+\delta \eps^2),
\end{align}
where in the last identity we used that $\cR_M,\widehat{\cR}_M=\cO_{\cZ}(\eps^{M+1}+\delta^2\eps^2).$ Having that $\Omega_{\rm app, NS},\widehat{\Omega}_{\rm app,NS}=\cO_{\cZ}(\eps^2),$ we finally get
\begin{align}
    \frac{R_{\rm app}^2-\mathsf{r}^2}{d^2}+\eps^2 \big(\mathrm{L}(\Omega_{\rm app,E}-\Omega_0)+\frac{\gamma}{N}\mathrm{L}(\widehat{\Omega}_{\rm app,E}-\Omega_0)\big)=\cO(\delta^{-1}\eps^{M+3}+\delta \eps^4),
\end{align}
which proves the desired result.
\end{proof}
\section{The nonlinear corrections}
\label{sec:nonlinear}
The goal of this section is to show that the approximate solutions constructed in Section \ref{sec:App} are close to the solutions $\Omega,\widehat{\Omega}$ in \cref{eq:Omega}, \cref{eq:P} with initial data \cref{eq:Init}. With the simplified notation in \cref{def:Xieps,def:Thetaeps,eq:simplified} the equations we need to study are 
\begin{equation}
\label{eq:NL0}
t\de_t\Omega=\cL\Omega-\frac{1}{\delta}\{\Phi_{\eps,R}(\Psi,\widehat{\Psi}),\Omega\},\qquad t\de_t\widehat{\Omega}=\cL \widehat{\Omega}-\frac{1}{\delta}\{\widehat{\Phi}_{\eps,R}(\Psi,\widehat{\Psi}),\widehat{\Omega}\}.
\end{equation}We then split the solution as 
\begin{align}
\label{def:NLsplitting}
&\Omega=\Omega_{\app}^{(M)}+\delta  w,   \qquad \widehat{\Omega}=\widehat \Omega_{\app}^{(M)}+\delta \widehat{w},\\
    &\Psi=\Psi_{\app}^{(M)}+ \delta \varphi,   \qquad \widehat{\Psi}=\widehat{\Psi}_{\app}^{(M)}+\delta \widehat{\varphi}.
\end{align}
The main objective is to control the solution on a time interval $[0,T_{\adv}\delta^{-\sigma})$ for a given fixed $\sigma\in [0,1)$. To ensure that the remainders $\cR_{M},\widehat{\cR}_M$ are small enough, we will  assume that on the whole time interval $[0,T_{\adv}\delta^{-\sigma})$
\begin{align}
\label{fixM}
    \delta^{-2}\eps^{M+1}(t)\ll  \eps^2(t) \quad \Longrightarrow \quad M>\frac{5-\sigma}{1-\sigma}.
\end{align}
Therefore, given any $\sigma\in[0,1)$, the order of the approximate solution is fixed as a number $M\in \NN$ satisfying the condition above. We will think to $\sigma$ and therefore $M$ as fixed parameters from now on. Thus, we will always omit the superscript $(M)$ in the rest of this section.

\subsection{Preliminary considerations on the moments}
Before writing down the nonlinear system, we need to specify the radius and angular speed. Unfortunately, 
given the complicated dependence of the operators with respect to changes in the radius, we are not able to fix exactly the first order moments with the choices made in \cref{lem:alpha'gives_moments}. Thus, we choose $$R(t)=R_{\app}(t)$$ with $R_{\app}(t)$ defined in \cref{def:Rapp}. Since this is fixed, from now on we avoid tracking the explicit dependence on all the operators on $R_{\rm app}$.
\begin{remark}
\label{rem:RRapp}
    With the choice of $R(t)=R_{\rm app}(t)$, in view of Proposition \ref{prop:deviation} we automatically satisfy the expansion provided in \cref{eq:expRa}.
\end{remark}
Then, as announced in Remark \ref{rem:commentRalpha}, we slightly modify the definition of \cref{def:alpha} to 
    \begin{align}
    \label{def:alphaNL}
t\alpha'(t)&\coloneqq\frac{d\eps(t)}{\delta (R_{\rm app}+\eps d\, \mrm_1^1(\Omega))(t)}\Im(\cB_{\eps}(\Psi(t),\widehat{\Psi}(t))),
    \end{align}
    where we need to guarantee that $\mrm_1^1(\Omega(t))$ remains sufficiently small on our time-scale. As one can directly see from \cref{eq:m21}, we know that the choice above guarantees that 
\begin{align}
    \mrm_2^{1}(\Omega(t))=0 \quad \Longrightarrow \quad \mrm_2^1(w(t))=0, \qquad \text{for all } t\geq0,
\end{align}
where we used that $\mrm^1(\Omega_{\rm app}(t))=(0,0).$ Another consequence of the zero first order moment condition for the approximate solution,  is that $\mrm^1_1(\Omega(t))=\delta \mrm^1_1(w(t))$ and we can thus rewrite \cref{def:alphaNL} as 
    \begin{align}
    \label{def:alphaNL1}
t\alpha'(t)&\coloneqq t\alpha_{\rm app}'(t)+\mathfrak{b}_{\eps}(\Psi_{\rm app},\widehat{\Psi}_{\rm app},\varphi,\widehat{\varphi})(t)
    \end{align}
    where, suppressing the explicit $t$ dependence, we define 
    \begin{align}
    \label{def:frakb}
        \mathfrak{b}_{\eps}(\Psi_{\rm app},\widehat{\Psi}_{\rm app},\varphi,\widehat{\varphi}):=\frac{d \eps}{R_{\rm app}}\Im\Big(&\cB_{\eps}(\Psi_{\rm app},\widehat{\Psi}_{\rm app},\varphi))+\cB_\eps(\varphi,\widehat{\varphi},\Psi_{\rm app}+ \delta \varphi)\\
        &-\frac{\eps d \mrm_1^1(w)}{R_{\rm app}+\delta\eps d \mrm_1^1(w)}\cB_{\eps}(\Psi_{\rm app}+\delta \varphi,\widehat{\Psi}_{\rm app}+\delta \widehat{\varphi})\Big)
    \end{align}

Having lost one modulation parameter, we can only hope to obtain some bound on the first order moment $\mrm_1^1(w)$. To this end, we crucially use the property of the Eulerian part of the approximate solution stated in Lemma \ref{lem:keyLEuler} to obtain the following. 
\begin{lemma}
\label{lem:keymom}
    Let $\Omega,\widehat{\Omega}$ be as in \cref{def:NLsplitting}. Then, there exists a constant $C>0$ such that
\begin{align}
\label{bd:m11priori}
    |\mrm_1^1(w(t))|\leq C\big(\delta^{-2}\eps^{M+2}+\eps^3+\eps (\mathrm{L}( w)+\frac{\gamma}{N}\mathrm{L}( \widehat{w}))\big)(t)
\end{align}
\end{lemma}
\begin{proof}
We proceed as in the proof of Lemma \ref{lem:keyLEuler}, but now our $\omega(t)$ solves the 2D Navier-Stokes equation \cref{eq:NSrot} without external force, meaning that Lemma \ref{lem:NSmom} holds true. Thus, arguing as done to obtain the identities \cref{eq:Ang1,eq:AngOmapp}, we see that
\begin{align}
\label{eq:Ang2}
    \mathsf{r}^2+4\nu t\big(1+\frac{\gamma}{N}\big)=\nu t \big(\mathrm{L}(\Omega(t))+\frac{\gamma}{N}\mathrm{L}(\widehat{\Omega}(t))\big)+2\sqrt{\nu t}R_{\rm app}\, \mrm^1_1(\Omega(t))+R_{\rm app}^2.
\end{align}

Since $\mrm_1^1(\Omega)=\delta \mrm^1_1(w)$, using Lemma \ref{lem:keyLEuler} we  get
\begin{align}
\mrm_{1}^1(w)=\,&\frac{d^2}{2\delta \eps dR_{\rm app}}\Big(\frac{\mathsf{r}^2-R_{\rm app}^2}{d^2}-\eps^2\big(\mathrm{L}(\Omega_{\rm app,E}-G)+\frac{\gamma}{N}\mathrm{L}(\widehat{\Omega}_{\rm app,E}-G)\big)\Big)\\
    &-\frac{ \eps d}{2R_{\rm app}}\big(\mathrm{L}( \Omega_{\rm app,NS}+w)+\frac{\gamma}{N}\mathrm{L}( \widehat{\Omega}_{\rm app,NS}+\widehat{w})\big)\\
    =\,&\cO(\delta^{-2}\eps^{M+2}+\eps^3)-\frac{ \eps d}{2R_{\rm app}}\big(\mathrm{L}(w)+\frac{\gamma}{N}\mathrm{L}( \widehat{w})\big),
\end{align}
where in the last identity we also used that $R_{\rm app}\approx \mathsf{r}$ by the constraint \cref{constraint:4} and $\Omega_{\rm app,NS}=\cO_{\cZ}(\eps^2).$
\end{proof}

\subsection{The nonlinear system}
Having fixed $R=R_{\rm app}$ and $\alpha$ as in \cref{def:alphaNL}, we are ready to write down the nonlinear system for $w,\widehat{w}$. To simplify the notation, we avoid keeping track of the subscripts $\alpha,R$ since are fixed from now on, and we will write $\mathfrak{b}_\eps$ to actually denote the integral operator defined in \cref{def:frakb}. 

Then, we observe that 
\begin{align}
&\Phi_{\eps}(\Psi,\widehat{\Psi})=\Phi_{\eps}(\Psi_{\app},\widehat{\Psi}_{\app})+\delta((I+\cQ_{\eps})(\varphi)+ \gamma \cT_\eps(\widehat{\varphi}))-\frac{\delta}{2}\mathfrak{b}_\eps|\cT_\eps\xi|^2,\\
&\widehat{\Phi}_{\eps}(\Psi,\widehat{\Psi})=\widehat{\Phi}_{\eps}(\Psi_{\app},\widehat{\Psi}_{\app})+  \delta\big(\gamma\widehat{\varphi}+\cV_{\eps}(\varphi)\big)-\frac{\delta}{2}\mathfrak{b}_\eps|\xi|^2.
\end{align}
recalling the definitions of $\Phi_{\rm app}, \widehat{\Phi}_{\rm app}$ in \cref{def:Phiapp} 
and $\Phi_{\rm app, E}, \widehat{\Phi}_{\rm app,E}$ in \cref{def:PhiappE,def:PhihappE} we get
\begin{align}
&\Phi_{\eps}(\Psi_{\app},\widehat{\Psi}_{\app})=\Phi_{\rm app}=\Phi_{\rm app, E}+\delta \Phi_{\rm app, NS}, \\
&\widehat{\Phi}_{\eps}(\Psi_{\app},\widehat{\Psi}_{\app})=\widehat{\Phi}_{\rm app}=\widehat{\Phi}_{\rm app, E}+\delta \widehat{\Phi}_{\rm app, NS}.
\end{align}
Moreover, by construction $\Phi_{\rm app, NS}, \widehat{\Phi}_{\rm app, NS}=\cO_{\cS_*}(\eps^2)$. 
Finally, recalling the definitions of $\cR_M,\widehat{\cR}_M$ in \cref{def:RM},
by a direct computation we find that $w,\widehat{w}$ satisfy 
\begin{align}
   \label{eq:w} &(t\de_t-\cL)w+\frac{1}{\delta}\big(\Lambda_{\app}(w)+\Upsilon_{\app}(\widehat{w})\big)=\frac{1}{2\delta}\mathfrak{b}_\eps\{|\cT_{\eps} \xi|^2,\Omega_{\app,E}\}+\delta^{-2}\cR_M+\cN(w,\widehat{w})\\
     \label{eq:hw}&(t\de_t-\cL)\widehat{w}+\frac{1}{\delta}\big(\widehat{\Upsilon}_{\app}(w)+\widehat{\Lambda}_{\app}(\widehat{w})\big)=\frac{1}{2\delta}\mathfrak{b}_\eps\{|\xi|^2,\widehat{\Omega}_{\app,E}\}+\delta^{-2}\widehat{\cR}_M+\widehat{\cN}(w,\widehat{w})
     \end{align}
where we define:
\begin{enumerate}
    \item The linear operators arising from the linearization at the Eulerian approximate solution  are 
\begin{align}
   \label{def:Lambdaapp} &\Lambda_{\app}(f)=\{\Phi_{\rm app, E},f\}+\{(I+\cQ_{\eps})\Delta^{-1}f,\Omega_{\rm app, E}\}, \\
\label{def:Upsilon}&\Upsilon_{\rm app}(f)=\gamma \{\cT_\eps\Delta^{-1}f,
     \Omega_{\rm app, E}\},\\
\label{def:hatLambdaapp}&\widehat{\Lambda}_{\app}(f)=\{\widehat{\Phi}_{\rm app, E},f\}+\gamma\{\Delta^{-1}f,\widehat{\Omega}_{\rm app, E}\}, \\
\label{def:hatUpsilon}&\widehat{\Upsilon}_{\rm app}(f)= \{\cV_{\eps}\Delta^{-1}f,\widehat{\Omega}_{\rm app, E}\}.
\end{align}
\item The errors associated to the Navier-Stokes part of the approximate solution and the nonlinear interactions for the outer vortices are 
\begin{align}
\label{def:N}\cN(w,\widehat{w})&=\sum_{j=1}^3\cN_j(w,\widehat{w})\\
\label{def:N1}\cN_1(w,\widehat{w})&=-\{\Phi_{\rm app,NS},\, w\}-\{(I+\cQ_{\eps})\varphi+\gamma\cT_{\eps}\widehat{\varphi},\, \Omega_{\rm app,NS}\}\\
\label{def:N2}\cN_2(w,\widehat{w})&=-\{(I+\cQ_{\eps})\varphi+\gamma \cT_{\eps}\widehat{\varphi},\, w\}\\
\label{def:N3}\cN_3(w,\widehat{w})&=\frac{1}{2}\mathfrak{b}_\eps\{|\cT_{\eps}\xi|^2,\Omega_{\rm app, NS}+w\}.
\end{align}
For the central vortex  we have
\begin{align}
\label{def:hN}
\widehat{\cN}(w,\widehat{w})&=\sum_{j=1}^3\widehat{\cN}_j(w,\widehat{w})\,\\
\label{def:hN2}\widehat{\cN}_1(w,\widehat{w})&=-\{\widehat{\Phi}_{\rm app,NS},\, \widehat{w}\}-\{\gamma\widehat{\varphi}+\cV_{\eps}\varphi,\, \widehat{\Omega}_{\rm app,NS}\}\\
\label{def:hN3}\widehat{\cN}_2(w,\widehat{w})&=-\{\gamma\widehat{\varphi}+ \cV_{\eps}\varphi,\, \widehat{w}\}\\
\label{def:hN4}\widehat{\cN}_3(w,\widehat{w})&=\frac{1}{2}\mathfrak{b}_\eps\{|\xi|^2,\widehat{\Omega}_{\rm app, NS}+\widehat{w}\}.
\end{align}
\end{enumerate}
The equations \cref{eq:w,eq:hw} are solved with zero initial data. Indeed,  $\Omega_{\rm app},\widehat{\Omega}_{\rm app}\to \Omega_0$ as $t\to 0$ and therefore, by the uniqueness of the solution \cite{GaGa2005}, we know that $w,\widehat{w}\to 0$ as $t\to0$. Regarding the first order moments, it is straightforward to see that
\begin{align}
    \mrM(w(t))=\mrM(\widehat{w}(t))=0, \text{ for all } t\geq0.
\end{align}
 Then, the $N$-fold symmetry of $\widehat{\Omega},\widehat{\Omega}_{\rm app}$ ensures that
\begin{align}
\label{momhw}
    \widehat{w}(Q_l\xi)=\widehat{w}(\xi) \text{ for any } l=1,\dots N,\quad \Longrightarrow\quad 
    \mrm^1(\widehat{w}(t))=(0,0) \text{ for all } t\geq0,
\end{align}
where in the last identity we used  $N\geq 2$. The choice of the angular speed  guarantees that $\mrm^1_2(w(t))=0$ whereas by Lemma \ref{lem:keymom} we have the a priori bound \cref{bd:m11priori} for  $\mrm_1^1(w(t))$. 
We are then ready to state the main theorem of this section, which implies the main result of the paper Theorem~\ref{theo:main} as we explain below.
\begin{theorem}
\label{th:mainNL}
    Fix $\sigma\in [0,1)$, let $M\in \NN$ be such that \cref{fixM} is satisfied and $w,\widehat{w}$ be the solutions to \cref{eq:w,eq:hw} with zero initial data. Then, there exists positive constants $C>1,\delta_0\in (0,1)$ such that for any $\Gamma,\nu$ satisfying  $\nu/\Gamma=\delta\leq \delta_0$ the following holds true:
    \begin{align}
    \label{bd:mainNLbd}
\|w(t)\|_{\cX_\eps}+\|\widehat{w}(t)\|_{\widehat{\cX}_\eps}\leq C\eps^2(t)\quad \text{ for all } t\in[0,T_{\adv}\delta^{-\sigma}),
    \end{align}
    where $\eps,T_{\adv}$ are defined in \cref{def:adpar},  and the functions spaces $\cX_{\eps},\widehat{\cX}_{\eps}\hookrightarrow L^1(\RR^2)$ are defined in \cref{def:Xeps}. 
\end{theorem}
The proof of this theorem follows the strategy put forward in \cite{GaSring} and adapted to reach time-scales arbitrary close to the diffusive one in \cite{dolce2024long}. The main ideas and difficulties in the proof are explained in Section \ref{sec:ideas} as a preparation for the technical proof, which is the content of the rest of the paper.

During the proof, we need to propagate a control of the first order moment $\mrm^1_1(w)$ that depends on $w$ as well. Thus, we are going to prove Theorem \ref{th:mainNL} under the bootstrap assumption that there exists a maximal time $T\leq T_{\rm adv}\delta^{-\sigma}$ such that
\begin{align}
\label{bootstrap}\tag{B} \|w(t)\|_{\cX_\eps}+\|\widehat{w}(t)\|_{\widehat{\cX}_\eps}\leq 2C_\star\eps^2(t) \qquad \text{ for all }t\in [0,T),
\end{align}
where $C_\star>0$ is a constant fixed during the proof \footnote{We can take $C_\star=\sqrt{KC_1}$ with $K$ in Proposition \ref{prop:E} and $C_1$ in Proposition \ref{prop:FAN}.}.
This hypothesis can be easily verified by a short time estimate for $T=1$ and with $2$ replaced by $1/2$. By continuity, we can assume that \cref{bootstrap} holds true up to some time $T>1$. If  under the bootstrap assumption \cref{bootstrap} we are able to show that
        \begin{align} \|w(t)\|_{\cX_\eps}+\|\widehat{w}(t)\|_{\widehat{\cX}_\eps}\leq C_\star\eps^2(t) \qquad \text{ for all }t\in [0,T),
\end{align} 
then by a standard continuity argument we can take $T=T_{\rm adv}\delta^{-\sigma}$ and thus prove Theorem \ref{th:mainNL}. Then the proof of Theorem \ref{theo:main} readily follows.
\begin{proof}[Proof of Theorem \ref{theo:main}]
  Assume that the result in Theorem \ref{th:mainNL} holds true. By rescaling, the bound \cref{bd:mainNLbd} readily implies \cref{bd:mainNLintro} in Theorem \ref{theo:main}. The proof of \cref{eq:expRa} is then a consequence of Remark \ref{rem:RRapp} and Corollary \ref{cor:alpha} below, which combines the bounds on $w,\widehat{w}$ with the definition of $\alpha'$ in \cref{def:alphaNL1} and the expansion for $\alpha'_{\rm app}$ established in Proposition \ref{prop:deviation}.
\end{proof}
\subsubsection*{\textbf{Further conventions and notations}}
In this section, we introduce the following conventions. Given an $L^2(\RR^2)$-based Hilbert space $X$, we denote 
\begin{equation}
    \widehat{X}=\{f\in X\, :\, f \text{ is $N$-fold symmetric}\}.
\end{equation}
Then, for $f\in X$ and $\widehat{f}\in \widehat{X}$ we write 
\begin{align}
    \label{vec1}&{\mathbf f}=(f,\widehat{f}\,)\in X\times\widehat{X}=:{\mathbf X}, \\
    \label{vec2}&\nabla{\mathbf f}=(\nabla f,\nabla \widehat{f}\,),\\
    \label{vec3}&\cL{\mathbf f}=(\cL f,\cL\widehat{f}\,),\\
    \label{vec4}&\langle {\mathbf f}, {\mathbf g}\rangle_{\mathbf{X}}=\langle f,g\rangle_{X} +\langle \widehat{f},\widehat{g}\rangle_{\widehat{X}},\\
    \label{vec5}&\mrm^k_j({\mathbf f})=(\mrm^k_j( f),\mrm^k_j({\widehat{f}})).
\end{align}
For  a given $N$-fold symmetric function $g$, we write $g{\mathbf f}=(gf,g\widehat{f})$. When $X=L^2(\RR^2)$, we do not write the subscript in the inner products and norms. 

\subsection{Main difficulties and key ideas}
\label{sec:ideas}
Looking at the equations \cref{eq:w,eq:hw}, we isolated explicitly the most dangerous terms containing a factor $\delta^{-1}$.
The key idea to control those terms, is to exploit the Arnold's strategy in the viscous setting as recently revisited in \cite{GaSarnold}, which was first applied to control a viscous vortex ring over a very long-time scale in \cite{GaSring} and then in 2D situations in \cite{dolce2024long,zhang2025long}. To apply this strategy, one has to understand  properties of the operator \begin{align}
\label{def:Japp}
    \cJ_{\rm app}\coloneqq\begin{pmatrix}
        \Lambda_{\app} & \Upsilon_{\app}\\\widehat{\Upsilon}_{\app}
        &\widehat{\Lambda}_{\app}
\end{pmatrix}.
\end{align}
In particular, one would like to solve the following problem:
\begin{quote}
\centering
    Find a functional space where the operator $\cJ_{\rm app}$ 
is skew-adjoint. 
\end{quote}
For our purposes, it is in fact enough that the operator $\cJ_{\rm app}$ is skew-adjoint up to errors of order $\cO(\eps^{M+1})$. Here, we have to crucially exploit Proposition \ref{prop:functional}, that provides us with an approximate functional relationship between $\Phi_{\rm app,E}$ and $\Omega_{\rm app,E}$  in the region $|\xi|\leq 2\eps^{-\sigma_1}$.  This allow us to heuristically approximate
\begin{align}
\label{eq:Jappapp}
    \cJ_{\rm app}\approx \begin{pmatrix}
       \big\{ \big(F'(\Omega_{\rm app,E})+(I+\cQ_{\eps})\Delta^{-1}\big)(\cdot),\Omega_{\rm app,E}\big\}& \gamma \big\{\cT_{\eps}\Delta^{-1}(\cdot),\Omega_{\rm app,E} \big\}
       \\[6pt]
       \big\{\cV_{\eps}\Delta^{-1}(\cdot),\widehat{\Omega}_{\rm app,E} \big\}
        &\big\{ \big(\widehat{F}'(\widehat{\Omega}_{\rm app,E})+\gamma\Delta^{-1}\big)(\cdot),\widehat{\Omega}_{\rm app,E}\big\}
\end{pmatrix}.
\end{align}
Then, exploiting the fact that $\widehat{w}$ is $N$-fold symmetric, we can restrict our attention to the action of the second column above to $N$-fold symmetric functions. Appealing to Lemma~\ref{lem:adjoints}, we know that 
\begin{align}
 \cT_{\eps}\Delta^{-1}f=\frac{1}{N}\cV^*_{\eps}\Delta^{-1}f \quad \text{for any } N\text{-fold symmetric } f.
\end{align}
To construct a functional space where $\cJ_{\rm app}$ is skew-adjoint, looking at its structure  we define 
\begin{align}
\label{def:Happ}
\cM_{\app}=\begin{pmatrix}
       F'(\Omega_{\rm app,E})+(I+\cQ_{\eps})\Delta^{-1}& \frac{\gamma}{N} \cV^*_{\eps}\Delta^{-1}
       \\[6pt]
       \frac{\gamma}{N}\cV_{\eps}\Delta^{-1}
        &\frac{\gamma}{N}(\widehat{F}'(\widehat{\Omega}_{\rm app,E})+\gamma\Delta^{-1}).
\end{pmatrix}
\end{align}
By Lemma~\ref{lem:adjoints}, we know that $\cM_{\rm app}$ 
is self-adjoint in $\mathbf{L^2}$.
Moreover, from the definition of $\cM_{\rm app}$, one has that $\cJ_{\rm app}$ is approximately skew-adjoint with respect to the $\mathbf{L^2}$ inner product $$\langle {\mathbf f},\,\cM_{\rm app}({\mathbf g})\rangle.$$ 
Namely, if we assume that \cref{eq:Jappapp} holds up to errors of order $\eps^{M+1}$, by a direct computation we get
\begin{align}
\label{heuristic1}
    \big\langle \cJ_{\rm app}({\mathbf f}), \cM_\app ({\mathbf f})\big\rangle\approx \cO(\eps^{M+1}).
\end{align}
Therefore, if $\cM_{\rm app}$ is a positive operator, we would have a nice inner-product space where we can work (since $\cM_{\rm app}$ is self-adjoint). Unfortunately, $\cM_{\rm app}$ can have non-trivial elements in its kernel or even be negative on some functions. Indeed, even by formally setting $\eps \to0$, one has that $\cM_{\rm app}$ reduces to
\begin{align}
\label{def:M0}
\cM_0\coloneqq\begin{pmatrix}
       F'_0(\Omega_{0})+\Delta^{-1}& 0
       \\
       0
        &\frac{\gamma^2}{N}(F'_0(\Omega_{0})+\Delta^{-1}).
\end{pmatrix}
\end{align}
For the operator above, one has that $\de_1\Omega_0,\de_2\Omega_0$ are in the kernel\footnote{This is also why $\de_i\Omega_0=\de_iG$ are in the kernel of $\Lambda$ as stated in Proposition \ref{prop:Lambda}}, from which one can deduce that $\cM_0$ is at most non-negative (see \cite{GaSarnold}). This implies that we cannot use $\cM_{0}$ to define an inner-product space. On the other hand, one can avoid problems related to these elements in the kernel by assuming that the first order moments are sufficiently small. Indeed,  from  \cite{GaSarnold} it is possible to deduce that
\begin{align}
\label{bd:coercivityM0}
    \langle {\mathbf f}, \cM_0 ({\mathbf f})\rangle \geq \kappa(\|\sqrt{F'(\Omega_0)}f\|^2_{L^2}+\|\sqrt{F'(\Omega_0)}\widehat{f}\|^2_{L^2}\big)-C(|\mrM({\mathbf f})|^2+|\mrm_1^1({\mathbf f})|^2+|\mrm_2^1({\mathbf f})|^2)
\end{align}
where $\kappa,C>0$ are fixed constant.  Therefore, if we can impose the zero first order moments condition, we would indeed have a nice inner-product space defined through $\cM_{0}$. Moreover,  once things are well behaved for $\cM_{0}$, following \cite{GaSring,dolce2024long} one can essentially see $\cM_{\rm app}$ as a perturbation of $\cM_{0}$. Regarding the first order moments, one possibility is to exploit the modulation of some parameters to set them to zero. The modulation can however bring new error terms that can be very large on long time-scales, and usually one has to exploit symmetries to handle them \cite{dolce2024long} in the time-scale of interest. To  avoid issues related to the modulation parameters (as the radius in our case), the bound \cref{bd:coercivityM0} is enough if one is able to  propagate the smallness of the first-order moments by some other arguments,  as done for instance in \cite{GaSring} by using some approximate conservation law in the system or in \cite{zhang2025long} by controlling \emph{pseudo-momenta}. 

In our case, we have seen in Lemma \ref{lem:keymom} that with a good enough approximate position we can guarantee the smallness of $\mrm_1^1(w)$ (provided $w,\widehat{w}$ decay sufficiently fast an infinity). However, the modulation of the angular speed used to guarantee $\mrm_2^1(w)=0$ is generating  the potentially dangerous term 
\begin{equation}
\label{eq:Kernel}
    \frac{\mathfrak{b_\eps}}{2\delta}\begin{pmatrix}
    \{|\cT_{\eps} \xi|^2,\Omega_{\app,E}\}\\
        \{|\xi|^2,\widehat{\Omega}_{\app,E}\}.
    \end{pmatrix}
\end{equation}
The element above is related to the rotation of the Eulerian part of the whole approximate vortex polygon around the center, because the derivative $\{|\cT_{\eps} \xi|^2,\cdot\}$ is just an angular derivative with respect to the global origin in the reference frame of an exterior vortex. But the whole Eulerian approximate configuration is symmetric with respect to rotation around the global origin. This symmetry is reflected in the fact that the element \cref{eq:Kernel} is approximately in the kernel of $\cM_{\rm app}$. Namely, as we show more precisely in Lemma \ref{lem:keykernel}, we expect that 
\begin{equation}
\label{eq:heuristic1}
    \cM_{\rm app}\begin{pmatrix}
        \{|\cT_{\eps} \xi|^2,\Omega_{\app,E}\}\\
        \{|\xi|^2,\widehat{\Omega}_{\app,E}\}
    \end{pmatrix}\approx \cO(\eps^{M+1}).
\end{equation}
This is also a consequence of the existence of the functional relationship between $\Omega_{\rm app,E}$, $ \widehat{\Omega}_{\rm app,E}$ and $\Phi_{\rm app,E}$, $ \widehat{\Phi}_{\rm app,E}$ and the commutation properties in Lemma \ref{lem:adjoints} (themselves related to symmetries in the problem). We note that a property analogous to the one above was in fact the starting point in the construction of the ``non-trivial pseudo-momenta" in \cite{zhang2025long}.

The rest of this section is devoted in making the arguments above rigorous, which requires the introduction of many technical tools developed in \cite{GaSring,GalWay2002,dolce2024long}. In particular, we will use some of the results in \cite{dolce2024long} as a black box and we only  focus on the novelties and difficulties of our problem.

\subsection{The functional setting and basic properties}
\label{sec:functional}
We introduce two parameters and a constant as follows 
$$\sigma_1\in (0,1), \qquad \sigma_2>1, \qquad \varsigma\coloneqq \frac{\sigma_1}{\sigma_2}<1.$$ Here $\sigma_1$ is a sufficiently small parameter, depending only on the order $M$, that is implicitly updated during the proof. A first condition on this parameter is that Proposition \ref{prop:functional} holds true (namely one can start with $\sigma_1$ as the parameter introduced in Proposition \ref{prop:functional}). Instead, $\sigma_2>1$ is sufficiently large, used to exploit the decay at infinity of some parts of the solution. Then, by Proposition \ref{prop:functional} in the region $|\xi|\leq 2\eps^{-\sigma_1}$ we have the existence of the functional relationship between $\Omega_{\rm app,E}$ and $\Phi_{\rm app,E}$. Therefore, for the exterior vortices we define three time-dependent (through the parameter $\eps$) regions
\begin{align}
    {\rm I}_{\eps}&=\{\xi\in\RR^2\, :\, |\xi|<2\eps^{-\sigma_1}, \, F'(\Omega_{\rm app, E}(\xi))<\exp(\eps^{-2\sigma_1/4})\}\\
    {\rm II}_{\eps}&=\{\xi\in\RR^2\, :\, \xi \notin {\rm I}_{\eps},\, |\xi|<\eps^{-\sigma_2}\}\\
    {\rm III}_{\eps}&=\{\xi\in\RR^2\, :\, |\xi|>\eps^{-\sigma_2}\}.
\end{align}
For the central vortex, we define ${\rm \widehat I}_\eps, {\rm \widehat{II}}_\eps$ by replacing $F'(\Omega_{\rm app,E})$ with  $\widehat{F}'(\widehat{\Omega}_{\rm app,E})$.
The weight functions are then defined as 
\begin{align}
    A_{\eps}(\xi)=\begin{cases}
        F'(\Omega_{\rm app,E}(\xi)) &\text{in } {\rm I}_\eps,\\
    \exp(\eps^{-2\sigma_1}/4)&\text{in } {\rm II}_\eps,\\
    \exp(|\xi|^{2\varsigma}/4)&\text{in } {\rm III}_\eps.
    \end{cases}\qquad  \widehat{A}_{\eps}(\xi)=\begin{cases}
        \widehat{F}'(\widehat{\Omega}_{\rm app,E}(\xi)) &\text{in } {\rm \widehat{I}}_\eps,\\
    \exp(\eps^{-2\sigma_1}/4)&\text{in } {\rm \widehat{II}}_\eps,\\
    \exp(|\xi|^{2\varsigma}/4)&\text{in } {\rm III}_\eps.
    \end{cases}
\end{align}
We introduce the weighted $L^2$ space 
\begin{align}
\label{def:Xeps}
    \cX_{\eps}=\big \{f\in L^2(\RR^2)\, :\, \|f\|_{\cX_\eps}=\sqrt{\langle f,A_\eps f\rangle}<\infty\big \}.
\end{align}
The space $\widehat{\cX}_\eps$ is defined as $\cX_{\eps}$ with $A_{\eps}$ replaced by $\widehat{A}_\eps$ and restricted to $N$-fold symmetric functions. Note that when defining $\widehat{\cX}_\eps$ as a function space acting only on $N$-fold symmetric functions, we are implicitly using the $N$-fold symmetry of $\widehat{\Omega}_{\rm app,E}$. Then, we denote 
\begin{align}
  \label{vec6}  &\mathbf{A_\eps w}=(A_{\eps}f,\widehat{A}_\eps \widehat{f}), \qquad \boldsymbol{\cX}_\eps=\cX_\eps\times\widehat{\cX}_\eps,\\
    \label{vec7}&\|{\mathbf f}\|_{\boldsymbol{\cX}_\eps}^2=\|f\|^2_{\cX_\eps}+\|\widehat{f}\|^2_{\widehat{\cX}_\eps}.
\end{align}
Moreover, since as $\eps\to 0$ we formally have $A_{\eps}(\xi)\to A_0(\xi)=F_0'(\Omega_0)=A(\xi)$ and $\widehat{A}_{\eps}(\xi)\to \widehat{F}_0'(\Omega_0)=\gamma A(\xi)$ with $A$ defined in \cref{def:A}, we denote$$\cX_0\coloneqq L^2(A(\xi)\dd \xi)\qquad \widehat{\cX}_0=\{f\in  \cX_0\, :\, f \text{ is }N-\text{fold symmetric}\}.$$

The important  properties of these spaces are summarized in the following. 
\begin{lemma}
\label{lem:Aeps}
For any $\eps>0$ sufficiently small, the following holds true:
\begin{enumerate}[label=\roman*)]
    \item $\cY\hookrightarrow\cX_{\eps},\widehat{\cX}_{\eps}\hookrightarrow L^p$ for any $p\in [1,2]$.
    \medskip
    \item $\exp(|\xi|^{2\varsigma}/4)\lesssim A_{\eps}(\xi),\widehat{A}_\eps(\xi)\lesssim A_0(\xi)$ for all $\xi\in \RR^2$.
        \medskip
    \item $\|A_\eps-A_0\|_{C^1({\rm I}_\eps)}+\|\widehat{A}_\eps-\widehat{A}_0\|_{C^1({\rm \widehat I}_\eps)}\lesssim \eps^{\varsigma_1}A_0$ for any $\varsigma_1<2$.
\end{enumerate}

\end{lemma}
These properties can be proved as in \cite[Proposition 4.4]{dolce2024long}, and we thus omit the proof here.
\begin{remark}
\label{rem:momXeps}
    By the exponential growth of $A_\eps$, we know that for a function $f\in \cX_\eps$  we can bound any moment at a fixed order $n\in \NN$ since
\begin{align}
\|(1+|\xi|)^nf\|_{L^1}\lesssim\|f\|_{\cX_\eps}.
\end{align}
Therefore, under the bootstrap hypothesis \cref{bootstrap}, we know that for any $T\leq T_{\rm adv}\delta^{-\sigma}$ we get
\begin{align}
\label{bd:m11boot}
    |\mrm_1^1(w(t))|\lesssim \eps^3(t)\quad \text{ for any } t\leq T.
\end{align}
\end{remark}
We also recall the following  estimates.
\begin{lemma}
\label{lem:bounds}
    Let $f\in \cX_{\eps}$ be such that $,|\nabla f|\in \cX_{\eps}$ and let $\phi=\Delta^{-1}f$. Then, for all $q\in (2,\infty)$ there exists a constant $C>0$ such that 
    \begin{align}
    \label{bd:LinftyLq}&\|(1+|\cdot|)^{2}\phi\|_{L^\infty}+\|(1+|\cdot|)^{-1}\phi\|_{L^q}+\|(1+|\cdot|)^{2}\nabla \phi\|_{L^q}\leq C\|f\|_{\cX_{\eps}}\\
\label{bd:nablaLinfty}&\|(1+|\cdot|)^{3}\nabla \phi\|_{L^\infty}\leq C\big(\|\nabla f\|_{\cX_{\eps}}^\frac12+\|f\|_{\cX_{\eps}}^\frac12\big)\|f\|_{\cX_{\eps}}^\frac12.
\end{align}
Moreover, if $\mrM(f)=\mrm^1_2(f)=0$ and $|Z|>c>0$ then
\begin{align}
\label{bd:InnerLq}
\|\mathbbm{1}_{\rm I_\eps}\cT_{\eps,Z }\phi\|_{L^q}+\eps^{-1}\|\mathbbm{1}_{\rm I_\eps}(1+|\cdot|)^{-1}\nabla \cT_{\eps,Z }\phi\|_{L^q}\leq C\eps(\eps\|f\|_{\cX_\eps}+ |\mrm_1^1(f)|).
    \end{align}
    The statements remain true true with $(\cX_\eps,{\rm I}_\eps)\to (\widehat{\cX}_\eps,{\rm \widehat{I}}_\eps)$ and $\cT_{\eps,Z}\to \cV_{\eps,Z},\cQ_{\eps,Z}$ (possibly with a different constant $C>0$). 
\end{lemma}
The proof with $\mrm_1^1(f)=0$ can be found in \cite[Section A.4]{dolce2024long}, and the extension to $\mrm^1_1(f)\neq 0$ is a straightforward consequence of that proof.

Finally, following the discussion in Section \ref{sec:ideas}, we modify the definition of the operator $\cM_{\rm app}$ and we introduce
\begin{align}
\label{def:Meps}
\cM_{\eps}=\begin{pmatrix}
       A_\eps+(I+\cQ_{\eps})\Delta^{-1}& \frac{\gamma}{N} \cV^*_{\eps}\Delta^{-1}
       \\[6pt]
       \frac{\gamma}{N}\cV_{\eps}\Delta^{-1}
        &\frac{\gamma}{N}(\widehat{A}_\eps+\gamma\Delta^{-1})
\end{pmatrix}.
 \end{align}
 Note that  $\cM_\eps=\cM_{\rm app}$ in the interior region $\boldsymbol{\rm I}_\eps$.
 It is also often convenient to isolate the ``lower-order entries'' of $\cM_\eps$  given by 
\begin{align}
\label{def:Beps}
\cK_{\eps}=\begin{pmatrix}
      \cQ_{\eps}\Delta^{-1}& \frac{\gamma}{N} \cV^*_{\eps}\Delta^{-1}
       \\[6pt]
       \frac{\gamma}{N}\cV_{\eps}\Delta^{-1}
        &0
\end{pmatrix},
 \end{align}
 so that we have the following splitting 
     \begin{align}
    \label{eq:splitMeps}
        \cM_{\eps}=\begin{pmatrix}
       A_\eps+\Delta^{-1}& 0
       \\[6pt]
       0
        &\frac{\gamma}{N}(\widehat{A}_\eps+\gamma\Delta^{-1})
\end{pmatrix}+\cK_{\eps}.
    \end{align}
We can then exploit  Lemmas \ref{lem:Aeps} and \ref{lem:bounds} to obtain the following useful properties.
\begin{lemma}
\label{lem:MepsBeps}
Let $\cM_{\eps},\cK_{\eps}:\boldsymbol{\cX}_\eps\to \boldsymbol{\cX}_\eps$ be the operators defined in \cref{def:Meps,def:Beps}. Then, $\cM_{\eps},\cK_{\eps}$ are self-adjoint with respect to the standard $L^2$-inner product. Moreover, under the hypotheses of Lemmas \ref{lem:Aeps} and \ref{lem:bounds}, there exists a constant $C>0$ such that
  \begin{align}
  \label{bd:MB0}&|\langle {\mathbf g},\cM_{\eps}({\mathbf f})\rangle|\leq C \|{\mathbf f}\|_{\boldsymbol{\cX}_\eps}\|{\mathbf g}\|_{\boldsymbol{\cX}_\eps}\\
  \label{bd:MB}&|\langle {\mathbf g},\cK_{\eps}({\mathbf f})\rangle|+\eps^{-1}|\langle {\mathbf g},\nabla\cK_{\eps}({\mathbf f})\rangle|\leq C\eps(\eps \|{\mathbf f}\|_{\boldsymbol{\cX}_\eps}+|\mrm_1^1(f)|)\|{\mathbf g}\|_{\boldsymbol{\cX}_\eps},
\end{align}
for any ${\mathbf g}\in \boldsymbol{\cX}_{\eps}$ and any ${\mathbf f}=(f,\widehat{f})$ with $\mrM(f)=\mrm^1_2(f)=0$ as in Lemma \ref{lem:bounds}.
\end{lemma}
\begin{proof}
    The fact that $\cM_{\eps},\cK_{\eps}$ are self-adjoint in $L^2$  is a direct consequence of Lemma~\ref{lem:adjoints}. To prove the bound \cref{bd:MB}, we first observe that by Lemma \ref{lem:bounds} and the H\"older inequality with exponents $1/p+1/q=1$ and $q>2$, we have 
    \begin{align}
        |\langle {\mathbf g},\Delta^{-1}{\mathbf f}\rangle|\lesssim \|{\mathbf g}\|_{L^p}\|\Delta^{-1}{\mathbf f}\|_{L^q}\lesssim\|{\mathbf f}\|_{\boldsymbol{\cX}_\eps}\|{\mathbf g}\|_{\boldsymbol{\cX}_\eps} 
    \end{align}
    where in the last inequality we used also Lemma \ref{lem:Aeps} since $p<2$. Thus, by the splitting \cref{eq:splitMeps} we get 
  \begin{align}
  |\langle {\mathbf g},\cM_{\eps}({\mathbf f})\rangle|\lesssim \|{\mathbf f}\|_{\boldsymbol{\cX}_\eps}\|{\mathbf g}\|_{\boldsymbol{\cX}_\eps}+|\langle {\mathbf g},\cK_{\eps}({\mathbf f})\rangle|.
\end{align}
It remains to control the term involving the operator $\cK_{\eps}$.  Let us consider for instance the term involving $\cQ_{\eps}\Delta^{-1}$, since the others are analogous. Splitting the domain and applying the H\"older inequality, observe that 
\begin{align}
    |\langle g,\cQ_{\eps}\Delta^{-1}f\rangle|\lesssim \|\mathbbm{1}_{{\rm I}_\eps}g\|_{L^p}\|\mathbbm{1}_{{\rm I}_\eps}\cQ_{\eps}\Delta^{-1}f\|_{L^q}+\big\|\mathbbm{1}_{{\rm I}^c_\eps}A^{\frac12}_\eps g\big\|\|\mathbbm{1}_{{\rm I}_\eps^c}A^{-\frac12}_\eps\cQ_{\eps}\Delta^{-1}f\|,\end{align}
    where $1/p+1/q=1$ with $q>2$.
    To bound the first term on the right-hand side above, we use \cref{bd:InnerLq} and Lemma \ref{lem:Aeps}. For the second term, we bound $\cQ_{\eps}\Delta^{-1}f$ in $L^\infty$ thanks to \cref{bd:LinftyLq} and we use the fast decay properties of $A_{\eps}^{-1}$ to get that $\mathbbm{1}_{{\rm I}_\eps^c}A_{\eps}^{-\frac12}(\xi)\lesssim \cO(\eps^\infty)(1+|\xi|)^{-2}$ . Overall, we have  
\begin{align}
    |\langle g,\cQ_{\eps}\Delta^{-1}f\rangle|\lesssim \eps\|g\|_{\cX_\eps}(\eps\|f\|_{\cX_\eps}+|\mrm^1_1(f)|)\lesssim \eps\|g\|_{\cX_\eps}\|f\|_{\cX_\eps},
    \end{align}
    where in the last bound we used Remark \ref{rem:momXeps}.
    For the other terms in $\cK_\eps$ we can perform analogous estimates and therefore we deduce that \cref{bd:MB} holds true.
    To prove the bound involving $\nabla\cK_\eps$, observe that 
    \begin{align}
    |\langle g,\nabla \cQ_{\eps}\Delta^{-1}f\rangle|\lesssim \, &\|\mathbbm{1}_{{\rm I}_\eps}(1+|\cdot|)g\|_{L^p}\|\|\mathbbm{1}_{{\rm I}_\eps}(1+|\cdot|)^{-1}\nabla\cQ_{\eps}\Delta^{-1}f\|_{L^q}\\
    &+\big\|\mathbbm{1}_{{\rm I}^c_\eps}A^{\frac12}_\eps g\big\|\|\mathbbm{1}_{{\rm I}_\eps^c}A^{-\frac12}_\eps\nabla\cQ_{\eps}\Delta^{-1}f\|.\end{align}
    In the first term, we gain an extra factor of $\eps$ thanks to \cref{bd:InnerLq}. For the second term, it is enough to apply H\"older inequality with $1/q+1/p=1/2$ in the term involving $A^{-\frac12}_\eps$ and combine the properties in Lemma \ref{lem:bounds} with the exponential decay of $A_\eps^{-\frac12}$ outside the inner region. Hence, it is not hard to conclude that \cref{bd:MB} holds true.
    \end{proof}
    Finally, to end this section let us formalize the property claimed in \cref{eq:heuristic1} as follows.
    \begin{lemma}
    \label{lem:keykernel}
        Let $\cM_{\eps}$ be defined as in \cref{def:Meps}. Then, for any $\mathbf{f}\in \boldsymbol{\cX}_\eps$
        \begin{align}
    \Big\langle\cM_{\eps}\begin{pmatrix}
        \{|\cT_{\eps} \xi|^2,\Omega_{\app,E}\}\\
        \{|\xi|^2,\widehat{\Omega}_{\app,E}\}
    \end{pmatrix}, \begin{pmatrix}
        \mathbbm{1}_{{\rm I}_\eps}f\\
        \mathbbm{1}_{\widehat{{\rm I}}_\eps}\widehat{f}
\end{pmatrix}\Big\rangle=\Big\langle\begin{pmatrix}
        \{|\cT_{\eps} \xi|^2,\Theta\}\\
        \{|\xi|^2,\widehat{\Theta}\}
    \end{pmatrix},\begin{pmatrix}
        \mathbbm{1}_{{\rm I}_\eps}f\\
        \mathbbm{1}_{\widehat{{\rm I}}_\eps}\widehat{f}
\end{pmatrix}\Big\rangle
        \end{align}
        where $\Theta, \widehat{\Theta}$ are the functions defined in Proposition \ref{prop:functional}.
    \end{lemma}
    \begin{proof}
        We check the identity component wise for the term involving $\cM_{\eps}$. Since we are looking at the first component in the interior region ${\rm I}_\eps$ and the second on $\widehat{{\rm I}}_{\eps}$, we know that $A_{\eps}=F'(\Omega_{\rm app,E})$ and $\widehat{A}_{\eps}=\widehat{F}'(\widehat{\Omega}_{\rm app,E})$. Thus, on the support of the integral we can write
        \begin{align}
    &\cM_{\eps}\begin{pmatrix}
        \{|\cT_{\eps} \xi|^2,\Omega_{\app,E}\}\\
        \{|\xi|^2,\widehat{\Omega}_{\app,E}\}
\end{pmatrix}\\
&=\begin{pmatrix}
        \{|\cT_{\eps} \xi|^2,F(\Omega_{\app,E})\}+(I+\cQ_{\eps})\Delta^{-1}\{|\cT_{\eps} \xi|^2,\Omega_{\app,E}\}+\frac{\gamma}{N}\cV_{\eps}^*\Delta^{-1}\{| \xi|^2,\widehat{\Omega}_{\app,E}\}\\
          \frac{\gamma}{N}\{|\xi|^2,\widehat{F}(\widehat{\Omega}_{\app,E})\}+\frac{\gamma^2}{N}\Delta^{-1}\{| \xi|^2,\widehat{\Omega}_{\app,E}\}+\frac{\gamma}{N}\cV_{\eps}\Delta^{-1}\{| \cT_{\eps}\xi|^2,{\Omega}_{\app,E}\}
    \end{pmatrix}.
        \end{align}
        Appealing to Lemma \ref{lem:adjoints} and using \cref{key:Q,key:V*,key:V}, we have that 
        \begin{align}
        &  \begin{pmatrix}
    (I+\cQ_{\eps})\Delta^{-1}\{|\cT_{\eps} \xi|^2,\Omega_{\app,E}\}+\frac{\gamma}{N}\cV_{\eps}^*\Delta^{-1}\{| \xi|^2,\widehat{\Omega}_{\app,E}\}\\
       \frac{\gamma^2}{N}\Delta^{-1}\{| \xi|^2,\widehat{\Omega}_{\app,E}\}+\frac{\gamma}{N}\cV_{\eps}\Delta^{-1}\{| \cT_{\eps}\xi|^2,{\Omega}_{\app,E}\}    \end{pmatrix}\\
       =\, &
    \begin{pmatrix}
    \{|\cT_{\eps} \xi|^2,(I+\cQ_{\eps})\Delta^{-1}\Omega_{\app,E}+\frac{\gamma}{N}\cV_{\eps}^*\Delta^{-1}\widehat{\Omega}_{\app,E}\}\\
       \frac{\gamma}{N}\{| \xi|^2,\gamma \Delta^{-1}\widehat{\Omega}_{\app,E}+\cV_{\eps}\Delta^{-1}{\Omega}_{\app,E}\}.    \end{pmatrix}
        \end{align}
        Recalling the definition of $\Phi_{\rm app,E},\widehat{\Phi}_{\rm app,E}$ in \cref{def:PhiappE,def:PhihappE}, since  $\cT_{\eps}\Delta^{-1}\widehat{\Omega}_{\rm app,E}=N^{-1}\cV_\eps^*\Delta^{-1}\widehat{\Omega}_{\rm app,E}$, we get that on the support of the integral
        \begin{align}
\cM_{\eps}\begin{pmatrix}
        \{|\cT_{\eps} \xi|^2,\Omega_{\app,E}\}\\
        \{|\xi|^2,\widehat{\Omega}_{\app,E}\}
\end{pmatrix}=\begin{pmatrix}
        \{|\cT_{\eps} \xi|^2,F(\Omega_{\app,E})+\Phi_{\rm app,E}\}\\
        \{|\xi|^2,\widehat{F}(\widehat{\Omega}_{\app,E})+\widehat{\Phi}_{\rm app,E}\}
\end{pmatrix}.
        \end{align}
        The proof of the lemma is then a direct consequence of the definition of $\Theta,\widehat{\Theta}$ in Proposition \ref{prop:functional}.
    \end{proof}
\subsection{Bound on the angular speed correction}
Our first task is to bound the correction to the approximate angular speed $\mathfrak{b}_\eps$ defined in \cref{def:frakb}.  This bound is reminiscent of what was done in \cite[Lemma 4.7]{dolce2024long}, since in that case one fixes the moment $\mrm^1_2(w)$ by a nonlinear correction to the vertical speed of the dipole. In our case, we can obtain better estimates compared to \cite[Lemma 4.7]{dolce2024long} because the approximate angular speed is exactly the one given by the approximate solution. This is reflected in the fact that for $w=\widehat{w}=0$ we get $\mathfrak{b}_\eps=0$, whereas in the dipole case in \cite{dolce2024long} there are errors left due to the iterative construction of the approximate vertical speed. 
\begin{lemma}
\label{lem:beta}
    Let $\mathfrak{b}_\eps(\cdot)$ be defined as in \cref{def:frakb}. Under the bootstrap hypothesis \cref{bootstrap}, there exists a constant $C>0$ such that for all $t\in [0,T)$
    \begin{align}
    |\mathfrak{b}_{\eps}(\Psi_{\rm app},\widehat{\Psi}_{\rm app}, \varphi,\widehat{\varphi})(t)|\leq C \big(\eps^3(|\mrm^1_1(w)|+\eps\|{\mathbf w}\|_{\boldsymbol{\cX}_\eps})(1+\delta \|w\|_{\cX_\eps})\big)(t)\lesssim \eps^6(t).
    \end{align}
\end{lemma}
\begin{proof}
We bound the term in bracket in \cref{def:frakb}. To this end, note that
     \begin{align}
    \label{eq:AtoBeps}
\Im(\cB_\eps(h,\widehat{h},\widetilde{h}))=\Im(\cB_{\eps}({\mathbf h},\widetilde{h}))=\langle \de_1\cK_\eps(\Delta{\mathbf h}),(\Delta\widetilde{h},0)\rangle.
    \end{align}
Then, by Lemma \ref{lem:expansionsR} we can expand $\de_1\cK_{\eps}(\boldsymbol{\Omega}_{\rm app})$ and use $\mrm^1(\Omega_{\rm app})=(0,0)$ to see that 
    \begin{align}
    |\Im(\cB_{\eps}(\Psi_{\rm app},\widehat{\Psi}_{\rm app}, \varphi))|&\lesssim \eps |\mrM(w)|+\eps^2|\mrm^1_1(w)|+\eps^3\|w\|_{\cX_\eps}\\
    &\lesssim \eps^2(|\mrm^1_1(w)|+ \eps\|w\|_{\cX_\eps}).
    \end{align}
For the second term in bracket in \cref{def:frakb}, we directly combine the identity \cref{eq:AtoBeps} with Lemma \ref{lem:MepsBeps} to get
\begin{align}
    |\Im\cB_{\eps}((\varphi,\widehat{\varphi},\Psi_{\rm app}+\delta \varphi))|\lesssim  \eps^2(|\mrm^1_1(w)|+\eps\|{\mathbf w}\|_{\boldsymbol{\cX}_\eps})(1+\delta \|w\|_{\cX_\eps}).
\end{align}
Finally, for the last term we observe that by the constraint \cref{constraint:4} and \cref{bd:m11boot} we get
\begin{align}
|R_{\rm app}+\delta \eps d \mrm^1_1(w(t))|\gtrsim \mathsf{r}
\end{align}
Since also $|\cB_{\eps}(\Psi_{\rm app},\widehat{\Psi}_{\rm app})|\lesssim \eps$, arguing as for the other terms it is not hard to get 
\begin{align}
    \frac{\eps d|\mrm^1_1(w)|}{|R_{\rm app}+\delta \eps d\mrm^1_1(w)|}|\cB_{\eps}(\Psi_{\rm app}+\delta \varphi,\widehat{\Psi}_{\rm app}+\delta \widehat{\varphi})|\lesssim  \eps^2(|\mrm^1_1(w)|+\eps\|{\mathbf w}\|_{\boldsymbol{\cX}_\eps})(1+\delta \|w\|_{\cX_\eps}).
\end{align}
Hence, the proof follows by plugging the bounds above in \cref{def:frakb} and the bound with $\eps^4$ is a consequence of the bootstrap hypothesis.
\end{proof}
As a corollary of the bound in \cref{lem:beta}, we can verify the formula for the vertical speed in \cref{eq:expRa}.
\begin{corollary}
\label{cor:alpha}
    Under the hypotheses of Lemma \ref{lem:beta}, the formula for $\alpha'$ in \cref{eq:expRa} holds true.
\end{corollary}
\begin{proof}
    Combining the identity \cref{def:alphaNL1} with \cref{eq:alphaappMproof} in the proof of Proposition \ref{prop:deviation}, note that 
    \begin{align}
        \delta t\alpha'= d\frac{\eps^2}{\mathsf{r}^2} c_{1,0} (S_{1,1}+\gamma)\Big( 1 + \alpha_4 \eps^6  
        + \cO(\eps^7 + \delta \eps^6)\Big)+\delta \mathfrak{b}_\eps.
    \end{align}
    Using Lemma \ref{lem:beta}, we can include the factor $\delta\mathfrak{b}_\eps$ in the $\cO(\delta \eps^6)$ term and proceed as in the proof of Proposition \ref{prop:deviation} to prove \cref{eq:expRa} and~\cref{maj:zeta_j}.
\end{proof}
\subsection{Energy and dissipation functionals}
In view of the discussion in Section \ref{sec:ideas} and the definition of the operator $\cM_{\eps}$ in \cref{def:Meps}, we introduce the energy functional that mimics the Arnold quadratic form, given by
\begin{align}
\label{def:Eeps}
    E_{\eps}({\mathbf w})=\frac12\langle {\mathbf w},\cM_{\eps}({\mathbf w})\rangle, \qquad {\mathbf w}\in \boldsymbol{\cX}_\eps,
\end{align}
where we have used the notation \cref{vec1,vec4}.
In the formal limit $\eps\to 0$, thanks to Lemma \ref{lem:Aeps} note that $\cM_{\eps}\to\cM_0$. In view of \cref{bd:coercivityM0}, we then expect that $E_0$ defines a norm equivalent to $\boldsymbol{\cX}_\eps$, provided that $\mrm^k_j({\mathbf w})=(0,0)$ for $k=0,1$ and $j=1,2$. As in \cite{dolce2024long,GaSring}, we are able to extend this property to any sufficiently small $\eps>0$, as stated below.
\begin{proposition}
\label{prop:E}
    If $\eps,\sigma_1>0$ are sufficiently small, there exist a constant $K>1$, independent of $\eps$, such that the following holds true: let ${\mathbf w} \in \boldsymbol{\cX}_\eps$ be such that $\mrM({\mathbf w})=\mrm^1_2({\mathbf w})=(0,0)$ and $\mrm^1_1(\widehat{w})=0$. Then
    \begin{align}
        \label{bd:Elow}
        \|{\mathbf w}\|_{\boldsymbol{\cX}_\eps}^2\leq K(E_{\eps}({\mathbf w})+|\mrm^1_1(w)|^2)
    \end{align}
\end{proposition}
\begin{remark}
    In the bound \cref{bd:Elow}, one cannot hope to easily get a small constant in front of $\mrm^1_1(w)$. This is the reason why we need to bound a priori $\mrm^1_1(w)$ with Lemma \ref{lem:keymom}.
\end{remark}
\begin{proof}
    We first   estimate the quadratic form associated to the operator \cref{def:Beps}.
In view of Lemma \ref{lem:MepsBeps}, we have \begin{align}
\label{bd:wBw}
     |\langle {\mathbf w},\cK_\eps({\mathbf w})\rangle|\lesssim \eps(|\mrm^1_1(w)|+\eps\|{\mathbf w}\|_{\boldsymbol{\cX}_\eps})\|{\mathbf w}\|_{\boldsymbol{\cX}_\eps}.
 \end{align}
 We then have to control 
  \begin{align}
     \langle {\mathbf w},(\cM_\eps-\cK_\eps){\mathbf w}\rangle=\big(\|w\|_{\cX_\eps}+\langle \varphi,w\rangle\big)+\big(\|\widehat{w}\|_{\widehat{\cX}_\eps}+\langle \widehat{\varphi},\widehat{w}\rangle\big).
 \end{align}
For these terms we can argue exactly as done in the proof of \cite[Proposition 4.9]{dolce2024long} (see the bounds for  $E_\eps^1(w)$ in \cite[Proposition 4.9]{dolce2024long}) and easily adapt the proof to include the case where $\mrm^1_1(w)\neq 0$. We omit these details and therefore the proof is over.
 \end{proof}
With this proposition at hand, we know that it is enough to control  $E_\eps({\mathbf w}(t))$ to improve our bootstrap hypothesis \cref{bootstrap} and consequently prove Theorem \ref{th:mainNL}. Before proceeding with the estimates for $E_\eps$, we introduce the \emph{diffusive quadratic form}
 \begin{align}
 \label{def:Deps}
     D_{\eps}({\mathbf w})\coloneqq -\frac12\Big\langle \begin{pmatrix}
         t\de_t A_\eps & 0\\
         0 &\frac{\gamma}{N}t\de_t \widehat{A}_\eps
     \end{pmatrix}{\mathbf w},{\mathbf w}\Big\rangle-\langle \cL({\mathbf w}),\cM_{\eps}({\mathbf w})\rangle.
 \end{align} 
 where we have used the notation \cref{vec1,vec4,vec3,vec6}.
 To state its properties, we define the radially symmetric function 
\begin{align}
    \rho_\eps(\xi)=\begin{cases}
        |\xi|, \qquad &\text{if }|\xi|<\eps^{-\sigma_1}\\
        \eps^{-\sigma_1}\qquad &\text{if }\eps^{-\sigma_1}\leq |\xi|<\eps^{-\sigma_2},\\
        |\xi|^{\varsigma} \qquad &\text{if }|\xi|\geq \eps^{-\sigma_2},
    \end{cases}
\end{align}
Then, in analogy with \cite[Proposition 4.13]{dolce2024long}, we have the following result.
\begin{proposition}
\label{prop:D}
If $\eps,\sigma_1>0$ are sufficiently small and $\sigma_2$ is sufficiently large, there exist $\kappa_D,C>0$, independent of $\eps$, such that the following holds true: let ${\mathbf w} \in \boldsymbol{\cX}_\eps$ be such that $\rho_\eps {\mathbf w} \in \boldsymbol{\cX}_\eps$, $|\nabla w|\in \cX_{\eps}, |\nabla\widehat{w}|\in \widehat{\cX}_\eps$, $\mrM({\mathbf w})=\mrm^1_2({\mathbf w})=(0,0)$ and $\mrm^1_1(\widehat{w})=0$. Then
\begin{equation}
\label{bd:Dlow}
D_{\eps}({\mathbf w})\geq \kappa_D\big(\|\nabla {\mathbf w}\|^2_{\boldsymbol{\cX}_\eps}+\|\rho_\eps {\mathbf w}\|^2_{\boldsymbol{\cX}_\eps}+\| {\mathbf w}\|^2_{\boldsymbol{\cX}_\eps}\big)-C|\mrm^1_1(w)|^2
\end{equation}
\end{proposition}
\begin{proof}
Using the notation in \cref{def:Beps}, observe that 
\begin{align}
    \langle \cL(\boldsymbol{w}),\cK_{\eps}(\boldsymbol{w})\rangle=-\langle \nabla\boldsymbol{w},\nabla(\cK_{\eps}(\boldsymbol{w}))\rangle-\frac12\langle \boldsymbol{w},\xi\cdot \nabla (\cK_{\eps}(\boldsymbol{w}))\rangle.
\end{align}
Hence, appealing to Lemma \ref{lem:MepsBeps} and using properties of the weights in  Lemma \ref{lem:Aeps}, we deduce that
\begin{align}
     |\langle \cL(\boldsymbol{w}),\cK_{\eps}(\boldsymbol{w})\rangle|\lesssim \eps^2(|\mrm_1^1(w)|+\eps\|\mathbf{w}\|_{\boldsymbol{\cX}_\eps})\big(\|\nabla {\mathbf w}_\eps\|_{\boldsymbol{\cX}_\eps}+\| {\mathbf w}_\eps\|_{\boldsymbol{\cX}_\eps}\big).
\end{align}
For the operator 
\begin{align}
    D_\eps^1(w)=  -\frac12\Big\langle \begin{pmatrix}
         t\de_t A_\eps & 0\\
         0 &\frac{\gamma}{N}t\de_t \widehat{A}_\eps
     \end{pmatrix}{\mathbf w},{\mathbf w}\Big\rangle-\langle \cL({\mathbf w}),(\cM_{\eps}-\cK_\eps)({\mathbf w})\rangle
\end{align}
we can argue as in the proof of \cite[Proposition 4.13]{dolce2024long} (specifically the control of their terms $\cD_{W_\eps}[w]$ and $\cD_{\cL,\eps}[w])$), and the arguments are  straightforward to adapt to the case where $\mrm_1^1(w)\neq0$. Thus, we obtain the desired lower bound provided $\eps$ is sufficiently small.
\end{proof}

\subsection{Energy estimates}
We now have all the ingredients to prove a priori bounds on the energy defined in \cref{def:Eeps}. Before proceeding with the arguments, we remark that we will be heavily using the notations in \cref{vec1,vec2,vec3,vec4,vec5,vec6,vec7} in this section. In particular, recalling the definition of $\cJ_{\rm app}$ in \cref{def:Japp}, we rewrite \cref{eq:w,eq:hw} as 
\begin{align}
\label{eq:vecw}
    (t\de_t-\cL){\mathbf w}+\frac{1}{\delta}\cJ_{\rm app}({\mathbf w})=\frac{\mathfrak{b}_\eps}{2\delta}\begin{pmatrix}
                \{|\cT_\eps \xi|^2,\Omega_{\app,E}\}\\
                \{|\xi|^2,\widehat{\Omega}_{\app,E}\}
            \end{pmatrix}+\delta^{-2}\boldsymbol{\cR}_M+\boldsymbol{\cN}({\mathbf w})
\end{align}
        where $\boldsymbol{\cR}_M=(\cR_M,\widehat{\cR}_M)$         and $\boldsymbol{\cN}({\mathbf w})=(\cN(w,\widehat{w}),\widehat{\cN}(w,\widehat{w}))$ as defined in \cref{def:N,def:hN}.

We are now ready to compute the time-derivative of $E_{\eps}({\mathbf w}(t))$.
\begin{lemma}
\label{lem:energy}
    Let ${\mathbf w}=(w,\widehat{w})$ with $w,\widehat{w}$ solving \cref{eq:w,eq:hw} with zero-initial data and $E_{\eps}({\mathbf w})$ be defined as in \cref{def:Eeps}. Then 
    \begin{align}
        t\de_tE_{\eps}({\mathbf w})+D_{\eps}({\mathbf w})={\rm F}+{\rm A}+{\rm N}
    \end{align}
    where $D_\eps({\mathbf w})$ is defined in \cref{def:Deps} and we define:
    \begin{itemize}[label=$\circ$]
        \item The forcing term
        \begin{align}
           \label{def:F} {\rm F}\coloneqq\frac{\mathfrak{b}_\eps}{2\delta}\Big\langle \begin{pmatrix}
                \{|\cT_\eps \xi|^2,\Omega_{\app,E}\}\\
                \{|\xi|^2,\widehat{\Omega}_{\app,E}\}
            \end{pmatrix}, \cM_{\eps}({\mathbf w}) \Big\rangle+\delta^{-2}\langle \boldsymbol{\cR}_M,\cM_{\eps}{\mathbf w}\rangle.
        \end{align}
        \item The advection term 
        \begin{align}
            \label{def:Adv}{\rm A}\coloneqq  -\frac{1}{\delta}\langle\cJ_{\rm app}({\mathbf w}),\cM_{\eps}({\mathbf w})\rangle.
        \end{align}
        \item The Navier-Stokes and nonlinear error term
        \begin{align}
          \label{def:Nen}  {\rm N}\coloneqq \langle\boldsymbol{\cN}({\mathbf w}),\cM_{\eps}({\mathbf w})\rangle+\frac12\langle {\mathbf w}, [t\de_t,\cK_{\eps}]({\mathbf w})\rangle.
        \end{align}
    \end{itemize}
\end{lemma}
\begin{proof}
    Since $\cM_{\eps}$ is self-adjoint on $L^2\times \widehat{L}^2$, one has that 
    \begin{align}
        t\de_t (E_{\eps}({\mathbf w}))&=\frac12\langle t\de_t{\mathbf w},\cM_{\eps}(\mathbf w)\rangle +\frac12\langle {\mathbf w},t\de_t(\cM_\eps({\mathbf w}))\rangle\\
        &=\langle t\de_t{\mathbf w},\cM_{\eps}(\mathbf w)\rangle +\frac12\langle {\mathbf w},[t\de_t,\cM_\eps]({\mathbf w})\rangle.
    \end{align}
Notice that 
\begin{align}
[t\de_t,\cM_{\eps}]=\begin{pmatrix}
        t\de_t A_\eps &0\\
        0 &\frac{\gamma}{N}t\de_t\widehat{A}_\eps
    \end{pmatrix}+ [t\de_t,\cK_{\eps}].
\end{align}
The first term on the right-hand side is included in the definition of $D_{\eps}({\mathbf w})$, and the other one in the definition of ${\rm N}.$ All other terms simply arise by using the equation \cref{eq:vecw} to replace the $t\de_t{\mathbf w}$, and notice that the term with $\cL$ is included in the definition of $D_{\eps}({\mathbf w}).$
 \end{proof}
The goal is to estimate each of the terms in \cref{def:F,def:Adv,def:Nen}. In view of the key property related to symmetries in Lemma \ref{lem:keykernel}, the proof is similar in spirit to \cite{dolce2024long}. Here we slightly simplify the exposition and we highlight the structural properties that one can exploit in a system rather than the scalar case handled in \cite{dolce2024long}. The  strategy in \cite{dolce2024long} was in fact already adapted in \cite{zhang2025long} to handle systems, but in their case they need to decompose the solution to make full use of the pseudo-momenta and choose appropriately the modulation of the angular speed.

Overall, our objective is to prove the following proposition which is at the core of this section.
\begin{proposition}
\label{prop:FAN}     Fix $\sigma\in [0,1)$, let $M\in \NN$ satisfying \cref{fixM} and ${\mathbf w}$ be the solution to \cref{eq:vecw} with zero initial data. Then, for $\delta,\sigma_1>0$ sufficiently small and $\sigma_2>1$ sufficiently large, there exists a positive constant $C_1>1$ such that the following holds true: under the bootstrap hypothesis \cref{bootstrap}, for any $t\in  [0,T)$ with $T\leq T_{\rm adv}\delta^{-\sigma}$, the quantities \cref{def:F,def:Adv,def:Nen} satisfy 
    \begin{align}
    \label{bd:FAN}
        (|{\rm F}|+|{\rm A}|+|{\rm N}|)(t)\leq (C_1\eps^4+\frac{1}{2}D_{\eps}({\mathbf w}))(t).
    \end{align}
\end{proposition}
In the proof of this proposition, it is not hard to deduce that the constant $C_1$ can be explicitly related to $\|\boldsymbol{\cR}_M\|_{\cX_\eps}$ (see in particular the bound \cref{bd:F1}). We can now finally prove Theorem \ref{th:mainNL}.

\begin{proof}[Proof of Theorem \ref{th:mainNL}]
    Combining the bound \cref{bd:FAN} with the energy identity in Lemma \ref{lem:energy}, we deduce that 
    \begin{align}
        t\de_t E_{\eps}({\mathbf w})+\frac12D_{\eps}({\mathbf w})\leq C_1\eps^4.
    \end{align}
    Dividing by $t$ and integrating in time, we see that 
    \begin{align}
        E_{\eps}({\mathbf w})\leq \frac12 C_1\eps^4.
    \end{align}
    Then, by Proposition \ref{prop:E} and Lemma \ref{lem:keymom}, since on the time-scale under consideration $\delta^{-2}\eps^{M+1}\ll \eps^2$, we get  
    \begin{align}
        \|{\mathbf w}\|^2_{\cX_\eps}\leq K(E_{\eps}({\mathbf w})+|\mrm^1_1(w)|^2)\leq \frac12 KC_1\eps^4+
        C(\eps^6+\eps^2 \|{\mathbf w}\|_{\cX_\eps}^2),
    \end{align}
    for a suitable constant $C>0$ independent of $\delta$ and $\eps$.
    Therefore, if $\delta$ is sufficiently small, we can conclude that 
    \begin{align}
        \|{\mathbf w}\|^2_{\cX_\eps}\leq KC_1\eps^4,
    \end{align}
    where we recall that $C_1,K$ are the constants in Proposition \ref{prop:FAN} and Proposition \ref{prop:E} that do not depend on the bootstrap hypothesis \cref{bootstrap}. It is then enough to choose $C_{\star}=\sqrt{KC_1}$ in \cref{bootstrap} and we see that indeed \cref{bootstrap} holds true with $2$ replaced by $1.$ Thus, by a continuity argument this imply that we can take $T=T_{\rm adv}\delta^{-\sigma}$, thus concluding the proof of the theorem.
\end{proof}
The rest of the paper is then dedicated to the proof of the bound \cref{bd:FAN}, which we divide into three steps below.
\subsubsection*{\textbf{Step 1: bound on the forcing term}}
To handle the forcing term ${\rm F}$ in \cref{def:F}, we use the fact that $\cM_{\eps}$ is self-adjoint and we split ${\rm F}={\rm F}_1+{\rm F}_2$ where 
        \begin{align}
           \label{def:F1}{\rm F}_1&\coloneqq\delta^{-2}\langle \cM_{\eps}\boldsymbol{\cR}_M,{\mathbf w}\rangle, \\
            \label{def:F2}
           {\rm F}_2&\coloneqq \frac{\mathfrak{b}_\eps}{\delta}\Big\langle\cM_{\eps} \begin{pmatrix}
                \{|\cT_{\eps} \xi|^2,\Omega_{\app,E}\}\\
                \{|\xi|^2,\widehat{\Omega}_{\app,E}\}
            \end{pmatrix}, {\mathbf w} \Big\rangle.
        \end{align}
        To bound ${\rm F}_1$, by Lemma \ref{lem:MepsBeps}, the bound for $\boldsymbol{\cR}_M$ in Proposition \ref{prop:construction} and using Proposition \ref{prop:D} we readily get 
\begin{align}
\label{bd:F1first}
   |{\rm F}_1|\lesssim (\delta^{-2}\eps^{M+1}+\eps^2)\|{\mathbf w}\|_{\boldsymbol{\cX}_\eps}\leq C\eps^4+\frac{1}{100}(D_{\eps}({\mathbf w})+|\mrm^1_1(w)|^2),
\end{align}
where we used that $\delta^{-2}\eps^{M+1}\ll \eps^2$ and $C>0$ is a constant independent of $\delta,\eps$ and $C_{\star}$ in \cref{bootstrap}. Appealing to the bootstrap hypothesis \cref{bootstrap}, combining the inequality above with Lemma \ref{lem:keymom} and \cref{bd:m11boot}, we conclude that for $\delta$ sufficiently small, there exists a constant $\widetilde{C}_1>0$ independent of $\delta,\eps$ and $C_{\star}$ in \cref{bootstrap} such that
\begin{align}
\label{bd:F1}
    |{\rm F}_1|\leq \widetilde{C}_1\eps^4 +\frac{1}{100}D_{\eps}({\mathbf w}).
\end{align}
 This bound is consistent with \cref{bd:FAN}.

For the term ${\rm F}_2$,  from Lemma \ref{lem:keykernel} we get
\begin{align}
    {\rm F}_2=\frac{\mathfrak{b}_\eps}{\delta}\Big\langle\begin{pmatrix}
                \{|\cT_{\eps} \xi|^2,\Theta\}\\
                \{|\xi|^2,\widehat{\Theta}\}
            \end{pmatrix}, \begin{pmatrix}
                \mathbbm{1}_{{\rm I}_{\eps}} w\\
                \mathbbm{1}_{\widehat{{\rm I}}_{\eps}} \widehat{w}
\end{pmatrix}\Big\rangle+\frac{\mathfrak{b}_\eps}{\delta}\Big\langle\cM_{\eps} \begin{pmatrix}
                \{|\cT_{\eps} \xi|^2,\Omega_{\app,E}\}\\
                \{|\xi|^2,\widehat{\Omega}_{\app,E}\}
            \end{pmatrix}, \begin{pmatrix}
                \mathbbm{1}_{{\rm I}_{\eps}^c} w\\
                \mathbbm{1}_{\widehat{{\rm I}}_{\eps}^c} \widehat{w}
\end{pmatrix}\Big\rangle
\end{align}
In the interior region, we can combine the properties of $\Theta, \widehat{\Theta}$ in Proposition \ref{prop:functional} with properties of the weight in Lemma \ref{lem:Aeps} to see that 
\begin{align}
\Big|\Big\langle\begin{pmatrix}
                \{|\cT_{\eps} \xi|^2,\Theta\}\\
                \{|\xi|^2,\widehat{\Theta}\}
            \end{pmatrix}, \begin{pmatrix}
                \mathbbm{1}_{{\rm I}_{\eps}} w\\
                \mathbbm{1}_{\widehat{{\rm I}}_{\eps}} \widehat{w}
\end{pmatrix}\Big\rangle\Big|\lesssim \eps^{-1}\|\mathbf{A}_{\eps}^{-\frac12}(1+|\xi|)\nabla \boldsymbol{\Theta}\|_{\boldsymbol{\cX}_\eps}\|\mathbf w\|_{\boldsymbol{\cX}_{\eps}}\lesssim \eps^{M}\|\mathbf w\|_{\boldsymbol{\cX}_{\eps}}
\end{align}
Outside the interior region, we can exploit the fact that $\Omega_{\rm app,E},\Omega_{\rm app,E}\in \cZ$ and argue as in the proof of Lemma \ref{lem:MepsBeps} to exploit the exponential decay of $A_{\eps}^{-1}$ and handle the terms involving $\Delta^{-1}$ in $\cM_{\eps}$. Overall, we obtain that 
\begin{align}
    |{\rm F}_2|\lesssim |\mathfrak{b}_{\eps}|\delta^{-1}\eps^{M}\|\mathbf w\|_{\boldsymbol{\cX}_{\eps}}
\end{align}
    Appealing to Lemma \ref{lem:beta} and the bootstrap hypothesis \cref{bootstrap}, we see that 
    \begin{align}
    \label{bd:F2}
        |{\rm F}_2|\lesssim \delta^{-1}\eps^{M+8}\lesssim \delta \eps^9,
    \end{align}
    where in the last bound we used that in the time-scale under consideration $\delta^{-2}\eps^{M+1}\leq \eps^2$. This bound is consistent with \cref{bd:FAN} upon choosing $\delta$ sufficiently small.

\subsubsection*{\textbf{Step 2: bound on the advection term}} To handle this term, we make rigorous the arguments sketched in Section \ref{sec:ideas}. To highlight some crucial cancellation, we split
\begin{align}
\cM_\eps =\cM_{d}+\cM_{nl}, \qquad \cJ_{\rm app}=\cJ_{d}+\cJ_{nl}
\end{align}
where we define the ``diagonal" terms as
\begin{align}
    \cM_d=\begin{pmatrix}
        A_{\eps} & 0\\
        0 & \frac{\gamma}{N}\widehat{A}_\eps
    \end{pmatrix}, \qquad \cJ_d=\begin{pmatrix}
        \{\Phi_{\rm app,E}, \cdot\} & 0\\
        0 & \{\widehat{\Phi}_{\rm app,E}, \cdot\}
    \end{pmatrix}
\end{align}
and the ``nonlocal" terms as
\begin{align}
    \cM_{nl}&=\begin{pmatrix}
        (I+\cQ_{\eps})\Delta^{-1} & \frac{\gamma}{N}\cV_{\eps,R}^*\Delta^{-1}\\
        \frac{\gamma}{N}\cV_{\eps,R}\Delta^{-1} & \frac{\gamma^2}{N}\Delta^{-1}
    \end{pmatrix}, \\
    \cJ_{nl}&=\begin{pmatrix}
        \{(I+\cQ_{\eps})\Delta^{-1}(\cdot), \Omega_{\rm app,E}\} & \frac{\gamma}{N}\{\cV_{\eps,R}^*\Delta^{-1}(\cdot), \Omega_{\rm app,E}\}
        \\
  \{\cV_{\eps,R}\Delta^{-1}(\cdot), \widehat{\Omega}_{\rm app,E}\} & \gamma\{\Delta^{-1}(\cdot),\widehat{\Omega}_{\rm app,E}\}.
    \end{pmatrix}
\end{align}
Hence 
\begin{align}
    {\rm A}=-\frac{1}{\delta}\Big(\langle \cM_{nl}({\mathbf w}),\cJ_{nl}({\mathbf w})\rangle+\langle \cM_{d}({\mathbf w}),\cJ_{d}({\mathbf w})\rangle+\langle \cM_{d}({\mathbf w}),\cJ_{nl}({\mathbf w})\rangle+\langle \cM_{nl}({\mathbf w}),\cJ_{d}({\mathbf w})\rangle\Big)
\end{align}
Then, since $\langle \{f,g\},f\rangle=0$ and $\langle\{f,g\},h\rangle+\langle f,\{h,g\}\rangle=0$ for any $f,g,h$, note that 
\begin{align}
\label{bd:MnlJnl}
    \langle \cM_{nl}({\mathbf w}),\cJ_{nl}({\mathbf w})\rangle =0.
\end{align}
Similarly, since $\langle \{f,g\},hf\rangle=\frac12\langle \{g,h\}f,f\rangle$ we have 
\begin{align}
\label{def:MdJd}
\langle \cM_{d}({\mathbf w}),\cJ_{d}({\mathbf w})\rangle=\frac12\big\langle \{A_{\eps},\Phi_{\rm app,E}\}\,w,w\big\rangle +\frac{\gamma}{2N}\big\langle \{\widehat{A}_{\eps},\widehat{\Phi}_{\rm app,E}\}\,\widehat{w},\widehat{w}\big\rangle.
\end{align}
The bound for this term is analogous to the term ${\rm A}_1$ in \cite{dolce2024long}. However, we have to be a bit more careful for the treatment of the exterior region in view a potentially problematic term related to $\{|\xi|^2,\cdot\}$ and the fact that $\rho_\eps$ only controls $|\xi|^\gamma$ for $\gamma<1$. This problem is easily solved thanks to the radial symmetry of the weight $A_\eps$ in the exterior region.

We claim that
\begin{align}
\label{bd:MdJd}
\big|\langle \cM_{d}({\mathbf w}),\cJ_{d}({\mathbf w})\rangle\big|\lesssim \eps^{M+1-\sigma_1\widetilde{N}}\|{\mathbf w}\|^2_{\boldsymbol{\cX}_\eps}+\eps^{1+\sigma_2}\|\rho_\eps{\mathbf w}\|^2_{\boldsymbol{\cX}_\eps}.
\end{align}
Indeed, consider the first term on the right-hand side of \cref{def:MdJd}. Thanks to Proposition \ref{prop:functional}, in the interior region one has
\begin{align}
    \mathbbm{1}_{{\rm I}_\eps}\{A_\eps,\Phi_{\rm app,E}\}=\mathbbm{1}_{{\rm I}_\eps}\{F'(\Omega_{\rm app,E}),\Phi_{\rm app,E}\}=\mathbbm{1}_{{\rm I}_\eps}\{F'(\Omega_{\rm app,E}),\Theta\}=\mathbbm{1}_{{\rm I}_\eps}\{A_\eps,\Theta\}.
\end{align}
In view of Proposition \ref{prop:functional} and Lemma \ref{lem:Aeps}, we know that there exists $\widetilde{N}$ such that
\begin{align}
    \mathbbm{1}_{{\rm I}_\eps}|\{A_\eps,\Theta\}|\lesssim  \mathbbm{1}_{{\rm I}_\eps}(1+|\xi|)^{\widetilde{N}}\eps^{M+1}A_\eps\lesssim \eps^{M+1-\sigma_1\widetilde{N}}A_\eps,
\end{align}
meaning that 
\begin{align}
\big|\big\langle\mathbbm{1}_{{\rm I}_\eps} \{A_{\eps},\Phi_{\rm app,E}\}\,w,w\big\rangle\big|\lesssim \eps^{M+1-\sigma_1\widetilde{N}}\|w\|^2_{\cX_\eps}\end{align}
In region ${\rm II}_{\eps}$  the weight $A_\eps$ is constant and therefore $\{A_{\eps},\Phi_{\rm app,E}\}=0$ here. Thus, it remains to study the exterior region ${\rm III}_{\eps}.$ Since the weight $A_\eps$ is radially symmetric in this region, note that 
\begin{align}
    \mathbbm{1}_{{\rm III}_\eps}\{A_\eps, |\cT_{\eps}\xi|^2\}=  \mathbbm{1}_{{\rm III}_\eps}\frac{2R_{\rm app}}{\eps d}\{A_\eps,\xi_1\}
\end{align} Therefore, from the definition of $\Phi_{\rm app,E}$ in \cref{def:PhiappE} and since $\mathbbm{1}_{{\rm III}_\eps}|\nabla A_{\eps}|\lesssim \mathbbm{1}_{{\rm III}_\eps}|\xi|^{2\varsigma-1}A_\eps$ we have
\begin{align}
    \mathbbm{1}_{{\rm III}_\eps}|\{A_\eps, \Phi_{\rm app,E}\}|\lesssim \mathbbm{1}_{{\rm III}_\eps}|\xi|^{2\varsigma-1}A_\eps \big(|\nabla (1+\cQ_{\eps})\Psi_{\rm app,E}|+\gamma|\nabla \cT_\eps\widehat{\Psi}_{\rm app,E}|+|\cB_\eps(\Psi_{\rm app, E},{\widehat \Psi}_{\rm app, E})|\big).
\end{align}
Since $|\cB_\eps(\Psi_{\rm app, E},{\widehat \Psi}_{\rm app, E})|=\cO(\eps)$ and $|\xi|^{-1}\lesssim \eps^{\sigma_2} $ in region ${\rm III}_\eps$, appealing to Lemma \ref{lem:bounds} we can control the terms involving the approximate solution and deduce that 
\begin{align}
   \big| \langle\mathbbm{1}_{{\rm III}_\eps}\{A_\eps,\Phi_{\rm app,E}\} w,w \rangle\big||\lesssim  \eps^{\sigma_2}\|\rho_\eps w\|_{\cX_\eps}^2.
\end{align}
Thus, the desired claim \cref{bd:MdJd} is proved since for the second term in \cref{def:MdJd} one can repeat exactly the same arguments above.

It remains to handle the mixed terms. Integrating by parts, note that
\begin{align}
    \langle \cM_{nl}({\mathbf w}),\cJ_d({\mathbf w})\rangle=\,& \big\langle \big\{(I+\cQ_{\eps})\Delta^{-1}w+\frac{\gamma}{N}\cV_{\eps,R}^*\Delta^{-1}\widehat{w},\Phi_{\rm app,E}\big\},w\big \rangle\\
    &+\frac{\gamma}{N}\big\langle \big\{\cV_{\eps,R}\Delta^{-1}w+\gamma\Delta^{-1}\widehat{w},\widehat{\Phi}_{\rm app,E}\big\},\widehat{w}\big \rangle
\end{align}
On the other hand
\begin{align}
    \langle \cM_{d}({\mathbf w}),\cJ_{nl}({\mathbf w})\rangle=\,& \big\langle \big\{(I+\cQ_{\eps})\Delta^{-1}w+\frac{\gamma}{N}\cV_{\eps,R}^*\Delta^{-1}\widehat{w},F(\Omega_{\rm app,E})\big\},\mathbbm{1}_{{\rm I}_\eps}w\big \rangle\\
    &+\frac{\gamma}{N}\big\langle \big\{\cV_{\eps,R}\Delta^{-1}w+\gamma\Delta^{-1}\widehat{w},\widehat{F}(\widehat{\Omega}_{\rm app,E})\big\},\mathbbm{1}_{\widehat{{\rm I}}_\eps}\widehat{w}\big \rangle\\
    &+\langle (\mathbbm{1}_{{\rm I}_\eps^c}\cM_{d}({\mathbf w})_1,\mathbbm{1}_{\widehat{{\rm I}}^c_\eps}\cM_{d}({\mathbf w})_2),\cJ_{nl}({\mathbf w})\rangle.
\end{align}
Thus, by adding the mixed terms we  get
\begin{align}
    \langle &\cM_{d}({\mathbf w}),\cJ_{nl}({\mathbf w})\rangle+\langle \cM_{nl}({\mathbf w}),\cJ_{d}({\mathbf w})\rangle\\
    =\,& \big\langle \big\{(I+\cQ_{\eps})\Delta^{-1}w+\frac{\gamma}{N}\cV_{\eps,R}^*\Delta^{-1}\widehat{w},\Theta\big\},\mathbbm{1}_{{\rm I}_\eps}w\big \rangle+\frac{\gamma}{N}\big\langle \big\{\cV_{\eps,R}\Delta^{-1}w+\gamma\Delta^{-1}\widehat{w},\widehat{\Theta}\big\},\mathbbm{1}_{\widehat{{\rm I}}_\eps}\widehat{w}\big \rangle\\
    &
    +\big\langle \big\{(I+\cQ_{\eps})\Delta^{-1}w+\frac{\gamma}{N}\cV_{\eps,R}^*\Delta^{-1}\widehat{w},\Phi_{\rm app,E}\big\},\mathbbm{1}_{{\rm I}_\eps^c}w\big \rangle\\
    &+\frac{\gamma}{N}\big\langle \big\{\cV_{\eps,R}\Delta^{-1}w+\gamma\Delta^{-1}\widehat{w},\widehat{\Phi}_{\rm app,E}\big\},\mathbbm{1}_{\widehat{{\rm I}}^c_\eps}\widehat{w}\big \rangle\\
    &+\langle (\mathbbm{1}_{{\rm I}_\eps^c}\cM_{d}({\mathbf w})_1,\mathbbm{1}_{\widehat{{\rm I}}^c_\eps}\cM_{d}({\mathbf w})_2),\cJ_{nl}({\mathbf w})\rangle.
\end{align}
We see that from the functional relationship we gain smallness thanks to the presence of $\Theta,\widehat{\Theta}.$ 
Outside the interior regions, we can always gain smallness by the fast growth of the weight $A_\eps$. Therefore, appealing to Proposition \ref{prop:functional} and Lemma \ref{lem:bounds}, it follows that 
\begin{align}
\label{bd:MdJnl}
    \big|\langle \cM_{d}({\mathbf w}),\cJ_{nl}({\mathbf w})\rangle+\langle \cM_{nl}({\mathbf w}),\cJ_{d}({\mathbf w})\rangle\big|\lesssim \eps^{M+1} \|{\mathbf w}\|_{\boldsymbol{\cX}_\eps}^2.
\end{align}
Collecting the bounds \cref{bd:MnlJnl,bd:MdJd,bd:MdJnl}, we arrive at 
\begin{align}
    |{\rm A}|\lesssim \delta^{-1}\big(\eps^{M+1-\sigma_1N} \|{\mathbf w}\|_{\boldsymbol{\cX}_\eps}^2+\eps^{1+\sigma_2} \|\rho_\eps{\mathbf w}\|_{\boldsymbol{\cX}_\eps}^2\big).
\end{align}
Choosing $\sigma_1$ sufficiently small and $\sigma_2$ sufficiently large such that $\sigma_1<(M+1)/(2N)$ and $\sigma_2>(M-1)/2$, in the time-scale under consideration the bound above and Proposition  \ref{prop:D} imply that 
\begin{align}
    |{\rm A}|\leq C\delta^{s} (D_{\eps}({\mathbf w})+|\mrm_1^1(w)|^2),    
\end{align}
where $C>0$ is a large constant and $s\in (0,1)$ is a sufficiently small constant. Hence, in view of \cref{bd:m11boot}, the bound above is consistent with \cref{bd:FAN} if $\delta$ is sufficiently small.

\subsubsection*{\textbf{Step 3: bound on the Navier-Stokes and nonlinear term}}
We finally present the bound for the term ${\rm N}$ in \cref{def:Nen}. First of all, we control the term involving  $[t\de_t,\cK_\eps]$. By the definition of $\cK_\eps$ in \cref{def:Beps}, we know that this operator involves only translations and rotations composed with $\Delta^{-1}$, where the time dependency is only through the functions $\eps(t)$ and $R_{\rm app}(t)$. Therefore, when computing the commutator, we get terms of the type $f(t)Q(\nabla\cT_{\eps})\circ \tilde{Q}$ where $f(t)$ is related to $t\de_t$ derivatives of $\eps(t)$ and $R_{\rm app}(t)$, $Q,\tilde{Q}$ are rotation matrices and there is a shift related to rotations of $(R_{\rm app},0)$. Thus, we can appeal to properties of $R_{\rm app}$,  Lemma \ref{lem:MepsBeps} and use the bootstrap hypothesis \cref{bootstrap} to deduce that
\begin{align}
  |\langle{\mathbf w},[t\de_t,\cK_\eps]{\mathbf w}\rangle|\lesssim \eps^2(|\mrm^1_1(w)|+\eps\|{\mathbf w}\|_{\boldsymbol{\cX}_\eps})\|{\mathbf w}\|_{\boldsymbol{\cX}_\eps}\lesssim \eps^7.
\end{align} 
In the time scale under consideration, this bound is consistent with \cref{bd:FAN} when $\delta$ is sufficiently small.

Then, let us control for instance the term arising from $\cN(w,\widehat{w})$ defined in \cref{def:N}, since the one in \cref{def:hN} can be treated in a completely analogous way. Namely, we aim at controlling
\begin{align}
    \sum_{j=1}^3\langle \cN_j(w,\widehat{w}),(\cM_\eps {
    \mathbf w})_1\rangle=\sum_{j=1}^3 {\rm N}_j, 
\end{align}
where $\cN_1,\cN_2,\cN_3$ are defined in \cref{def:N1,def:N2,def:N3}.
The term ${\rm N}_1$ is similar to the Navier-Stokes error handled in \cite{dolce2024long}. Indeed, it is enough to combine $\Phi_{\rm app, NS}=\cO_{\cS_*}(\eps^2)$  $\Omega_{\rm app, NS}=\cO_{\cZ}(\eps^2)$ with  Lemma \ref{lem:bounds} and Proposition \ref{prop:D} to infer that 
\begin{align}
    |{\rm N}_1|&\lesssim \eps^2 \|{\mathbf w}\|_{\boldsymbol{\cX}_\eps}(\|{\mathbf w}\|_{\boldsymbol{\cX}_\eps}+\|{\nabla\mathbf w}\|_{\boldsymbol{\cX}_\eps})\lesssim \eps^4\sqrt{D_{\eps}({\mathbf w})+|\mrm_1^1(w)|^2}\\
    &\lesssim \eps^6+\eps^2(D_{\eps}({\mathbf w})+|\mrm_1^1(w)|^2)\lesssim \eps^6+\eps^2D_{\eps}({\mathbf w}),
\end{align}
where in the last bound we used the bootstrap hypothesis \cref{bootstrap} (and Remark \ref{rem:momXeps}). This  is again in agreement with \cref{bd:FAN} if  $\delta$ is sufficiently small.

For the nonlinear term ${\rm N}_2$, note that 
\begin{align}
    {\rm N}_2&=-\big\langle\{(I+\cQ_{\eps})\varphi+\frac{\gamma}{N} \cV_\eps^*\widehat{\varphi} ,w\},A_\eps w+(I+\cQ_{\eps})\varphi+\frac{\gamma}{N} \cV_\eps^*\widehat{\varphi} \big\rangle\\
    &=\frac12\big\langle\{(I+\cQ_{\eps})\varphi+\frac{\gamma}{N} \cV_\eps^*\widehat{\varphi} ,A_\eps\}w, w\big\rangle.
\end{align}
Now, we can argue as in the treatment of the nonlinear term in \cite{dolce2024long} and deduce that 
\begin{align}
    |{\rm N}_2|\lesssim \sqrt{E_{\eps}({\mathbf w})+|\mrm_1^1(w)|^2}(D_\eps({\mathbf w})+|\mrm_1^1(w)|^2)\lesssim \eps(D_\eps({\mathbf w})+\eps^6),
\end{align}
where we used the bootstrap hypothesis \cref{bootstrap} in the last inequality.
Finally, for ${\rm N}_3$ we have to be a bit careful for the term involving $w$. We write
\begin{align}
    &\langle \{|\cT_{\eps}\xi|^2,w\},(\cM_\eps(\mathbf{w}))_1\rangle=\langle \{|\cT_{\eps}\xi|^2,w\},A_\eps w\rangle+\langle \{|\cT_{\eps}\xi|^2,w\},(I+\cQ_\eps)\Delta^{-1} w+\frac{\gamma}{N}\cV_\eps^*\Delta^{-1}\widehat{w}\rangle\\
   \label{bdN3} &\qquad =\frac12 \langle \{|\cT_{\eps}\xi|^2,A_\eps\}w,w\rangle-\langle \{|\cT_{\eps}\xi|^2,(I+\cQ_\eps)\Delta^{-1} w+\frac{\gamma}{N}\cV_\eps^*\Delta^{-1}\widehat{w}\},w\rangle.
\end{align}
For the first term, we see that since $A_0$ is radially symmetric, we can write
\begin{align}
    \mathbbm{1}_{{\rm I}_\eps}\{|\cT_{\eps}\xi|^2,A_{\eps}\}=\mathbbm{1}_{{\rm I}_\eps}\{|\xi|^2,A_{\eps}-A_0\}+\mathbbm{1}_{{\rm I}_\eps}
    \frac{2 R_{\rm app}}{\eps d}\de_2A_{\eps}
\end{align}
By Lemma \ref{lem:Aeps}, since $|\xi|\lesssim \eps^{-\sigma_1}$ in ${\rm I}_\eps$, we deduce that 
\begin{align}
       \mathbbm{1}_{{\rm I}_\eps}|\{|\cT_{\eps}\xi|^2,A_{\eps}\}|\lesssim  \mathbbm{1}_{{\rm I}_\eps}(\eps A_\eps+\eps^{-1}|\xi|A_\eps)\lesssim\mathbbm{1}_{{\rm I}_\eps} (\eps +\eps^{-1}\rho_\eps)A_{\eps}.
\end{align}
Similarly, since $A_{\eps}$ is constant on ${\rm II}_{\eps}$ and radial on ${\rm III}_\eps$, we have
\begin{align}
    \mathbbm{1}_{{\rm I}^c_\eps}|\{|\cT_{\eps}\xi|^2,A_{\eps}\}|=\mathbbm{1}_{{\rm I}^c_\eps}
    \frac{2 R_{\rm app}}{\eps d}|\de_2A_{\eps}|\lesssim \eps^{-1+\sigma_2} \mathbbm{1}_{{\rm III}_\eps}\rho_\eps A_{\eps}.
\end{align}
For the remaining terms in \cref{bdN3} (the ones involving the $\Delta^{-1}$) we can combine \cref{bd:m11boot} in Lemma \ref{lem:bounds} with H\"older's inequality and the fast growth of $A_\eps$ to get that 
\begin{align}
     |\langle \{|\cT_{\eps}\xi|^2,w\},(\cM_\eps(\mathbf{w}))_1\rangle|&\lesssim \eps^{-1} \|\mathbf{w}\|_{\cX_\eps}(\|\mathbf{w}\|_{\cX_\eps}+\|\rho_\eps\mathbf{w}\|_{\cX_\eps})
\end{align}
For the terms in ${\rm N}_3$ involving the Navier-Stokes part of the approximate solution, we can simply use $\Omega_{\rm app,NS}=\cO_\cZ(\eps^2)$. Thus, appealing to Lemma \ref{lem:beta}, Proposition \ref{prop:D} and using the bootstrap hypothesis \cref{bootstrap} we finally see that
\begin{align}
    |{\rm N}_3|&\lesssim |\mathfrak{b}_\eps|(\eps \|{\mathbf w}\|_{\boldsymbol{\cX}_\eps}+\eps^{-1}\|{\mathbf w}\|_{\boldsymbol{\cX}_\eps}(\|{\mathbf w}\|_{\boldsymbol{\cX}_\eps}+\|\rho_\eps{\mathbf w}\|_{\boldsymbol{\cX}_\eps}))\\
    &\lesssim \eps^9+\eps^7\sqrt{D_{\eps}({\mathbf w})+|\mrm^1_1(w)|^2}\lesssim \eps^9+\eps^2D_{\eps}({\mathbf w}).
\end{align}
This bound is also in agreement with \cref{bd:FAN} if  $\delta$ is sufficiently small. Therefore, recalling that the terms involving $\widehat{\cN}(w,\widehat{w})$ can be handled similarly, we conclude that the whole ${\rm N}$ satisfy a bound consistent with \cref{bd:FAN} and therefore the proof is over. \qed

\section{Numerical simulations}\label{sec:sim}

The aim of this section is to present numerical simulations with various combinations of parameters $N$ and $\gamma$ and check that the behaviour of the deformation of vortices matches what is predicted by the computation of $\Omega_2$ and $\widehat{\Omega}_2$. 

All the following simulations, as well those presented in Figure~\ref{fig:deformation_sim} are made using the centered Navier-Stokes solver of the free software \href{http://basilisk.fr/}{Basilisk}. We use $\delta = 2\cdot 10^{-4}$, and instead of Dirac masses, we consider the initial vortices to be concentrated Gaussian functions, so that at time 0 we have that, depending on the exact choice of parameters for each simulation, $0.02 \lesssim \eps(0) \lesssim 0.05$ (except when $N \ge 10$ where $d$ becomes small, increasing the value of $\eps(0)$). 
Let us recall from \cite{donati2025fast} that it is expected that the solution of the Navier-Stokes equation with such initial data relaxes quickly to the solution of the problem with Dirac masses as initial datum. Except for the initial time, where vortices are radially symmetric, the pictures shown below are always taken after that transient process is completed.

\subsection{Different situations when \texorpdfstring{$N=2$}{N = 2}}
We start with displaying the simulation of the pair of co-rotating vortices, corresponding to $N=2$ and $\gamma = 0$, in Figure~\ref{fig:pair}. The expected two-fold symmetric deformation is visibly clear.
\begin{figure}[hbtp]
\centering
\begin{subfigure}{0.20\linewidth}
  \includegraphics[width=\linewidth]{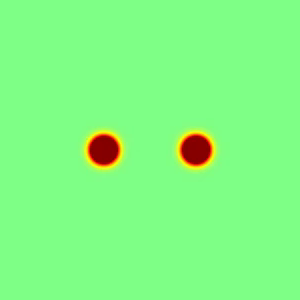}
\end{subfigure}\hspace{1mm}
\begin{subfigure}{0.20\linewidth}
  \includegraphics[width=\linewidth]{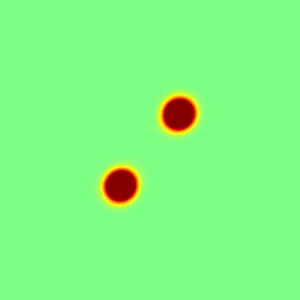}
\end{subfigure}\hspace{1mm}
\begin{subfigure}{0.20\linewidth}
  \includegraphics[width=\linewidth]{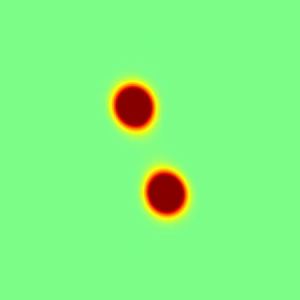}
\end{subfigure}\hspace{1mm}
\begin{subfigure}{0.20\linewidth}
  \includegraphics[width=\linewidth]{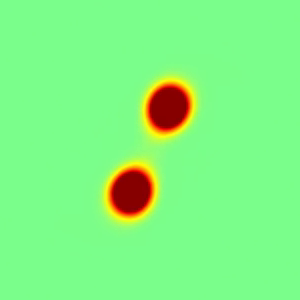}
\end{subfigure}
\caption{The pair of initially Gaussian vortices rotates and deforms.}
\label{fig:pair}
\end{figure}

Taking now $\gamma = 1$, we see a similar behaviour, with the central vortex also deforming heavily due to the strain, see Figure~\ref{fig:tripole}.
\begin{figure}[hbtp]
\centering
\begin{subfigure}{0.20\linewidth}
  \includegraphics[width=\linewidth]{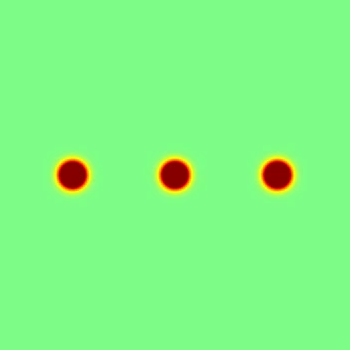}
\end{subfigure}\hspace{1mm}
\begin{subfigure}{0.20\linewidth}
  \includegraphics[width=\linewidth]{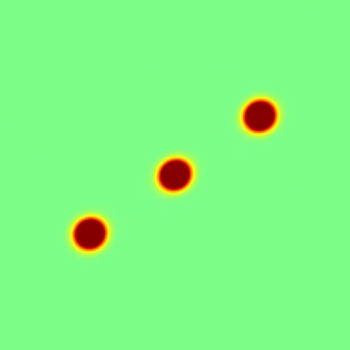}
\end{subfigure}\hspace{1mm}
\begin{subfigure}{0.20\linewidth}
  \includegraphics[width=\linewidth]{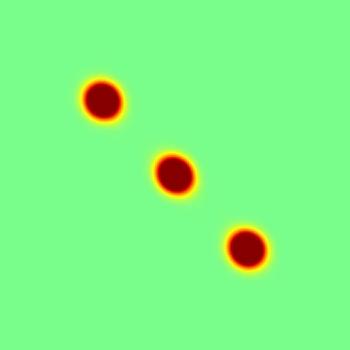}
\end{subfigure}\hspace{1mm}
\begin{subfigure}{0.20\linewidth}
  \includegraphics[width=\linewidth]{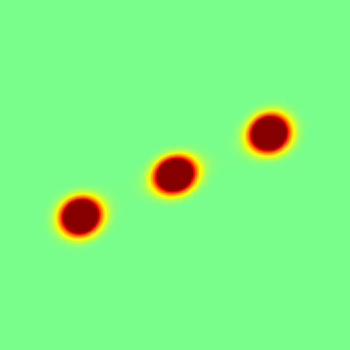}
\end{subfigure}
\caption{Three identical vortices rotate and deform.}
\label{fig:tripole}
\end{figure}

Let us recall that only when $N=2$ can $\widehat{\Omega}_2$ be non-vanishing. Moreover, unlike $\Omega_2$, no value of $\gamma$ can ensure $\widehat{\Omega}_2 = 0$ either (except $\gamma =0$). This means that for $N=2$ and $\gamma \neq 0$, we expect the central vortex to always display 2-fold symmetric deformation. For the outer vortices, the deformation changes direction at the critical value $\gamma_2^* = - \frac{1}{4}$. All three cases are displayed in Figure~\ref{fig:tripoles}.
\begin{figure}[hbtp]
\centering
\begin{subfigure}{0.22\linewidth}
  \includegraphics[width=\linewidth]{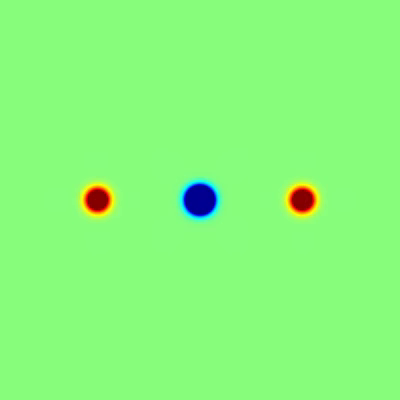}
\end{subfigure}\hspace{4mm}
\begin{subfigure}{0.22\linewidth}
  \includegraphics[width=\linewidth]{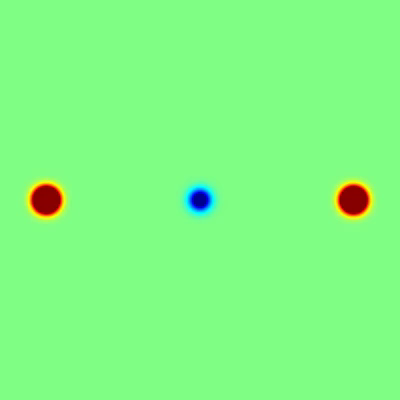}
\end{subfigure}\hspace{4mm}
\begin{subfigure}{0.22\linewidth}
  \includegraphics[width=\linewidth]{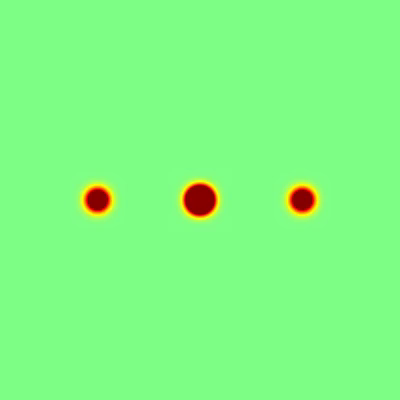}
\end{subfigure}

\vspace{1mm}

\begin{subfigure}{0.22\linewidth}
  \includegraphics[width=\linewidth]{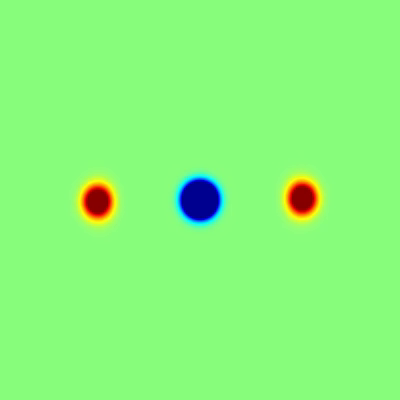}
\end{subfigure}\hspace{4mm}
\begin{subfigure}{0.22\linewidth}
  \includegraphics[width=\linewidth]{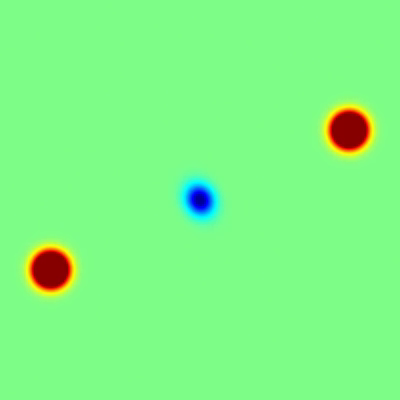}
\end{subfigure}\hspace{4mm}
\begin{subfigure}{0.22\linewidth}
  \includegraphics[width=\linewidth]{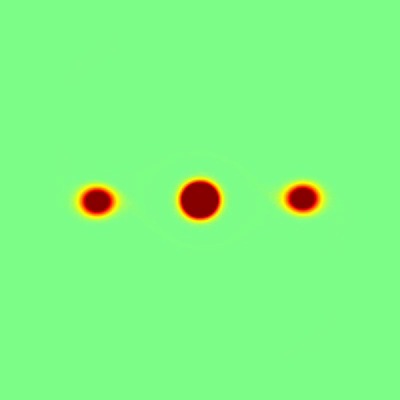}
\end{subfigure}
\caption{Initial configuration (top) evolves in time (bottom). In the cases $\gamma = -3$ (left), $\gamma = \gamma_2^* = - \frac{1}{4}$ (middle) and $\gamma = 3$ (right), after half of a rotation, we see the deformation of the outer vortices in the non-critical cases (left and right) with different orientations depending on the case, and a less visible but present deformation of the central vortex. In the critical case, after just a fraction of rotation, the central vortex has largely deformed while the outer vortices have not.}
\label{fig:tripoles}
\end{figure}

\subsection{The \texorpdfstring{$N=5$}{N = 5} case}

An interesting particular situation is taking $N=5$, as in that case the critical value $\gamma_5^* = 0$. This means that we expect the pentagonal vortex crystal, without central vortex, to not display any clear 2-fold symmetric deformation. This is visible in Figure~\ref{fig:pentagon}. Even after a full rotation of the configuration, and although the diffusive effects have significantly spread the vortices, they remain mostly radially symmetric. Only after the third, and last before vortex merging, rotation the next order deformation becomes clearly visible.
\begin{figure}[hbtp]
\centering
\begin{subfigure}{0.20\linewidth}
  \includegraphics[width=\linewidth]{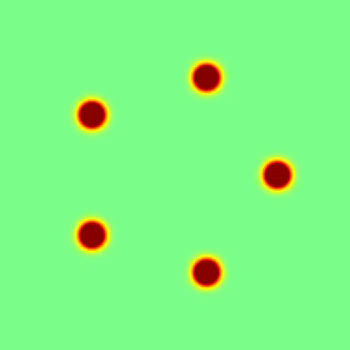}
\end{subfigure}\hspace{1mm}
\begin{subfigure}{0.20\linewidth}
  \includegraphics[width=\linewidth]{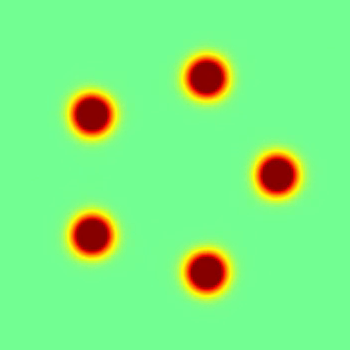}
\end{subfigure}\hspace{1mm}
\begin{subfigure}{0.20\linewidth}
  \includegraphics[width=\linewidth]{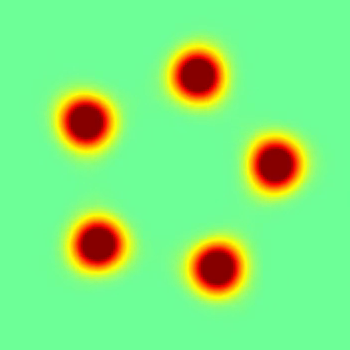}
\end{subfigure}\hspace{1mm}
\begin{subfigure}{0.20\linewidth}
  \includegraphics[width=\linewidth]{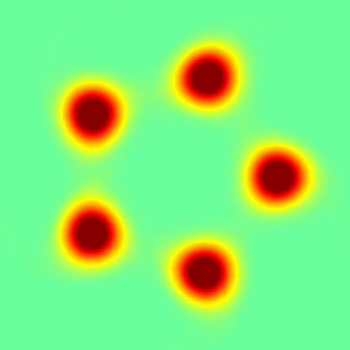}
\end{subfigure}
\caption{The pentagonal vortex crystal does not display any clear 2-fold symmetric deformation. After a very long time, close to the beginning of vortex merging, the 3-fold symmetric deformation starts to be visible. Each picture is taken after a full rotation of the configuration. It is interesting to compare the shape of these vortices with the ones found in \cite{Dritschel1985} for a pentagonal configuration.}
\label{fig:pentagon}
\end{figure}

\subsection{Stability of the polygonal vortex crystal}\label{sec:stability}

Before going further, let us quickly discuss the stability of these polygonal vortex crystal.
For general details on how to compute the stability properties of configurations of point-vortices, we refer the reader to \cite{Rob2013}. The stability of our particular vortex crystal of interest has been studied in \cite{CabSch1999}, which states the following result.
\begin{theorem}[{\cite[Theorem 5.1]{CabSch1999}}]\label{theo:stabCS}
    If the intensity of the central vortex satisfies that
    \begin{equation}
        \begin{split}
            \displaystyle  \frac{N^2-8N+8}{16}< \gamma < \frac{(N-1)^2}{4}& \text{ if $N$ is even,} \vspace{3mm} \\
            \displaystyle  \frac{N^2-8N+7}{16} <\gamma < \frac{(N-1)^2}{4} & \text{ if $N$ is odd,}
        \end{split}
    \end{equation}
    then this equilibrium is nonlinearly stable to perturbations that preserve $A$.
\end{theorem}
We mention that according to \cite{LP05}, the equilibrium is nonlinearly stable only under the additional assumption to the hypotheses of Theorem~\ref{theo:stabCS} that $\gamma \ge 0$, which makes a difference in the case $N \le 6$. Notice that changing $\gamma$ changes the angular velocity of the configuration, so in general this equilibrium can only be orbitally nonlinearly stable. If $\gamma =0$, then the Theorem~\ref{theo:stabCS} states that the equilibrium is nonlinearly stable only for $N \le 6$. The case $N=7$ is a degenerate equilibrium (see \cite[Section 7]{CabSch1999}) and for $N\ge 8$, the configuration is unstable. See also the results in \cite{aref2002vortex,Barry12} for other interesting relative equilibria with $N+1$-vortices.

\medskip

In Theorem~\ref{theo:main}, which is the main result of the paper, we prove in particular that the vortices remain close to the point-vortex solution at least for times $t \ll T_\adv \delta^{-2/3}$ (see Remark~\ref{rem:delta23}), and orbitally close for times $t \le T_\adv \delta^{-\sigma}$, for any value of $\gamma$. The mechanism that prevents possible point-vortex instabilities from appearing is the $N$-fold symmetry. With the additional assumption that the perturbation should preserve the $N$-fold symmetry, the point vortex stability problem is completely trivialized, and the configuration is always orbitally stable, for any value of $\gamma$. However, this means that Theorem~\ref{theo:main} strongly relies on the $N$-fold symmetry, and necessarily fails to be true when $\gamma$ is outside of the stability range, if for instance one changes the initial position of the vortices in an unstable $2(N+1)$-dimensional direction. 

\subsection{Larger values of \texorpdfstring{$N$}{N}}

The result and behaviour does not change in principle for larger values of $N$. For $N=10$ for instance, we compute that $\gamma_{10}^* = 3$, and thus for various values of $\gamma$ compared to $\gamma_{10}^*$ we see the deformation of the outer vortices changing direction, see Figure~\ref{fig:N10gneq0}. However, one has to be more careful with instabilities. Recall from Theorem~\ref{theo:stabCS} that for $N=10$ one has a threshold $\gamma_{\rm min} = \frac{7}{4}$, which does not leave much room to observe the deformation for $\gamma < \gamma_{10}^*$ without triggering instabilities. In particular, for $\gamma =0$, we see in Figure~\ref{fig:N18} that the instability kicks in very early, the symmetry is broken and vortices start merging by pairs. Taking even larger values of $N$, the same applies, and in particular, we recover the vortex sheet instability as discussed in Remark~\ref{rem:sheet}.

\begin{figure}[hbtp]
\centering
\begin{subfigure}{0.22\linewidth}
  \includegraphics[width=\linewidth]{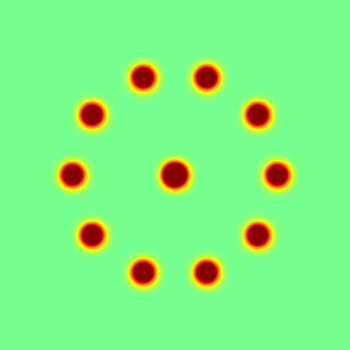}
\end{subfigure}\hspace{6mm}
\begin{subfigure}{0.22\linewidth}
  \includegraphics[width=\linewidth]{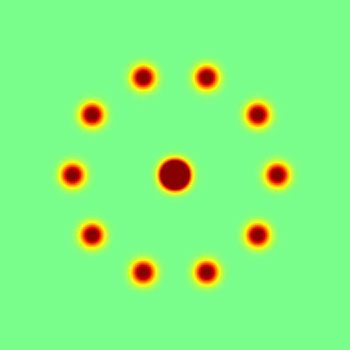}
\end{subfigure}\hspace{6mm} 

\vspace{1mm}

\begin{subfigure}{0.22\linewidth}
  \includegraphics[width=\linewidth]{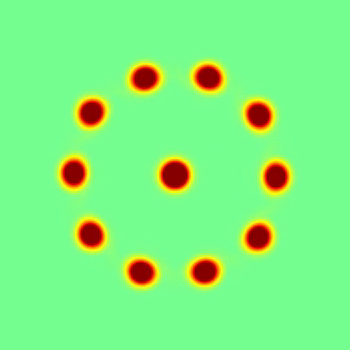}
\end{subfigure}\hspace{6mm}
\begin{subfigure}{0.22\linewidth}
  \includegraphics[width=\linewidth]{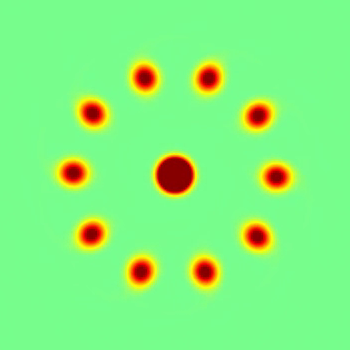}
\end{subfigure}\hspace{6mm}
\caption{With $N=10$, in the case $\gamma =5$ (right) the deformation towards the center is well visible, more than in the case $\gamma = \frac{3}{2}$ which is close to $\gamma^* = 3$, but already below the stability threshold $\gamma_{\rm min} = \frac{7}{4}$.}
\label{fig:N10gneq0}
\end{figure}

\begin{figure}[hbtp]
\centering
\begin{subfigure}{0.20\linewidth}
  \includegraphics[width=\linewidth]{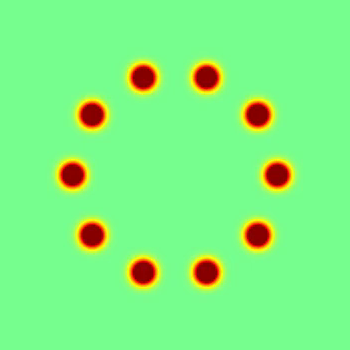}
\end{subfigure}\hspace{0.5mm}
\begin{subfigure}{0.20\linewidth}
  \includegraphics[width=\linewidth]{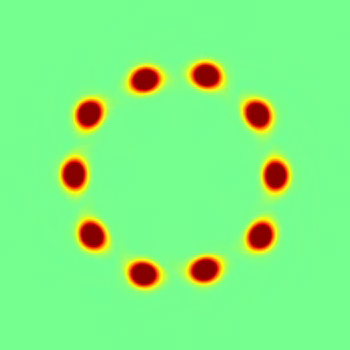}
\end{subfigure}\hspace{0.5mm}
\begin{subfigure}{0.20\linewidth}
  \includegraphics[width=\linewidth]{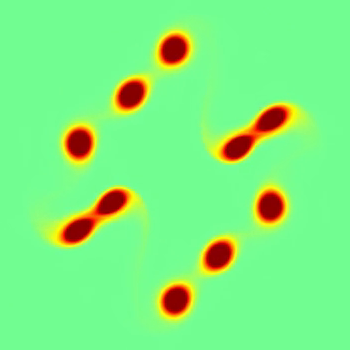}
\end{subfigure}\hspace{0.5mm}
\begin{subfigure}{0.20\linewidth}
  \includegraphics[width=\linewidth]{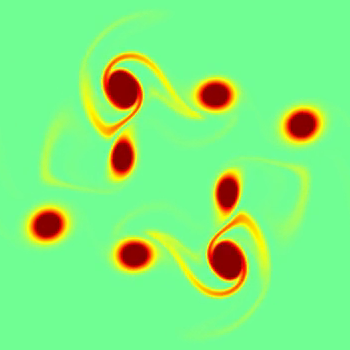}
\end{subfigure}

\vspace{4mm}

\begin{subfigure}{0.20\linewidth}
  \includegraphics[width=\linewidth]{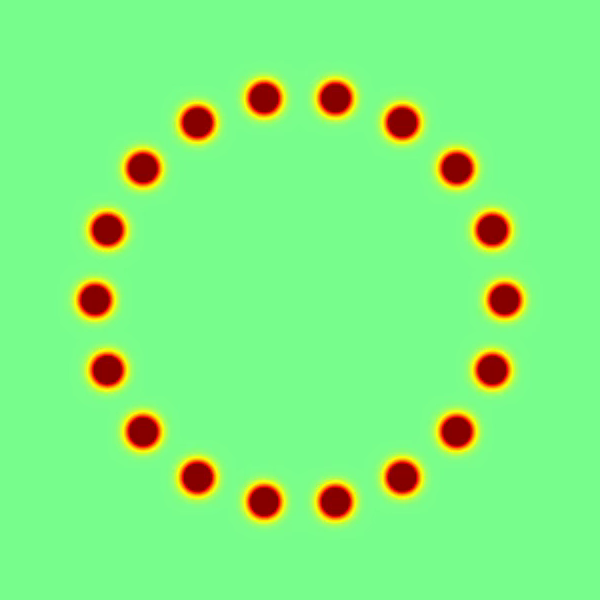}
\end{subfigure}\hspace{0.5mm}
\begin{subfigure}{0.20\linewidth}
  \includegraphics[width=\linewidth]{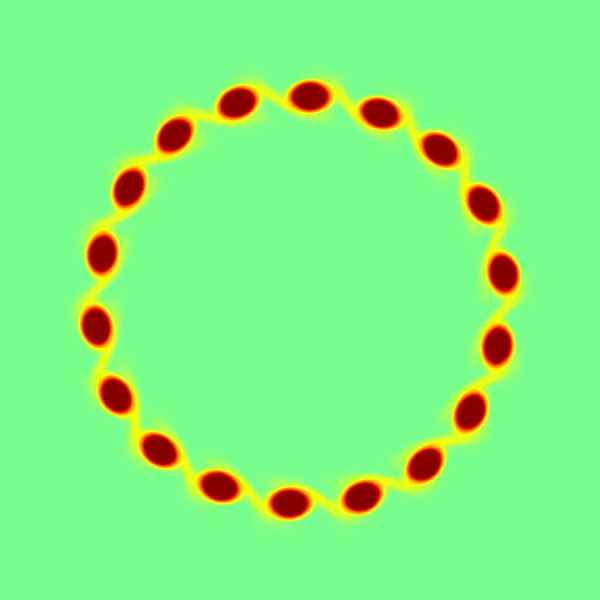}
\end{subfigure}\hspace{0.5mm}
\begin{subfigure}{0.20\linewidth}
  \includegraphics[width=\linewidth]{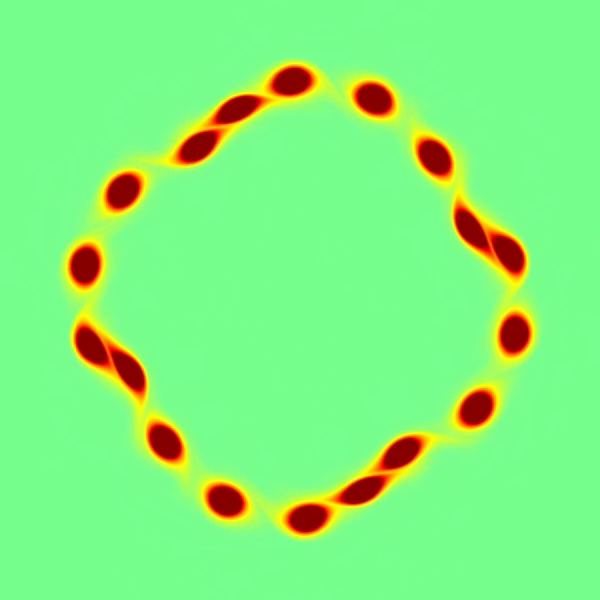}
\end{subfigure}\hspace{0.5mm}
\begin{subfigure}{0.20\linewidth}
  \includegraphics[width=\linewidth]{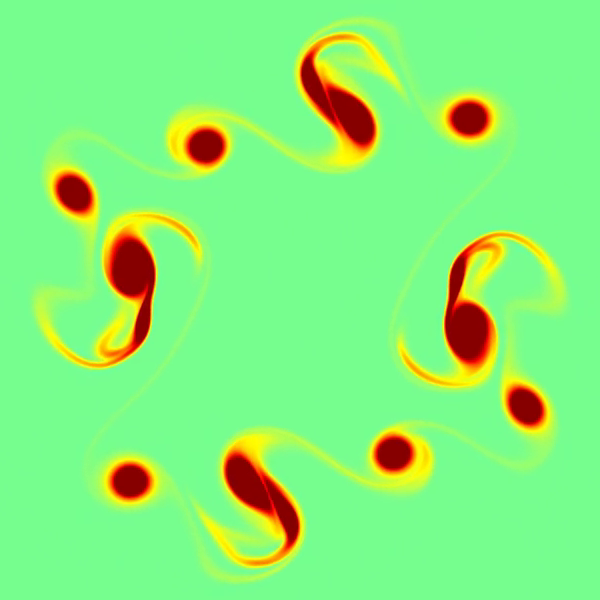}
\end{subfigure}
\vspace{1mm}
\caption{For $N =10$ (top) or $N=18$ (bottom) with $\gamma =0$, the expected deformation happens, then the typical vortex sheet instability develop very fast. Taking more concentrated vortices would delay this instability, and this is related to the joint $N\to \infty$ and $\delta\to0 $ limit discussed in Remark~\ref{rem:sheet}, corresponding essentially to the first two columns above. However, fixing $\delta$ while taking $N \to \infty$ the Kelvin-Helmholtz type instability, visible in the third and fourth column above, is the expected behavior as soon as $t > 0$.}
\label{fig:N18}
\end{figure}

\FloatBarrier

\appendix

\section{Function spaces and operators}\label{app:functional}
In this appendix we give the details of the proofs in Section \ref{sec:Func}.

\subsection{Preliminary: relevant constants and roots of unity}\label{app:constants}
We first establish formulas to compute the constants ${\rm S}_{n,k}$ explicitly, which we recall are defined in \cref{def:Snk}.
\begin{lemma}
\label{lem:Snk}
    For any $0 \le k \le N$, there holds that
    \begin{equation}
\label{eq:Snn}
    {\rm S}_{0,0
    }=N-1, \qquad {\rm S}_{n,0
    }=-\frac{1}{(n-1)!}\Big(\frac{p(z)'}{p(z)}\Big)^{(n-1)}\Big|_{z=1} \in \RR \qquad \text{for } n\geq1
\end{equation}
where $p(z)=\frac{1-z^N}{1-z}$, and
\begin{equation}
   {\rm  S}_{n,k} = \sum_{j=0}^k \binom{k}{j} {\rm  S}_{n-k+j,0}.
\end{equation}
\end{lemma}
\begin{proof}
We rewrite the polynomial $p(z)$ as \begin{equation}
\label{def:pz}
    p(z)=\frac{1-z^N}{1-z}=1+z+\dots+z^{N-1}=(z-Q_1)\dots (z-Q_{N-1}),
\end{equation}
and thus
\begin{equation}
    \frac{p'(z)}{p(z)}=\sum_{l=1}^{N-1}\frac{1}{z-Q_l}.
\end{equation}
Differentiating, we have that
\begin{equation}
    \Big(\frac{p(z)'}{p(z)}\Big)^{(m)}=(-1)^m m!\sum_{l=1}^{N-1}\frac{1}{(z-Q_l)^{m+1}}.
\end{equation}
Therefore, for $n\geq 1 $, we have
\begin{equation}
\label{eq:Snn}
    {\rm S}_{n,0
    }=(-1)^n\sum_{l=1}^{N-1}\frac{1}{(1-Q_\ell)^n }=-\frac{1}{(n-1)!}\Big(\frac{p(z)'}{p(z)}\Big)^{(n-1)}\Big|_{z=1} \in \RR.
\end{equation}
Moreover 
\begin{equation}
\label{eq:S00}
    {\rm S}_{0,0}=N-1.
\end{equation}
To compute the value of ${\rm S}_{n,k}$, we first observe that for $1 \le k \le n$,
\begin{equation}
    {\rm S}_{n,k}=\sum_{l=1}^{N-1}\frac{Q_l^k}{(Q_l-1)^n}=\sum_{l=1}^{N-1}\frac{Q_\ell^{k-1}}{(Q_l-1)^{n-1}}+\sum_{l=1}^{N-1}\frac{Q_\ell^{k-1}}{(Q_l-1)^{n}}={\rm S}_{n-1,k-1}+{\rm S}_{n,k-1}.
\end{equation}
We recognize a Pascal's triangle-type formula, and one can prove by  induction that for every $0 \le k \le n$,
\begin{equation}
   {\rm  S}_{n,k} = \sum_{j=0}^k \binom{k}{j} {\rm  S}_{n-k+j,0}.
\end{equation}
\end{proof}
In view of Lemma \ref{lem:Snk}, we can compute recursively every value of ${\rm S}_{n,k}$ for $0 \le k \le n$. Let us give the result for some relevant values.
\begin{align}
    &{\rm S}_{1,0}=-\frac{p'(1)}{p(1)}=-\frac{N-1}{2}, \\
    &{\rm S}_{1,1}=\frac{N-1}{2}\label{eq:S11}\\
    &{\rm S}_{2,0}=\Big(\frac{p'(1)}{p(1)}\Big)^2-\frac{p''(1)}{p(1)}=-\frac{(N-1)(N-5)}{12}\\
    &{\rm S}_{2,1}={\rm S}_{1,0}+{\rm S}_{2,0}=\frac{1-N^2}{12}\\
    &{\rm S}_{2,2}={\rm S}_{0,0}+2{\rm S}_{1,0}+{\rm S}_{2,0}={\rm S}_{2,0}.\label{eq:S22}
\end{align}

\subsection{Expansion of the configurational operators}

Recall the definition of $\cT_{\eps,Z}$, $\cV_{\eps,Z}$ and $\cQ_{\eps,Z}$ defined in~\cref{def:Teps,def:Veps,def:Qeps}.
We then prove a more general version of Lemma \ref{lem:expansionsR}.
\begin{lemma}\label{lem:Expansions}
    Let $f \in \cZ$,  $\varphi = \Delta^{-1}f$, $0<z_{\min}\leq |Z|$ and $M\in \NN$. Then, the translation operator admits the expansion
    \begin{equation}
        \cT_{\eps,Z}\varphi(\xi) \,=\, \frac{{\rm M}(f)}{2\pi}\log\Big(\frac{|Z|}{\eps d}\Big)
    +\sum_{n=1}^M\sum_{k=0}^{n} \eps^nc_{n,k} \Re\Big(\frac{\xi^{n-k}}{Z^n}{\rm m}^k(f)\Big) + \cO_{\cS_*}\bigl(\epsilon^{M+1}\bigr)\,.
    \end{equation}
    The polygon-center operator admits the expansion
    \begin{equation}
  \begin{split}
      \cV_{\eps,Z}\varphi(\xi) \,=\,& \mrM(\Omega)\frac{N}{2\pi}\log\Big(\frac{|Z|}{\eps d}\Big)
  +N\sum_{n=1}^M (-1)^n\eps^nc_{n,n} \Re\Big(\frac{1}{Z^n}{\rm m}^n(f)\Big)\\
&+N\sum_{n=N}^M\sum_{\substack{\ell \in \mathbb{N}\setminus \{0\}\\
\ell N\leq M}}(-1)^n\eps^n c_{n,n-\ell N} \Re\Big(\frac{\xi^{\ell N}}{Z^n}\mrm^{n-\ell N}(f)\Big)+\cO_{\cS_*}\bigl(\epsilon^{M+1}\bigr)\,.
        \end{split}
\end{equation}
    The polygon-polygon operator admits the expansion
\begin{equation}
  \begin{split}
\cQ_{\eps,Z}\varphi(\xi) \,=&\, \frac{\mrM(\Omega)}{2\pi}\sum_{l=1}^{N-1}\log\Big(\frac{|Z_l|}{\eps d}\Big)
  +\sum_{n=1}^M \sum_{k=0}^n\epsilon^nc_{n,k}{\rm S}_{n,n-k}\Re\Bigg(\frac{\xi^{n-k}}{Z^n}\mrm^k(f)\Bigg)+\cO_{\cS_*}(\eps^{M+1}).
        \end{split}
\end{equation}
\end{lemma}


\begin{proof}
    By \cref{def:psiom} and the definition \cref{def:Teps}, we have
\begin{align}
\cT_{\eps,Z}\varphi(\xi) \,&=\, \Phi\big(\xi+\frac{1}{\eps d}Z\big) \,=\,
  \frac{1}{2\pi}\int_{\RR^2}\log\big|\big(\xi+\frac{1}{\eps d}Z\big)-\eta\big|\,f(\eta)\,\dd\eta \\
  \,&=\, \frac{\mrM(f)}{2\pi}\,\log\Big(\frac{|Z|}{\eps d}\Big)
  + \frac{1}{2\pi}\int_{\RR^2}\log\Bigl(\Big|1+\frac{\eps d}{Z}(\xi-\eta)\Big|\Bigr)\,f(\eta)\,\dd\eta\,.    \end{align}
We thus have to study the last integral. Since $\cT_{\eps,z} \varphi \in \cS_{*}(\RR^2)$, it is enough to consider $\xi \in B_{r_\eps}$ with $r_\eps:=|z_{\min}|/(4\eps d)$. Indeed, for $\xi \in B_{r_\eps}^c$ the function is automatically of order $\eps^{M+1}$. This follows because for any $f\in \cS_{*}(\RR^2)$ we know that there exists an $N$ such that for all $|\xi|\geq r_\eps$ we have $|f(\xi)|\leq (1+|\xi|)^{N}\leq (1+|\xi|)^{N+M+1}r_\eps^{-M-1}$, which implies that $f|_{B_{r_\eps}^c}=\cO_{\cS_*}(\eps^{M+1})$. 
Similarly, we can restrict the domain of integration so that $\eta \in B_{r_\eps}$, 
up to negligible errors. Hence, if $\xi,\eta \in B_{r_\eps}$, we have $|(\eps d/Z)(\xi-\eta)|<1$. We can then use the expansion
\begin{equation} \log|1+x|=\Re\log(1+x)=\sum_{n=1}^\infty \frac{(-1)^{n-1}}{n}
  \,\Re(x^n)\, \qquad \text{ valid if }|x| < 1\,,
\end{equation} 
to get
\begin{equation}
 \cT_{\eps,Z}\varphi(\xi)=\frac{\mrM(f)}{2\pi}\,\log\Big(\frac{|Z|}{\eps d}\Big)
  + \sum_{n=1}^M
  \frac{(-1)^{n-1}}{n} \frac{(\eps d)^n}{2\pi}\int_{B_{r_\eps}}\Re\bigl(Z^{-1}(\xi-\eta)\bigr)^n
  \,f(\eta)\,\dd\eta + \cO_{\cS_*}\bigl(\epsilon^{M+1}\bigr)\
\end{equation}
Since $f\in \cZ$, $|Z|^{-1}\leq 1/z_{\min}$ and $\xi\in B_{r_\eps}$, we can replace $\int_{B_{r_\eps}}$ with $\int_{\RR^2}$ up 
to an error of size $\mathcal{O}_{\cS_*}(\eps^{M+1})$, thus completing the proof for the translation operator by expanding the $n$-th power in the integral.

By the definition of $\cV_{\eps,Z}$ and the expansion in Lemma~\ref{lem:expansionsR},  we have
\begin{align}
 \cV_{\eps,Z} \varphi(\xi)&=  \sum_{l=1}^N\cT_{-\eps,Z}(\varphi)(Q_l\xi)\\
\label{eq:expVeps}            &=\mrM(\Omega)\frac{N}{2\pi}\log\Big(\frac{|Z|}{\eps d}\Big)
  +\sum_{n=1}^M\sum_{k=0}^{n} \sum_{l=1}^N(-1)^n\eps^nc_{n,k} \Re\Big(\frac{(Q_l\xi)^{n-k}}{Z^n}{\rm m}^k(f)\Big) + \cO_{\cS_*}\bigl(\epsilon^{M+1}\bigr)\,.
    \end{align}
    Then, we rewrite a part of the sum as
    \begin{align}
\sum_{l=1}^N \Re\Big(\frac{(Q_l\xi)^{n-k}}{Z^n}{\rm m}^k(f)\Big)=\Re\Bigg(\frac{\xi^{n-k}}{Z^n}\mrm^k(f)\sum_{l=1}^NQ_l^{n-k}\Bigg).
    \end{align}
    By the properties of the roots of unity, we know 
\begin{align}
\sum_{l=1}^{N}Q_l^{n-k}=\sum_{l=1}^{N}\e^{\frac{2\pi il}{N}(n-k)}=\begin{cases}
        0  & \text{if } n-k\neq 0 \, \mod N\\
        N & \text{if } n-k= 0 \, \mod N.
    \end{cases}.
\end{align}
Therefore, we can first simplify the sum in \cref{eq:expVeps} by isolating the case $k=n$, which gives the second term on the right-hand side of \cref{Vexpansion}. Then, we need $0\leq n-k=\ell N$ with $\ell \in \mathbb{N}\setminus\{0\}$ since $0\leq k\leq n-1$, meaning that we have a nontrivial contribution only for $n\geq N$. Thus, summing in $\ell$ instead of $k$, we finally get
\begin{align}
\sum_{n=1}^M\sum_{k=0}^{n-1}\sum_{l=1}^N(-1)^n\eps^nc_{n,k}\Re\Big(\frac{(Q_l\xi)^{n-k}}{Z^n}\mrm^k(f)\Big)=N\sum_{n=N}^M\sum_{\substack{\ell \in \mathbb{N}\setminus \{0\}\\
\ell N\leq n}}(-1)^n\eps^n c_{n,n-\ell N} \Re\Big(\frac{\xi^{\ell N}}{Z^n}\mrm^{n-\ell N}(f)\Big),
\end{align}
and the proof for the polygon-center operator is over.

We finally consider the operator $\cQ_{\eps,Z}$, which is more convenient to rewrite as
\begin{equation}
\label{eq:QepsZl}
    \cQ_{\eps,Z}(f)=\sum_{l=1}^{N-1} f\big(Q_l \xi+\frac{1}{\eps(t) d}Z_l\big), \qquad Z_l(t)=(Q_l-I)Z.
\end{equation}

 Using the expansion of $\cT_{\eps,Z}$, we get that
    \begin{align}
            \cQ_{\eps,Z} \varphi(\xi)=\frac{\mrM(\Omega)}{2\pi}\sum_{l=1}^{N-1}\log\Big(\frac{|Z_l|}{\eps d}\Big)
  +\sum_{n=1}^M\sum_{k=0}^{n} \sum_{l=1}^{N-1}\eps^nc_{n,k} \Re\Big(\frac{(Q_l\xi)^{n-k}}{Z_l^n}{\rm m}^k(f)\Big)+
\cO_{\cS_*}\bigl(\epsilon^{M+1}\bigr).
    \end{align}
     Recalling that $Z_l=(Q_l-1)Z$, in view of the properties of ${\rm S}_{n,k}$ in Lemma \ref{lem:Snk}, we readily find that
\begin{align}
    \sum_{l=1}^{N-1} \Re\Big(\frac{(Q_l\xi)^{n-k}}{Z_l^n}{\rm m}^k(f)\Big)=\Re\Bigg(\frac{\xi^{n-k}}{Z^n}\mrm^k(f)\sum_{l=1}^{N-1}\frac{Q_l^{n-k}}{(Q_l-1)^n}\Bigg)={\rm S}_{n,n-k}\Re\Bigg(\frac{\xi^{n-k}}{Z^n}\mrm^k(f)\Bigg),
\end{align}
which concludes the proof.

\end{proof}


\subsection{Expansion of the integral quantities}

We now compute expansions of integral quantities involving the operators $\cT_{\eps,Z}$ and $\cQ_{\eps,Z}$. Indeed, in view of the definitions of the radius and angular speed, these computations are necessary to rigorously establish their asymptotic expansions as stated in \cref{eq:expRa}. The following lemma also establishes the expansion~\cref{eq:expansionBeps}.
\begin{lemma}
\label{lem:expVel}
    Let $f,h\in \cZ$, $\varphi=\Delta^{-1}f$ and assume that $0<z_{\min}\leq |Z|$. Then, for all $M\in \mathbb{N}$ one has the expansions
    \begin{align}
\label{eq:expspeed}&\int_{\RR^2}\nabla^\perp(\cT_{\eps,Z}\varphi)h\,\dd \xi=\sum_{n=1}^M\eps^n\frac{Z^\perp}{|Z|^{2}}\frac{1}{\overline{Z^{n-1}}}\sum_{k=0}^{n-1}(n-k)c_{n,k}\overline{\mrm^k(f)}\overline{\mrm^{n-k-1}(h)}+\cO(\eps^{M+1})    \end{align}
and
\begin{equation}
    \int_{\RR^2}\nabla^\perp(\cQ_{\eps,Z}\varphi)h\,\dd \xi=\sum_{n=1}^M\eps^n\frac{Z^\perp}{|Z|^{2}}\frac{1}{\overline{Z^{n-1}}}\sum_{k=0}^{n-1}(n-k)c_{n,k}{\rm S}_{n,n-k}\overline{\mrm^k(f)}\overline{\mrm^{n-k-1}(h)}+\cO(\eps^{M+1}) .
\end{equation}
\end{lemma}
\begin{proof}
  Appealing to \cref{lem:Expansions}, we know that 
    \begin{align}
\nabla^\perp\cT_{\eps,Z}\varphi=\sum_{n=1}^M\sum_{k=0}^n\eps^nc_{n,k}\nabla^\perp\Re\bigg(\frac{\xi^{n-k}}{Z^n}\mrm^k(f)\bigg)+\cO_{\cS_*}(\eps^{M+1}).
    \end{align}
    To compute the orthogonal gradient, we pass to polar coordinates $\xi=r\e^{i\theta}$ (keeping the complex number notation) and we note that
    \begin{equation}
        \nabla^\perp =i(\de_1+i\de_2)\to i\e^{i\theta}(\de_r+i\frac{1}{r}\de_\theta).
    \end{equation}
    Then, denoting 
    \begin{align}
        Z=|Z|\e^{i\phi_Z}, \qquad \mrm^k(f)=|\mrm^k(f)|\e^{i\phi_k},
    \end{align}
we get    \begin{align}
    \Re\bigg(\frac{\xi^{n-k}}{Z^n}\mrm^k(f)\bigg)\to \frac{|\mrm^k(f)|}{2|Z|^n}r^{n-k}\big(\e^{i(n-k)\theta +i(\phi_k-n\phi_Z)}+\e^{-i(n-k)\theta -i(\phi_k-n\phi_Z)}\big).
    \end{align}
    Thus, applying the $\nabla^\perp$ we see that the right-hand side becomes 
    \begin{align}
    &i(n-k)\frac{|\mrm^k(f)|}{2|Z|^n}r^{n-k-1}\big(\e^{i(n-k+1)\theta +i(\phi_k-n\phi_Z)}+\e^{-i(n-k-1)\theta -i(\phi_k-n\phi_Z)}\big)\\
    &-i(n-k)\frac{|\mrm^k(f)|}{2|Z|^n}r^{n-k-1}\big(\e^{i(n-k+1)\theta +i(\phi_k-n\phi_Z)}-\e^{-i(n-k-1)\theta -i(\phi_k-n\phi_Z)}\big)\\
    &=i(n-k)\frac{|\mrm^k(f)|}{|Z|^n}r^{n-k-1}\e^{-i(n-k-1)\theta -i(\phi_k-n\phi_Z)}.
    \end{align}
    Going back to Cartesian coordinates, note that
    \begin{align}
        r^{n-k-1}\e^{-i(n-k-1)\theta}\to \overline{\xi^{n-k-1}}, \qquad |\mrm^k(f)|\e^{-i\phi_k}\to \overline{\mrm^{k}(f)}
    \end{align}
    and
    \begin{align}
        i\frac{1}{|Z|^n}\e^{in\phi_Z}=i\frac{|Z|}{|Z|^{n+1}}\e^{in\phi_Z}\to \frac{Z^\perp}{|Z|^{2}}\frac{1}{\overline{Z^{n-1}}}.
    \end{align}
    Therefore, we have computed that 
    \begin{align}
\nabla^\perp\Re\bigg(\frac{\xi^{n-k}}{Z^n}\mrm^k(f)\bigg)=(n-k)\frac{Z^\perp}{|Z|^{2}}\frac{1}{\overline{Z^{n-1}}}\overline{\mrm^{k}(f)}\overline{\xi^{n-k-1}}.
    \end{align}
    Since $h\in \cZ$, we then easily deduce that 
    \begin{align}
        \int_{\RR^2}(\nabla^\perp\cT_{\eps,Z}\varphi)h\, \dd\xi=\sum_{n=1}^M\eps^n\frac{Z^\perp}{|Z|^{2}}\frac{1}{\overline{Z^{n-1}}}\sum_{k=0}^{n-1}(n-k)c_{n,k}\overline{\mrm^k(f)}\overline{\mrm^{n-k-1}(h)}+\cO(\eps^{M+1}),
    \end{align}
    which proves \cref{eq:expspeed}.
    The proof for $\cQ_{\eps,Z}$ is now straightforward using the same arguments.
\end{proof}

\subsection{Commutators and adjoints}\label{app:proofs_adjoints}
We now prove Lemma~\ref{lem:adjoints}. Most of the properties hold for any vector $Z$, but we rely on the specific choice $Z=(R,0)$ to prove the preservation of parity in $\xi_2$ for $\cQ_{\eps,Z}$. Therefore, we present the proof for a general $Z$, explicitly restricting to $Z=(R,0)$ only when necessary.
\begin{proof}[Proof of Lemma~\ref{lem:adjoints}]
    The identity for $\cT_{\eps,Z}^*$ is a direct computation using the definition of $\cT_{\eps,Z}$. Regarding the commutation property with $\Delta$, we simply exploit that a translation commutes with spatial derivatives.

By the definition of $\cV_{\eps,Z}$, since $Q_l$ are rotation matrices we have 
\begin{align}
    \langle \cV_{\eps,Z}g,f\rangle&=\int_{\RR^2}\sum_{l=1}^N g\big(Q_{l}\xi-\frac{1}{\eps d}Z\big) f(\xi)\dd \xi\\ &=\int_{\RR^2}g\big(\eta) \Big(\sum_{l=1}^N f\big(Q_{-l}(\eta+\frac{1}{\eps d}Z)\big)\Big)\dd \xi=\langle g,\cV_{\eps,Z}^* f\rangle.
\end{align}
The commutation properties with $\Delta$ follow by the fact that $\cV_{\eps,Z},\cV_{\eps,Z}^*$ are compositions with a rotation operator $({\sf R}f)(\xi)=f(Q_l\xi)$ and translations ${\sf T}(f)=f(\xi+z)$ for a fixed vector $z\in \RR^2$. Indeed, since clearly $[{\sf R},\Delta]=0$ and $[{\sf T},\Delta]=0$ then also $[\sf{R}\sf{T},\Delta]=[\sf{T}\sf{R},\Delta]=0$ as desired. 
The property about $N$-fold symmetric functions is a direct consequence of the first identity in \cref{def:Veps*}.

For $\cQ_{\eps,Z}^*$, it is enough to observe that 
\begin{equation}
 Q_l\xi+\frac{1}{\eps d}Z_l=\eta\quad \Longrightarrow\quad    \xi=Q_{-\ell}\eta+\frac{1}{\eps d}Z_{-l},
\end{equation}
which readily imply that $\cQ_{\eps,Z}$ is self-adjoint in $L^2(\RR^2)$ by the identification of $Q_{-l}=Q_{N-l}$. The commutation property with the Laplacian follows since $\cQ_{\eps,Z}$ is a combination of translations and rotations. Finally, the property for even functions can be seen as follows: identify $(\xi_1,-\xi_2)$ with the complex conjugation $\overline{\xi}$, $Z$ with the real number $R$ and $Q_l=\exp(2\pi i l/N)$. Then, observe that 
\begin{align}
    Q_{N-l} \overline{\xi}+(\eps d)^{-1}(Q_{N-l}-1)R= Q_{-l} \overline{\xi}+(\eps d)^{-1}(Q_{-l}-1)R=\overline{Q_{l} \xi+(\eps d)^{-1}(Q_{l}-1)R}.
\end{align}
Hence, if $f$ is even we can combine the identity above with the fact that we are summing over $l=1,\dots, N-1$ to conclude the proof. 

It remains to show \cref{key:Q,key:V*,key:V}. First of all, we denote $Z_\eps=Z/(\eps d)$ and we introduce 
\begin{equation}
    \eta_l(\xi)=Q_l(\xi+Z_\eps)-Z_\eps.
\end{equation}
In this way, we can rewrite 
\begin{align}
    \cQ_{\eps, Z}f(\xi)=\sum_{l=1}^{N-1}f(\eta_l(\xi))
\end{align}
Then we observe that 
\begin{align}
    \frac12\{|\cT_{\eps,Z}\xi|^2,f\}=(\xi+Z_\eps)^\perp \cdot \nabla f.
\end{align}
Thus 
\begin{align}
    \frac12(\cQ_{\eps,Z}\{|\cT_{\eps,Z}\xi|^2,f\})(\xi)&=\sum_{l=1}^{N-1}(\eta_l(\xi)+Z_\eps)^\perp \cdot (\nabla f)(\eta_l(\xi))\\
    \label{eq:keyQV0}&=\sum_{l=1}^{N-1}(Q_l(\xi+Z_\eps))^\perp \cdot (\nabla f)(\eta_l(\xi)),
\end{align}
    where in the last identity we used the definition of $\eta_l(\xi).$ On the other hand 
    \begin{align}
            \frac12\{|\cT_{\eps,Z}\xi|^2,\cQ_{\eps,Z}f\}&=(\xi+Z_\eps)^\perp \cdot \nabla (\cQ_{\eps,Z}f)=\sum_{l=1}^N(\xi+Z_\eps)^\perp \cdot \nabla (f(\eta_l(\xi)))\\
    \label{eq:keyQV1}&=\sum_{l=1}^N(\xi+Z_\eps)^\perp \cdot Q_l^T (\nabla f)(\eta_l(\xi)).
    \end{align}
    Then, by the property of the dot product $v\cdot (M^T w)=Mv\cdot  w$ and $Q_lv^\perp=(Q_l v)^\perp$ since $Q_l$ is a rotation matrix, we see that indeed \cref{key:Q} holds true by combining \cref{eq:keyQV0} with \cref{eq:keyQV1}.

    To prove  \cref{key:V*}, we denote $y_l(\xi)=Q_l(\xi+Z_\eps)$ and, in view of \cref{def:Veps}, we have that 
    \begin{align}
        \frac12\cV_{\eps,R}^*(\{|\xi|^2,f\})=\sum_{l=1}^N y_l(\xi)^\perp \cdot (\nabla f)(y_l(\xi)).
    \end{align}
    Instead
    \begin{align}
          \frac12\{|\cT_{\eps,Z}\xi|^2,  \cV_{\eps,R}^*f\}&=\sum_{l=1}^N(\xi+Z_\eps)^\perp \cdot \nabla (f(y_l(\xi)))=\sum_{l=1}^N(\xi+Z_\eps)^\perp \cdot Q_l^T (\nabla f)(y_l(\xi))\\
      &\sum_{l=1}^N(Q_l(\xi+Z_\eps))^\perp \cdot  (\nabla f)(y_l(\xi))=\sum_{l=1}^N y_l(\xi)^\perp \cdot (\nabla f)(y_l(\xi)),
    \end{align}
    meaning that \cref{key:V*} is proved.
    Finally, \cref{key:V} directly follows by taking the $L^2$ adjoint of the identity in \cref{key:V*} (or it can be proved directly as before).
\end{proof}

\section{Viscous deviations from the point-vortex dynamics}\label{app:cA}
Let the family of approximate solutions be constructed as in Section~\ref{sec:App}. We establish the following result, which provides the necessary tools to compute the first-order perturbations to $\Mapp{R}{M}$ and $\Mapp{\alpha}{M}$ relative to the point-vortex dynamics.
\begin{lemma}\label{lem:expansionBtot}
    For every $M \ge 5$ and $\eps$ small enough, there holds that
    \begin{align}
         \Im(&\cB_{\eps,\Mapp{R}{M}}(\Mapp{\Psi}{M},\Mapp{\widehat{\Psi}}{M})) = \frac{\eps}{\Mapp{R}{M}} c_{1,0} (S_{1,1}+\gamma) \\& + \frac{\eps^5}{(\Mapp{R}{M})^3} \Big(3c_{3,0}({\rm S}_{3,3}+\gamma) \mrm_1^2(\Omega_2) + c_{3,2} ({\rm S}_{3,1} \mrm_1^2(\Omega_2)+\gamma \mrm_1^2(\widehat{\Omega}_2)\big) \Big) + \cO(\eps^6 + \delta \eps^5)
    \end{align}
    and
    \begin{multline}
        \Re(\cB_{\eps,\Mapp{R}{M}}(\Mapp{\Psi}{M},\Mapp{\widehat{\Psi}}{M})) \\ = -\delta\frac{\eps^5}{(\Mapp{R}{M})^3} \Big( 3c_{3,0} \big( {\rm S}_{3,3}  + \gamma\big) \mrm_2^{2}(\Omega_{2,\mathrm{NS}}) + c_{3,2} \big( {\rm S}_{3,1} \mrm_2^2( \Omega_{2,\mathrm{NS}}) + \gamma \mrm_2^2( \widehat{\Omega}_{2,\mathrm{NS}})\big) \Big) + \cO(\delta \eps^6).
    \end{multline}
\end{lemma}

\begin{proof}
    We first split the Euler and Navier-Stokes part, with the cross-interactions being denoted with the three-variable definition of $\cB_{\eps,\Mapp{R}{M}}$:
\begin{multline}
        \cB_{\eps,\Mapp{R}{M}}(\Mapp{\Psi}{M},\Mapp{\widehat{\Psi}}{M}) = \cB_{\eps,\Mapp{R}{M}}(\Mappp{\Psi}{M}{E},\Mappp{\widehat{\Psi}}{M}{E})  + \delta \cB_{\eps,\Mapp{R}{M}}( \Mappp{\Psi}{M}{E},\Mappp{\widehat{\Psi}}{M}{E} ,\Mappp{\Psi}{M}{NS}) \\ + \delta \cB_{\eps,\Mapp{R}{M}}( \Mappp{\Psi}{M}{NS},\Mappp{\widehat{\Psi}}{M}{NS} ,\Mappp{\Psi}{M}{E})  + \delta^2 \cB_{\eps,\Mapp{R}{M}}(\Mappp{\Psi}{M}{NS},\Mappp{\widehat{\Psi}}{M}{NS}).
    \end{multline}
We recall the expansion~\cref{eq:expansionBeps} to be:
\begin{equation}
    \label{eq:expansionBepsbis}
    \cB_{\eps,\Mapp{R}{M}}(\varphi,\widehat{\varphi},\widetilde{\varphi}) = i\sum_{n=1}^M\left(\frac{\eps}{\Mapp{R}{M}}\right)^n\sum_{k=0}^{n-1}(n-k)c_{n,k}\left({\rm S}_{n,n-k}\overline{\mrm^k(w)} + \gamma\overline{\mrm^k(\widehat{w})}\right)\overline{\mrm^{n-k-1}(\widetilde{w})}+\cO(\eps^{M+1}),
\end{equation}
Now, we recall from the construction of the first orders that
    \begin{align}
        & \Mappp{\Omega}{M}{NS} = \cO(\eps^2), \quad \Mappp{\widehat{\Omega}}{M}{NS}= \cO(\eps^2), \\
        &  \mrM(\Mappp{\Omega}{M}{NS}) = \mrM(\Mappp{\widehat{\Omega}}{M}{NS}) = 0,\\ 
        &  \mrm^1(\Mappp{\Omega}{M}{NS}) = \mrm^1(\Mappp{\Omega}{M}{E}) = \mrm^1(\Mappp{\widehat{\Omega}}{M}{NS}) = \mrm^1(\Mappp{\widehat{\Omega}}{M}{E}) = 0.
    \end{align}
We deduce from plugging these identities into the expansion~\cref{eq:expansionBepsbis} that
\begin{equation}\label{eq:delta2Term}
    \delta^2 \cB_\eps(\Mappp{\Psi}{M}{NS},\Mappp{\widehat{\Psi}}{M}{NS},\Mappp{\Psi}{M}{NS}) = \cO(\delta^2 (\eps^9+\eps^{M+1})),
\end{equation}
since to produce no vanishing terms, it must hold that $n-k-1 \ge 2$ and $k \ge 2$. This  can happen only if $n \ge 5$, while we also have two factors $\eps^2$ from $\Mappp{\Omega}{M}{NS}$ itself. Following the same reasoning, since $M \ge 5$, we have that
\begin{equation}\label{eq:deltaTerms}
        \delta \cB_{\eps,\Mapp{R}{M}}( \Mappp{\Psi}{M}{E},\Mappp{\widehat{\Psi}}{M}{E} ,\Mappp{\Psi}{M}{NS}) = \cO(\delta \eps^5), \quad  \delta \cB_{\eps,\Mapp{R}{M}}( \Mappp{\Psi}{M}{NS},\Mappp{\widehat{\Psi}}{M}{NS} ,\Mappp{\Psi}{M}{E}) = \cO(\delta \eps^5)
    \end{equation}
    with 3 powers of $\eps$ coming from $\eps^n$ and 2 from $\Mappp{\Omega}{M}{NS}$ or $\Mappp{\widehat{\Omega}}{M}{NS}$.
    
    We now compute the imaginary part of the purely Eulerian part $\cB_{\eps,\Mapp{R}{M}}(\Mappp{\Psi}{M}{E},\Mappp{\widehat{\Psi}}{M}{E})$. To obtain an expansion of this expression, we plug into~\eqref{eq:expansionBepsbis} the fact that $\Mappp{\Omega}{M}{E}$ and $\Mappp{\widehat{\Omega}}{M}{E}$ are even in $\xi_2$ by construction (constraint~\ref{constraint:2}), meaning that for every $k \ge 1$,
    \begin{equation}\label{eq:eulerian_part_even}
        \mrm^k_2(\Mappp{\Omega}{M}{E}) = \mrm^k_2(\Mappp{\widehat{\Omega}}{M}{E}) = 0.
    \end{equation}
    This gives the following expansion:
    \begin{multline}
        \Im(\cB_{\eps,\Mapp{R}{M}}(\Mappp{\Psi}{M}{E},\Mappp{\widehat{\Psi}}{M}{E})) \\ =\sum_{n=1}^M \left(\frac{\eps}{\Mapp{R}{M}}\right)^n \sum_{k=0}^{n-1} (n-k)c_{n,k} \Big( S_{n,n-k} \mrm_1^k( \Mappp{\Omega}{M}{E}) + \gamma \mrm_1^k( \Mappp{\widehat{\Omega}}{M}{E})\Big) \mrm_1^{n-k-1}(\Mappp{\Omega}{M}{E}) + \cO(\eps^{M+1}).
    \end{multline}
     We now compute the first terms in that expression. The $n=1$ term is non vanishing, equal to $c_{1,0} (S_{1,1}+\gamma)$ since
    \begin{equation}
        \mrM(\Mappp{\Omega}{M}{E}) = \mrM(\Mappp{\widehat{\Omega}}{M}{E}) = 1,
    \end{equation}
    and is the term giving the point-vortex dynamics angular velocity as the first-order value. The $n=2$ terms vanish due to the facts that the first moments vanish, as in that case, either $k = 1$ or $n-k-1 = 1$. We now turn to the case $n=3$, computing for each value of $k$ the terms in the sums:
    \begin{align}
        & k = 0 \quad \rightarrow \quad 3c_{3,0} \Big( {\rm S}_{3,3}  + \gamma\Big) \mrm_1^{2}(\Mappp{\Omega}{M}{E}), \\
        & k = 1 \quad \rightarrow \quad 0 \\
        & k = 2 \quad \rightarrow \quad c_{3,2} \Big( {\rm S}_{3,1} \mrm_1^2( \Mappp{\Omega}{M}{E}) + \gamma \mrm_1^2( \Mappp{\widehat{\Omega}}{M}{E})\Big).
    \end{align}
    Since $\mrm_1^2( \Mappp{\Omega}{M}{E}) = \cO(\eps^2)$ and $\mrm_1^2( \Mappp{\widehat{\Omega}}{M}{E}) = \cO(\eps^2)$, these terms give rise to terms of order $\cO(\eps^5)$ in $\Im(\cB_{\eps,\Mapp{R}{M}}(\Mappp{\Psi}{M}{E},\Mappp{\widehat{\Psi}}{M}{E}))$. Let us now observe that there are no other lower-order terms. For $n=4$, since we already have  $4$ powers due to $\eps^n$, to obtain a term of order 5 or less it would require a contribution from the moments of order at most 1. However, $\Mappp{\Omega}{M}{E} = \Omega_0 + \cO(\eps^2)$ and $\Omega_0$ has only vanishing moments except $\mrM(\Omega_0)$. Since $k=0$ and $n-k-1 = 0$ can not be realized at the same time, no such term can exist. The same applies to the terms $n=5$, and higher terms are at least of order $\eps^n$. In conclusion, we have shown, using in addition that $\mrm_1^{2}(\Mappp{\Omega}{M}{E}) = \eps^2\mrm_1^2(\Omega_2) + \cO(\eps^3+\delta \eps^2)$ and $\mrm_1^{2}(\Mappp{\widehat{\Omega}}{M}{E}) = \eps^2\mrm_1^2(\widehat{\Omega}_2) + \cO(\eps^3+\delta \eps^2)$, that
    \begin{multline}
        \Im(\cB_{\eps,\Mapp{R}{M}}(\Mappp{\Psi}{M}{E},\Mappp{\widehat{\Psi}}{M}{E})) =  \frac{\eps}{\Mapp{R}{M}} c_{1,0} (S_{1,1}+\gamma) \\ + \frac{\eps^5}{(\Mapp{R}{M})^3} \Big(3c_{3,0}({\rm S}_{3,3}+\gamma) \mrm_1^2(\Omega_2) + c_{3,2} ({\rm S}_{3,1} \mrm_1^2(\Omega_2)+\gamma \mrm_1^2(\widehat{\Omega}_2)\big) \Big) + \cO(\eps^6 + \delta \eps^5).
    \end{multline}
    Gathering this expansion with the estimates~\cref{eq:deltaTerms,eq:delta2Term}, gives that
    \begin{multline}
         \Im(\cB_{\eps,\Mapp{R}{M}}(\Mapp{\Psi}{M},\Mapp{\widehat{\Psi}}{M})) = \frac{\eps}{\Mapp{R}{M}} c_{1,0} (S_{1,1}+\gamma) \\ + \frac{\eps^5}{(\Mapp{R}{M})^3} \Big(3c_{3,0}({\rm S}_{3,3}+\gamma) \mrm_1^2(\Omega_2) + c_{3,2} ({\rm S}_{3,1} \mrm_1^2(\Omega_2)+\gamma \mrm_1^2(\widehat{\Omega}_2)\big) \Big) + \cO(\eps^6 + \delta \eps^5).
    \end{multline}
    Notice that applying the expansions for $M \ge 5$ is actually necessary to conclude with that precision, although all terms of order $\le 5$ are induced only by $\Omega_2$ and $\widehat{\Omega}_2$. 

    We now turn to the real part of $\cB_{\eps,\Mapp{R}{M}}(\Mapp{\Psi}{M},\Mapp{\widehat{\Psi}}{M})$. By Lemma~\ref{lem:eveneven}, the purely Eulerian term has no real part. To compute the first terms of the expansion, one thus has to look to a more precise expansion of~\cref{eq:deltaTerms}. Using again relation~\cref{eq:eulerian_part_even} into~\cref{eq:expansionBepsbis}, we compute that
    \begin{multline}
        \Re\Big(\cB_{\eps,\Mapp{R}{M}}( \Mappp{\Psi}{M}{E},\Mappp{\widehat{\Psi}}{M}{E} ,\Mappp{\Psi}{M}{NS}) + \cB_{\eps,\Mapp{R}{M}}( \Mappp{\Psi}{M}{NS},\Mappp{\widehat{\Psi}}{M}{NS} ,\Mappp{\Psi}{M}{E}) \Big) \\ = - \sum_{i=1}^M \left(\frac{\eps}{\Mapp{R}{M}}\right)^n \sum_{k=0}^{n-1} (n-k)c_{n,k}
         \bigg[\big(S_{n,n-k} \mrm_1^k(\Mappp{\Omega}{M}{E}) + \gamma \mrm_1^k(\Mappp{\widehat{\Omega}}{M}{E})\big) \mrm_2^{n-k-1}(\Mappp{\Omega}{M}{NS})  \\ 
        + \big(S_{n,n-k} \mrm_2^k(\Mappp{\Omega}{M}{NS}) + \gamma \mrm_2^k(\Mappp{\widehat{\Omega}}{M}{NS})\big) \mrm_1^{n-k-1}(\Mappp{\Omega}{M}{E})\bigg] + \cO(\eps^{M+1}).
    \end{multline}
    To produce non vanishing terms, one needs in the first part of the sum that $n-k-1 \ge 2$, namely $n \ge 3$, and in the second part that $k \ge 2$, which again, requires $n \ge 3$. We thus compute the terms $n=3$, which are of order $\eps^5$:
        \begin{align}
        & k = 0 \quad \rightarrow \quad 3c_{3,0} \Big( {\rm S}_{3,3}  + \gamma\Big) \mrm_2^{2}(\Mappp{\Omega}{M}{NS}), \\
        & k = 1 \quad \rightarrow \quad 0 \\
        & k = 2 \quad \rightarrow \quad c_{3,2} \Big( {\rm S}_{3,1} \mrm_2^2( \Mappp{\Omega}{M}{NS}) + \gamma \mrm_2^2( \Mappp{\widehat{\Omega}}{M}{NS})\Big).
    \end{align}
    Terms with $n \ge 4$ are of order at least $\eps^6$ due to the presence of Navier-Stokes pieces. In conclusion,
    \begin{multline}
        \Re(\cB_{\eps,\Mapp{R}{M}}(\Mapp{\Psi}{M},\Mapp{\widehat{\Psi}}{M})) \\ = \delta\frac{\eps^5}{(\Mapp{R}{M})^3} \Big( 3c_{3,0} \big( {\rm S}_{3,3}  + \gamma\big) \mrm_2^{2}(\Omega_{2,\mathrm{NS}}) + c_{3,2} \big( {\rm S}_{3,1} \mrm_2^2( \Omega_{2,\mathrm{NS}}) + \gamma \mrm_2^2( \widehat{\Omega}_{2,\mathrm{NS}})\big) \Big) + \cO(\delta \eps^6).
    \end{multline}
    \end{proof}


\subsection*{Acknowledgements}
\text{~}
We would like to thank Thierry Gallay for helpful discussions.
Part of this work was conducted when the second author was supported by the grant BOURGEONS ANR-23-CE40-0014-01 of the French National Research Agency.
The first author was supported by the Swiss National Science Foundation (SNF Ambizione grant PZ00P2\_223294).

\bibliographystyle{siam}
\bibliography{bibclass}

\end{document}